\documentclass[reqno, 12pt]{amsart} 

\usepackage{microtype} 
\usepackage{amsfonts,amsthm,amsmath,amssymb,amscd,mathrsfs} 
\allowdisplaybreaks
\usepackage{latexsym} 
\usepackage[colorlinks=true,linkcolor=blue,citecolor=red,urlcolor=black]{hyperref} 
\usepackage{graphics} 
\usepackage{indentfirst} 
\usepackage{cite} 
\usepackage[dvips]{epsfig} 
\usepackage{color} 
\usepackage{lmodern} 
\usepackage{soul} 
\usepackage{wasysym}
\theoremstyle{plain} 
\newtheorem{Thm}{Theorem}[section] 
\newtheorem{Lem}[Thm]{Lemma}     
\newtheorem{Prop}[Thm]{Proposition}
\newtheorem{Cor}[Thm]{Corollary}

\theoremstyle{definition}

\theoremstyle{remark}
\newtheorem{Rem}[Thm]{Remark}

\numberwithin{equation}{section} 

\newcommand{\beq}{\begin{equation}}
	\newcommand{\eeq}{\end{equation}}
\newcommand{\ben}{\begin{eqnarray}}
	\newcommand{\een}{\end{eqnarray}}
\newcommand{\beno}{\begin{eqnarray*}}
	\newcommand{\eeno}{\end{eqnarray*}}
\newcommand{\no}{\nonumber}
\newcommand{\lt}{\left}
\newcommand{\rt}{\right}

\newcommand{\px}{\partial_x}
\newcommand{\py}{\partial_y}

\newcommand{\pt}{\partial_t}
\newcommand{\pxi}{\partial_\xi}
\newcommand{\abs}[1]{\lvert#1\rvert}  

\begin{document}
	
	\title[Stability threshold of the Couette flow for Boussinesq equation]{Stability threshold of Couette flow for  Boussinesq equations in $\mathbb{R}^2$}	
	\author{Yubo~Chen}
	\address[Yubo~Chen]{School of Mathematical Sciences, Dalian University of Technology, Dalian, 116024,  China}
	\email{1220823215@mail.dlut.edu.cn}
	
	\author{Wendong~Wang}
	\address[Wendong~Wang]{School of Mathematical Sciences, Dalian University of Technology, Dalian, 116024,  China}
	\email{wendong@dlut.edu.cn}
	
	\author{Guoxu~Yang}
	\address[Guoxu~Yang]{School of Mathematical Sciences, Dalian University of Technology, Dalian, 116024,  China}
	\email{guoxu\_dlut@outlook.com}
	\begin{abstract}
		
		This paper establishes the asymptotic stability threshold for the Couette flow $(y,0)$ under the 2D Boussinesq system in $\mathbb{R}^2$. 
		It was proved that for initial perturbations in Sobolev spaces with controlled low horizontal frequencies, 
		the stability threshold is at most $\left\{\frac{1}{3}+, \frac{2}{3}+\right\}$, extending the known threshold results from the periodic case $\mathbb{T}_x \times \mathbb{R}_y$ to the whole space. 
		
		The core innovations are twofold: First, the $\langle D_x^{-1} \rangle$ control on the initial data simultaneously resolves horizontal frequency singularities and optimizes integral indices when applying Young's convolution inequality. 
		Second, we develop a modified multiplier $\mathcal{M}_3$ that effectively absorbs the $|D_x|^{1/3}$ derivative structure induced by the temperature equation while handling nonlinear echo cascades. 
	\end{abstract}

	\maketitle
	\tableofcontents
	\section{Introduction}    
	In this paper, we investigate the nonlinear stability of the Couette flow $(y, 0)$ for the two-dimensional Boussinesq equations in $\mathbb{R}^2$ :
	\begin{align} \label{eq:main0}
		\left\{\begin{array}{l}
			\partial_t v-\nu \Delta v+v \cdot \nabla v+\nabla p =  \vartheta e_2, \\
			\pt \vartheta - \mu \Delta \vartheta +v\cdot \nabla \vartheta = 0    ,\\
			\operatorname{div} v=0, \\
			\left.v\right|_{t=0}=v_{\text{in}}(x, y), \quad \left.\vartheta\right|_{t=0}=\vartheta_{\text{in}}(x, y),
		\end{array}\right.
	\end{align}
	where $v(t, x, y)=\left(v^1, v^2\right)$ denotes the velocity field, $\vartheta$ the temperature and $p$ the pressure. Moreover, $e_2 = (0, 1)^T$, $\nu$ is the viscosity coefficient and $\mu$ is the thermal diffusivity.

The Boussinesq approximation simplifies fluid dynamics by assuming that density variations are negligible except in the buoyancy term (typically due to temperature or concentration gradients). The Boussinesq approximation and Boussinesq equations provide a balanced approach to modeling buoyancy-driven flows by simplifying density variations while retaining essential physics. They are fundamental in studying thermal convection, environmental flows, and heat transfer where gravitational effects dominate; see \cite{ZZ2023, DWZ2021}. Consider the perturbation of the system \eqref{eq:main0} under the Couette flow $(y, 0)$ and the constant temperature, and 
letting $u=v-(y, 0)$ and $\theta = \vartheta - a$, where $a\in \mathbb{R}$, we have
	$$
	\left\{\begin{array}{l}
		\partial_t u-\nu \Delta u+y \partial_x u+\left(u^2,\, 0\right)+u \cdot \nabla u+\nabla p  =   (\theta + a)e_2,\\
		\partial_t \theta -\mu \Delta \theta +y \partial_x \theta  +  u \cdot \nabla \theta  = 0,  \\
		\operatorname{div} u=0, \\
		\left.u\right|_{t=0}=u_{\text {in}}(x, y),\quad \left.\theta \right|_{t=0}= \theta_{\text{in}}(x, y).
	\end{array}\right.
	$$
	Define $\omega=\partial_y u^1-\partial_x u^2$ be the vorticity, and it follows that
	\begin{equation}\begin{aligned} \label{eq:main}
			\left\{\begin{array}{l}
				\partial_t \omega-\nu \Delta \omega+y \partial_x \omega+u \cdot \nabla \omega = \px \theta,  \\
				\partial_t \theta -\mu \Delta \theta +y \partial_x \theta  +  u \cdot \nabla \theta  = 0,  \\
				u=\nabla^{\perp}(-\Delta)^{-1} \omega=\left(-\partial_y, \partial_x\right)(-\Delta)^{-1} \omega, \\
				\left.\omega\right|_{t=0}=\omega_{\mathrm{in}}(x, y) ,\quad \left.\theta \right|_{t=0}= \theta_{\mathrm{in}}(x, y).
			\end{array}\right.	
	\end{aligned}\end{equation}
    When $\vartheta=0$, \eqref{eq:main0} is reduced to the famous Navier-Stokes equations.
    The transition threshold problem was formulated by Bedrossian, Germain and Masmoudi \cite{BGM2019} as follows:
	Given a norm $\|\cdot\|_X$, determine a $\beta=\beta(X)$ so that
	$$
	\begin{aligned}
		& \left\|u_0\right\|_X \ll R e^{-\gamma} \Rightarrow \text { stability, } \\
		& \left\|u_0\right\|_X \gg R e^{-\gamma} \Rightarrow \text { instability. }
	\end{aligned}
	$$
	The exponent $\gamma$ is referred to as the transition threshold.
    Similarly, for the Boussinesq system, it is interesting to consider the following stability threshold problem:\\
	Given norms $\|\cdot\|_{Y_1}$ and $\|\cdot\|_{Y_2}$, find $\alpha=\alpha\left(Y_1, Y_2\right)$ and $\beta=\beta\left(Y_1, Y_2\right)$ such that
	$$
	\begin{aligned}
		& \left\|u_{\mathrm{in}}\right\|_{Y_1} \leq \nu^\alpha \text { and }\left\|\theta_{\mathrm{in}}\right\|_{Y_2} \leq \nu^\beta \Rightarrow \text { stability, } \\
		& \left\|u_{\mathrm{in}}\right\|_{Y_1} \gg \nu^\alpha \text { or }\left\|\theta_{\mathrm{in}}\right\|_{Y_2} \gg \nu^\beta \Rightarrow \text { instability. }
	\end{aligned}
	$$

	 For $\mathbb{T} \times \mathbb{R}$, Deng--Wu--Zhang \cite{DWZ2021} proved the asymptotic stability of the steady state if the initial perturbations $(u_{\mathrm{in}}, \theta_{\mathrm{in}})$ satisfy
    \begin{equation} \label{eq:temp0.0}
        \left\|u_{\mathrm{in}}\right\|_{H^2} \leq \epsilon_0 \nu^{\frac{1}{2}}, \quad \left\|\theta_{\mathrm{in}}\right\|_{H^1}+\left\|\left|D_x\right|^{\frac{1}{6}} \theta_{\mathrm{in}}\right\|_{H^1} \leq \epsilon_1 \nu^{\frac{11}{12}}.
    \end{equation}
	The threshold obtained was improved by Zhang--Zi \cite{ZZ2023} to
	$$
	\left\|u_{\text {in}}\right\|_{H^{s+1}} \leq \epsilon_0 \nu^{\frac{1}{3}},\quad \left\|\theta_{\text {in}}\right\|_{H^s}+\nu^{\frac{1}{6}}\left\|\left|D_x\right|^{\frac{1}{3}} \theta_{\text {in}}\right\|_{H^s} \leq \epsilon_1 \nu^{\frac{5}{6}},
	$$
	where $s>7$. Later, the threshold  was improved again by Niu--Zhao \cite{NZ2024} to
	$$
	\left\|u_{\mathrm{in}}\right\|_{H^{s+1}} \leq \epsilon_0 \nu^{\frac{1}{3}}, \quad \left\|\left\langle\partial_x\right\rangle \theta_{\mathrm{in}}\right\|_{H^s} \leq \epsilon_0 \nu^{\frac{2}{3}},
	$$
	where $s>5$.

    For $\mathbb{T} \times \mathbb{I}$ with non-slip boundary data, Masmoudi--Zhai--Zhao \cite{MZZ2023} obtained the same threshold as \eqref{eq:temp0.0}. For $\mathbb{T} \times \mathbb{I}$ with Navier-slip boundary data, the threshold was obtained by Wang--Yang \cite{WY2024} as follows:
	$$
	\left\|u_{\text{in}}\right\|_{H^{4}}  + \nu^{-\frac{7}{12}}\left\|\theta_{\text{in}}\right\|_{H^3} + \nu^{-\frac13}\left\|\partial_x \theta_{\text{in}}\right\|_{H^3} \leq \epsilon_0 \nu^{\frac{1}{3}}.
	$$

    In the whole space $\mathbb{R} \times \mathbb{R}$, Wang--Wang \cite{WW202502} proved the asymptotic stability of the steady state if the initial perturbations $(u_{\mathrm{in}}, \theta_{\mathrm{in}})$ satisfy
	$$
	\left\|\omega_{\mathrm{in}}\right\|_{H^1 \cap L^1} \leq \epsilon_0 \nu^{\frac{3}{4}}, \quad\left\|\theta_{\mathrm{in}}\right\|_{H^1 \cap L^1} \leq \epsilon_1 \nu^{\frac{5}{4}}.
	$$
Recently, for the Navier-Stokes equations on $\mathbb{R} \times \mathbb{R}$, Arbon--Bedrossian \cite{AB2024} proved
       \beno
        \sum_{0\leq j\leq 1}\|\langle\partial_x\rangle ^m\langle\frac{\partial_x}{\nu}\rangle^{-\frac{j}{3}}\partial_y^jw_{in}\|_{L^2_{x,y}}+\|w_{in,k}\|_{L^\infty_k L_y^2}\leq \delta_1 \frac{\nu^{\frac12}}{1+\ln(\frac{1}{\nu})^{1/2}},
        \eeno
        where $m\in (1/2,1)$, $\langle \cdot \rangle:=\sqrt{1 + (\cdot)^2}$ and $\delta_1>0 $ is a constant.
        Then, the result was improved in \cite{LLZ2025} by
        Li--Liu--Zhao 
		$$
		\left\|\left\langle D_x, D_y\right\rangle^6\left\langle\frac{1}{D_x}\right\rangle^4 \omega_{\mathrm{in}}\right\|_{L^2}+\left\|\left\langle D_x, D_y\right\rangle^5\left\langle\frac{1}{D_x}\right\rangle^4 \omega_{\mathrm{in}}\right\|_{L^1} \leq \nu^{\frac{1}{3}+\delta},
		$$
		where $\delta>0$ is small, $D = \frac1i \partial$  and $\langle D_x, D_y\rangle:=\sqrt{1 + D_x^2 + D_y^2}$.

For the 3D Boussinesq equations via the Couette flow, we refer to \cite{ZZ2025,ZZW2024,CWW202501} and the reference therein. 

    Motivated by the progress in \cite{LLZ2025}, it is interesting whether the  threshold holds for the Boussinesq system (\ref{eq:main}).

    Our main result is follows: 
    \begin{Thm} \label{thm:main}
		Suppose that $\nu=\mu$. For arbitrarily small $0<\delta<1$, $\frac{1-\delta}{2}<\epsilon<\frac{1}{2}$, there exists $0<c, \nu_0<1$ such that for all $0<\nu \leq \nu_0$, if the initial data $\omega_{\mathrm{in}}$ and $\theta_{\mathrm{in}}$ satisfy
		\begin{equation}\begin{aligned} \label{eq:the initial data}
				&\|\langle D_x, D_y\rangle^6\langle\frac{1}{D_x}\rangle^4 \omega_{\mathrm{in}}\|_{L^2\cap L^1} \leq \nu^{\frac{1}{3}+\delta},  \\
				&\|\langle D_x, D_y\rangle^6\langle\frac{1}{D_x}\rangle^4 \langle D_x \rangle \theta_{\mathrm{in}}\|_{L^2\cap L^1} \leq \nu^{\frac{2}{3} + 2\delta},
		\end{aligned}\end{equation}
		then \eqref{eq:main} has a pair of unique global solution $(\omega(t),\,\theta(t))$ with the following stability estimate:
		\begin{equation*}\begin{aligned}
				\| e^{c \nu^{\frac{1}{3}}  \lambda(D_x) t} \langle D_x \rangle &\langle \tfrac{1}{D_x}  \rangle^\epsilon \omega\|_{L^\infty_{[0, +\infty)} L^2}  \\
				&\quad + \nu^{-\frac{1}{3} - \delta} \| e^{c \nu^{\frac{1}{3}} \lambda(D_x) t} \langle D_x \rangle \langle \tfrac{1}{D_x} \rangle^\epsilon \langle D_x \rangle^{\frac{1}{3}} \theta \|_{L^\infty_{[0, +\infty)} L^2} 
				\leq  C \nu^{\frac{1}{3} +\delta},
		\end{aligned}\end{equation*}
		where $\lambda(k)=\min \{1,|k|^{\frac{2}{3}} \}$.
	\end{Thm}

\begin{Rem} 
		When the temperature $\theta$ vanishes, the system reduces to the Navier-Stokes equations, as studied in Arbon--Bedrossian~\cite{AB2024} and Li--Liu--Zhao~\cite{LLZ2025}. The above theorem extends their results to the Boussinesq setting, achieving the same stability threshold. Compared with~\cite{LLZ2025}, the main difficulty arises from the nonlinear term $u\cdot\nabla \theta$ when certain derivatives are propagated onto $u$. In particular, the multiplier $\mathcal{M}_3$ in~\cite{LLZ2025} appears insufficient for estimating 
		\[
		|D_x|^{1/3} u^{\mathrm{NL}} \cdot \nabla \theta^{\mathrm{L}}
		\]
		(see Step~VI in Proposition~\ref{lem:est of thetanl}). To address this, we construct {\bf a new multiplier $\mathcal{M}_3$} in~\eqref{eq:M}, which is employed to handle the term $I_4$ in Step~VI:
		$$
		\begin{gathered}
			I_4 :=\left|\int_{\mathbb{R}^4} \mathcal{M}(t, k, \xi) e^{2 c \nu^{\frac{1}{3}} \lambda(k) t}\langle k\rangle^2\langle\frac{1}{k}\rangle^{2 \epsilon} |k|^\frac13\hat{\theta}_k^{\mathrm{NL}}(\xi) |l|^\frac13  \frac{t l(k-l)}{|l|^2+|\eta|^2} \hat{\omega}_l^{\mathrm{NL}}(\eta) \hat{\theta}_{k-l}^{\mathrm{L}}(\xi-\eta) d k d l d \xi d \eta\right| .
		\end{gathered}
		$$ 
		When $|k-l| < |k|$, this new $\mathcal{M}_3$ successfully absorbs the additional factor $|l|^{1/3}$ due to the fact $\langle 1/l \rangle^{-2/3} {|l|^{-2/3}}  \lesssim 1$.
	\end{Rem}

    \begin{Rem}
		It is worth mentioning that it remains challenging to improve the transition threshold of $\{\omega,\theta\}$ from $\{\frac{1}{3}+\delta,\, \frac{2}{3}+2\delta\}$ to $\{\frac{1}{3},\, \frac{2}{3}\}$ under Sobolev perturbations. The latter seems to be the best known result for the Navier-Stokes equations in $\Omega = \mathbb{T} \times \mathbb{R}$, as established by Masmoudi--Zhao~\cite{MZ2022} and Wei--Zhang~\cite{WZ2023} by different methods. The main difficulty lies in the singular behavior of the frequency variables $(l,\eta)$ as $|l| + |\eta| \to 0$. For instance, when dealing with the term $u^{\mathrm{NL}} \cdot \nabla \theta^{\mathrm{L}}$, it is difficult to use the inviscid damping to control $L^2$-norm of $u^{\mathrm{NL}}$. The additional $\delta$ in the threshold, however, allows this estimate to succeed (see~\eqref{eq:temp0.1} in Lemma~\ref{eq:est of theta NL short time} for details).
	\end{Rem}

	\begin{Rem}
		For more precise estimates, we apply a standard quasi-linearization method to handle the difficulty arising from the buoyancy term $\partial_{x}\theta$. This approach is natural when making energy estimates $L^{p}$. The estimates of the equations for $\{\omega, \theta\}$ are divided into linear and non-linear parts. Estimates for linear parts are determined by the size of the initial data and even the required regularity. For linear parts, there are many splendorous ideas introduced by Kelvin\cite{K1887} for the incompressible Navier-Stokes equation in whole space. However, if it is coupled with the temperature equation, it seems puzzling. Here, {\bf a new observation} is we calculate the Kelvin solution of the coupled system as in \cite{DWZ2021} and make some energy estimates in the sense of the Fourier transform, which will be used in the nonlinear parts.


        Besides, the difficulty of the nonlinear term comes from the interaction between linear and nonlinear terms.
	 Inspired by the tricks in \cite{WZ2023}, the calculations of the nonlinear parts will be separated into short time scale $t\leq \nu^{-\frac16}$ and long time scale $t\geq \nu^{-\frac16}$. In the short time interval, we do not need to propagate the derivative of $\partial_{x}\theta$ since the time integral can be handled. During the long time interval, we need to propagate some derivatives to deal with the buoyancy term. Due to the pointwise inviscid damping estimate of the potential function, i.e. $\Delta^{-1}$, it can provide a power of $t^{-1}$, which  is effective to control  some  non-linear terms (see, e.g., \eqref{eq:temp9} in Proposition \ref{lem:est of thetanl}). 
     \end{Rem}


	\begin{Rem}
	Another difficulty  may lie in the estimates for Fourier frequency. For instance, when dealing with the nonlinear terms in the equation of $\omega$ and $\theta$, we will encounter with the convolution for both horizontal and vertical frequency, it needs to  carefully  do some scaling and overcome the singularity. Inspired by \cite{LLZ2025}, one can follow their ideas when dealing with the transform between $k$, $l$ and $k-l$. Moreover, the setting of $\langle \frac{1}{D_x}\rangle$ and $\langle D_x\rangle$ can prevent the singularity in horizontal frequency in one hand, and share responsibility for integral index when using Young's convolution inequality. 

 \end{Rem}

	\subsection{The progress of Navier-Stokes equations}
	One special case of \eqref{eq:main0} is the  Navier-Stokes equations by taking $\vartheta=0$. The stability of the Couette flow for the Navier-Stokes equation has been widely studied. Let us recall some related results. 
	
	When $\Omega=\mathbb{T} \times \mathbb{R}$,
	\begin{itemize}
		\item if the perturbation for vorticity is in the Gevrey class, then $\gamma=0$ \cite{BMV2016},
		\item if the perturbation for velocity is in Sobolev space $H^2$, then $\gamma \leq \frac{1}{2}$ \cite{BWV2018},
		\item if the perturbation for vorticity is in Sobolev space $H_x^{\text {log }} L_y^2$, then $\gamma \leq \frac{1}{2}$ \cite{MZ202002},
		\item if the perturbation for vorticity is in Gevrey-$\frac{1}{s}$, then $\gamma \in\left[0, \frac{1}{3}\right]$ \cite{LMZ2022},
		\item if the perturbation for vorticity is in Sobolev space $H^\sigma (\sigma \geq 40)$, then $\gamma \leq \frac{1}{3}$ \cite{MZ2022},
		\item if the perturbation for vorticity is in Sobolev space $H^b (b \geq 2)$, then $\gamma \leq \frac{1}{3}$ \cite{WZ2023}.
	\end{itemize}
	
	When $\Omega=\mathbb{T} \times \mathbb{I}$,
	\begin{itemize}
		\item if the velocity perturbation is in the Sobolev space $H^2$ with non-slip boundary data, then $\gamma \leq \frac{1}{2}$ \cite{CLWZ2020},
		\item if the perturbation for vorticity is in Sobolev space $H^3$ with Navier-slip boundary data, then $\gamma \leq \frac{1}{3}$ \cite{WZ2024}.
	\end{itemize}
	
	When $\Omega=\mathbb{R} \times \mathbb{R}$,
	\begin{itemize}
		\item if the perturbation for velocity is in Sobolev space $L^2 \cap L^1$, then $\gamma \leq \frac{3}{4}$ \cite{WW2025},
		\item the perturbation for velocity satisfies \cite{AB2024}
      $\gamma \leq \frac{1}{2}+,$
        which was improved  to  $\gamma \leq \frac{1}{3}+$ in 
        \cite{LLZ2025}.
		.
	\end{itemize} 
	
	When $\Omega=\mathbb{T} \times \mathbb{R} \times \mathbb{T}$,
	\begin{itemize}
		\item if the perturbation for velocity is in Gevery class, then $\gamma=1$ \cite{BGM2020,BGM2022},
		\item if the perturbation for velocity is in Sobolev class $H^\sigma\left(\sigma>\frac{9}{2}\right)$, then $\gamma \leq \frac{3}{2}$ \cite{BGM2017},
		\item if the perturbation for velocity is in Sobolev class $H^2$, then $\gamma \leq 1$ \cite{WZ2021}.
	\end{itemize}
	
	When $\Omega=\mathbb{T} \times \mathbb{I} \times \mathbb{T}$,
	\begin{itemize}
		\item if the velocity perturbation is in the Sobolev class $H^2$, then $\gamma \leq 1$ \cite{CWZ2024}.
	\end{itemize}
	
	As for the  plane Poiseuille flow, the Kolmogorov flow or the rotation term, we refer to \cite{ZEW2020,DL2022,DL2024,CWZ2023,CDLZ2024, HSX2024-1, LWZ2020} and the reference therein.

	\subsection{Two stability mechanisms}
	The study of the stability of Couette flow for the Navier-Stokes equation has been a prominent topic in fluid mechanics since the pioneering works of Kelvin \cite{K1887}, Reyleigh \cite{R1879}, Orr \cite{O1907} and Sommerfeld \cite{S1908}. Consider the vorticity formulation of the linearized Navier–Stokes equations via the Couette flow:
	\begin{equation*}\begin{aligned} 
			\left\{\begin{array}{l}
				\partial_t \omega+y \partial_x \omega-\nu \Delta \omega=0,  \\
				\left. \omega\right|_{t=0}=\omega_{\mathrm {in}}.
			\end{array}\right.	
	\end{aligned}\end{equation*}
	By Fourier transform, the Kelvin's solution is
	$$
	\hat{\omega}(t, k, \xi)=\hat{\omega}_{\mathrm{in}}(k, \xi+k t) e^{-\nu \int_0^t|k|^2+|\xi+k(t-s)|^2 d s},
	$$
	which satisfies the following two linear estimates:
	\begin{equation}\begin{aligned} \label{eq:enhanced dissipation}
			|\hat{\omega}(t, k, \xi)| \leq C\left|\hat{\omega}_{\mathrm{in}}(k, \xi+k t)\right| e^{-c \nu^{\frac{1}{3}}|k|^{\frac{2}{3}} t},
	\end{aligned}\end{equation}
	\begin{equation}\begin{aligned}  \label{eq:inviscid damping}
			|\hat{\phi}(t, k, \xi)| \leq C\langle t\rangle^{-2} \frac{1+|k|^2+|\xi+k t|^2}{|k|^4}\left|\hat{\omega}_{\text {in }}(k, \xi+k t)\right| e^{-c \nu^{\frac{1}{3}}|k|^{\frac{2}{3}} t},
	\end{aligned}\end{equation}
	for some $c>0$, where $\phi=\Delta^{-1} \omega$ is the stream function.
	
	\subsubsection{Enhanced dissipation}
	\eqref{eq:enhanced dissipation} is the enhanced dissipation estimate, since the dissipation time scale is $O\left(\nu^{-\frac{1}{3}}\right)$, which is much faster than the standard heat dissipation time scale $O\left(\nu^{-1}\right)$ for small $\nu$. Clearly the dissipation rate is inhomogeneous and depends on the frequencies $k$. 
	
	The phenomenon of enhanced dissipation has been widely observed and studied in physics literature (see, e.g., \cite{BL1994,LB2001,RY1983,T1887}). It has recently attracted enormous attention from the mathematics community and significant progress has been made. One of the earliest rigorous results on the enhanced dissipation is obtained by Constantin--Kiselev--Ryzhik--Zlatos \cite{CKRZ2008} on the enhancement of diffusive mixing. However, the quantitative enhanced dissipation rate is usually hard to obtain except for some special flows such as shear flow, spiral flow and Anosov flow \cite{ZDE2020}.
	
	In additional to the transition threshold problem, the enhanced dissipation also plays an important role for the suppression of blow-up in the Keller-Segel system \cite{BH2017,KX2016}. Let us refer to \cite{CW2024,CWWW2025} for more recent works.
	
	\subsubsection{Inviscid damping}
	\eqref{eq:inviscid damping} is so called inviscid damping estimate, which is due to the mixing of the vorticity
	induced by the shear flow. This effect is mainly based on the following inequality
	\begin{equation*}\begin{aligned}
			\frac{1}{|l|^2+|\eta|^2} \lesssim \frac{\langle l, \eta+l t\rangle^2}{|l|^4\langle t\rangle^2} \lesssim\langle t\rangle^{-2}\langle\frac{1}{l}\rangle^4\langle l, \eta+l t\rangle^2,
	\end{aligned}\end{equation*}
	which plays an important role in our proof. 
	
	Inviscid damping is analogous to Landau damping in plasma physics found by Landau \cite{L1946}. For general shear flows, establishing linear inviscid damping remains a challenging problem. In a series of works \cite{WZZ2018,WZZ2019,WZZ202001}, Wei-Zhang-Zhao proved linear inviscid damping for both monotone and certain non-monotone flows, including the Poiseuille and Kolmogorov flows. We refer the reader to \cite{O1907,BZV2019,ZZ2019,J2020,WZZ202002,Z2016} and the references therein for related results and recent developments on linear inviscid damping. Nonlinear inviscid damping poses even greater difficulties. Inspired by the groundbreaking work of Mouhot-Villani \cite{MV2011} on nonlinear Landau damping, Bedrossian-Masmoudi \cite{BM2015} established nonlinear inviscid damping for the Couette flow in the domain $ \mathbb{T} \times \mathbb{R}$. We refer to \cite{IJ2020,IJ2022,IJ2023,DM2023,MZ202001} for recent significant progress on nonlinear inviscid damping.

	\subsection{Some notations and outlines}
	
	Here are some notations used in this paper.
	
	\noindent\textbf{Notations}:
	\begin{itemize}
		\item For a given function $f(x,y)$ on on $\mathbb{R}^2$, its $k$-th horizontal Fourier modes can be defined by
		\begin{equation*}
			f_k(y)=\mathcal{F}_{x \rightarrow k}(f)(k, y) = \int_{\mathbb{R}} f(x, y) \mathrm{e}^{- ikx} dx,
		\end{equation*}
		In addition, we denote 
		$$
		\hat{f}_k(\xi)=\hat{f}(k, \xi) = \mathcal{F}_{y \rightarrow \xi}(f_k)(y) = \int_{\mathbb{R}^2} f(x, y) \mathrm{e}^{- i(kx + \xi y)} dxdy.
		$$
		\item Denote $C$ by  a positive constant independent of $\nu$, $\mu$, $t$ and the initial data, and it may be different from line to line. $A \lesssim B$ means there exists a absolutely constant $C$, such that $A \leq C B$.
		\item Denote by $(f | g)$ the $L^2\left(\mathbb{R}^2\right)$ inner product of $f$ and $g$.
		\item The space norm $\|f\|_{L^{p}}$ is defined by	
		$\|f\|_{L^{p}(\mathbb{R}^2)}=\left(\int_{\mathbb{R}^2}|f|^p dxdy\right)^{\frac{1}{p}}  $. The time-space norm $\|f\|_{L_T^{q}L^{p}}$ is defined by	
		$\|f\|_{L_T^qL^p}=\|\|f\|_{L^p(\mathbb{R}^2)}\ \|_{L^q(0,T)}$.
		Moreover, $\|f\|_{L_{(T_1, T_2)}^qL^p}=\|\|f\|_{L^p(\mathbb{R}^2)}\ \|_{L^q(T_1,T_2)}$. For simplicity, we write $\|f\|_{L^p(\mathbb{R}^2)}$ as $\|f\|_{L^p}.$
	\end{itemize}
	
	The paper is organized as follows. Section \ref{Sec.2} presents the main ideas and proof of the Theorem \ref{thm:main}.  Section \ref{sec.3} is dedicated to deriving estimates for the linear part. In Section \ref{sec.4}, we establish the energy estimates in short time scale $t\leq T_0 = \nu^{-\frac16}$. While in Section \ref{sec.5}, we establish the energy estimates in long time scale $t\geq T_0 = \nu^{-\frac16}$.
	
	\section{The proof of the main theorem} \label{Sec.2}
	\subsection{Construction of the multipliers}
	For $k,\xi \in \mathbb{R}$ and small $\nu, \kappa>0$, denote that 
	\begin{equation} \label{eq:M}
		\begin{aligned}
			\mathcal{M}_1(k, \xi) & := \arctan \left(\nu^{\frac{1}{3}}|k|^{-\frac{1}{3}} \operatorname{sgn}(k) \xi\right)+ \frac{\pi}{2},\\
			\mathcal{M}_2(k, \xi) & := \arctan \left(\frac{\xi}{k}\right)+\frac{\pi}{2}, \\
			\mathcal{M}_3(t, k, \xi) & := \int_{\mathbb{R}}\langle\frac{1}{l}\rangle^{-\frac13-\kappa} \frac{1}{|l|^\frac43}\left(\operatorname{sgn}(l) \arctan \left(\frac{\xi+t(k-l)}{1+|k-l|+|l|}\right)+\frac{\pi}{2}\right) d l .
		\end{aligned}
	\end{equation}
	Then the three multipliers above are self-adjoint Fourier multipliers and verify that
	\begin{equation} \label{eq:bound of M}
		1\leq \mathcal{M}:=\mathcal{M}_1+\mathcal{M}_2+\mathcal{M}_3+1 \leq C_\kappa.
	\end{equation}
	We also denote
	\begin{align} \label{eq:m3}
		\Upsilon(t, k, \xi) :=\left(-\partial_t+k \partial_{\xi}\right) \mathcal{M}_3=\int_{\mathbb{R}}\langle\frac{1}{l}\rangle^{-\frac13-\kappa} \frac{1}{|l|^\frac13} \frac{1+|k-l|+|l|}{(1+|k-l|+|l|)^2+|\xi+t(k-l)|^2} d l,
	\end{align}
	with $\Upsilon(t, k, \xi)>0$ for all $t>0$ and $k,\xi \in \mathbb{R}$.
	
	The enhanced dissipation multiplier $\mathcal{M}_1$ and the inviscid damping multiplier $\mathcal{M}_2$ are  motivated by \cite{DWZ2021}. The construction of multiplier $\mathcal{M}_3$ is motivated by \cite{WZ2023}. The multiplier $\mathcal{M}_3$ is designed to control the growth of the reaction term caused by echo cascades. 
	\textbf{An important observation} is that the multiplier $\mathcal{M}_3$ constructed here is not only effective in controlling the real reaction terms $u^{\mathrm{NL}} \cdot \nabla \omega^{\mathrm{L}}$ and $u^{\mathrm{NL}} \cdot \nabla |D_x|^{1/3} \theta^{\mathrm{L}}$, but also sufficient to handle staggered derivative term $|D_x|^{1/3} u^{\mathrm{NL}} \cdot \nabla \theta^{\mathrm{L}}$ (see around \eqref{eq:im} and \eqref{eq:im2} for more details).
	On the Fourier side, these multipliers can be understood as some kind of `ghost weight', which are bounded weights providing additional dissipation properties. The most crucial feature of the three multipliers mentioned above is that
	$$
	2 \Re\left(\left(\partial_t+y \partial_x\right) f \mid \mathcal{M}_i f\right)=\frac{d}{d t}\|\sqrt{\mathcal{M}_i} f\|_{L^2}^2+\int_{\mathbb{R}^2}\left(-\partial_t+k \partial_{\xi}\right) \mathcal{M}_i(t, k, \xi)|\hat{f}|^2 d k d \xi,
	$$
	for some $f= f(t, x ,y)$ on $(0, +\infty) \times \mathbb{R}^2$.
	\begin{Lem} \label{lem:M12}
		For smooth enough function $f= f(t, x ,y)$ on $(0, +\infty) \times \mathbb{R}^2$, it holds that
		\begin{align}  \label{eq:m1m2}
			\int_{\mathbb{R}^2}\left(-\partial_t+k \partial_{\xi}\right) &\mathcal{M}(t, k, \xi)|\hat{f}|^2 d k d \xi \no \\
			& \geq \frac{\nu^{\frac{1}{3}}}{4}\|\left|D_x\right|^{\frac{1}{3}} f\|_{L^2}^2-\frac{\nu}{2}\|\partial_y f\|_{L^2}^2 + \|\partial_x \nabla \Delta^{-1} f\|_{L^2}^2 + \|\sqrt{\Upsilon} f\|_{L^2}^2.
		\end{align}
	\end{Lem}
	\begin{proof}
		Direct calculation shows that
		$$
		k \partial_{\xi} \mathcal{M}_1(k, \xi)=\frac{\nu^{\frac{1}{3}}|k|^{\frac{2}{3}}}{1+\nu^{\frac{2}{3}}|k|^{-\frac{2}{3}}|\xi|^2} \geq \frac{1}{4} \nu^{\frac{1}{3}}|k|^{\frac{2}{3}}-\frac{1}{2} \nu|\xi|^2
		$$
		and
		$$
		k \partial_{\xi} \mathcal{M}_2(k, \xi)=\frac{k^2}{k^2+\xi^2},
		$$
		which, along with \eqref{eq:m3}, completes the proof.
	\end{proof}

	\subsection{Proof of Theorem \ref{thm:main}} The proof of Theorem \ref{thm:main} relies on the following three key propositions, which will be proved in next sections:
	\begin{Prop} \label{lem:linear theta}
		For initial data satisfying \eqref{eq:the initial data}, there exists a universal constant $0<c_0<\frac{1}{16\pi}$ such that for any $t>0$, there holds that
		\begin{equation*}
			\begin{aligned}
				\| e^{c_0 \nu^\frac13  \left|D_x\right|^\frac23 t}   &\langle D_x, D_y+t D_x\rangle^6  \langle\frac{1}{D_x}\rangle^4  \omega^{\mathrm{L}}\|_{L_t^{\infty} L^2}  \\
				& +\nu^{\frac{1}{6}}\|e^{c_0 \nu^\frac13 \left|D_x\right|^\frac23 t}\langle D_x, D_y+t D_x\rangle^6\langle\frac{1}{D_x}\rangle^4\left|D_x\right|^{\frac{1}{3}} \omega^{\mathrm{L}}\|_{L_t^2 L^2} \lesssim \nu^{\frac{1}{3}+ \delta}
			\end{aligned}
		\end{equation*}
		and
		\begin{equation*}
			\begin{aligned}
				\| e^{c_0 \nu^\frac13  \left|D_x\right|^\frac23 t}   &\langle D_x, D_y+t D_x\rangle^6  \langle\frac{1}{D_x}\rangle^4  \langle D_x \rangle \theta^{\mathrm{L}}\|_{L_t^{\infty} L^2}  \\
				& +\nu^{\frac{1}{6}}\|e^{c_0 \nu^\frac13 \left|D_x\right|^\frac23 t}\langle D_x, D_y+t D_x\rangle^6\langle\frac{1}{D_x}\rangle^4\left|D_x\right|^{\frac{1}{3}} \langle D_x \rangle \theta^{\mathrm{L}}\|_{L_t^2 L^2} \lesssim \nu^{\frac{2}{3}+ 2\delta} .
			\end{aligned}
		\end{equation*}
	\end{Prop}
	
	\begin{Prop} \label{cor}
		For initial data satisfying \eqref{eq:the initial data}, $T_0 = \nu^{-\frac16}$ and $1<p<2$, there exists a small enough $0< \nu_1<1$ such that if $0<\nu<\nu_1$, it holds that
		\begin{equation}\begin{aligned} \label{eq:temp7}
				&\|\langle D_x, D_y+t D_x\rangle^3 \omega^{\mathrm{NL}}\|_{L_{T_0}^\infty L^2} + \|\langle D_x, D_y+t D_x\rangle^2 \omega^{\mathrm{NL}}\|_{L_{T_0}^\infty L^p} \lesssim \nu^{\frac23 + 2\delta}, \\
				&\|\langle D_x, D_y+t D_x\rangle^3 \theta^{\mathrm{NL}}\|_{L_{T_0}^\infty L^2} + \|\langle D_x, D_y+t D_x\rangle^2 \theta^{\mathrm{NL}}\|_{L_{T_0}^\infty L^p} \lesssim \nu^{1 + 3\delta}, \\
				&\|\langle D_x, D_y+t D_x\rangle^3 \px \theta^{\mathrm{NL}}\|_{L_{T_0}^\infty L^2} + \|\langle D_x, D_y+t D_x\rangle^2 \px \theta^{\mathrm{NL}}\|_{L_{T_0}^\infty L^p} \lesssim \nu^{\frac56 + 3\delta}.
		\end{aligned}\end{equation}
	\end{Prop}

	\begin{Cor} \label{cor2}
		Under the same assumptions as in Proposition \ref{cor}, there exists a small enough $0< \nu_1<1$ such that if $0<\nu<\nu_1$, there holds that
		\begin{equation*}\begin{aligned} 
				&\|e^{c \nu^{\frac{1}{3}} \lambda\left(D_x\right) t}\left\langle D_x\right\rangle\left\langle\frac{1}{D_x}\right\rangle^\epsilon \omega^{\mathrm{NL}}\|_{L_{T_0}^\infty L^2}  \lesssim \nu^{\frac23 + 2\delta}, \\
				&\|e^{c \nu^{\frac{1}{3}} \lambda\left(D_x\right) t}\left\langle D_x\right\rangle\left\langle\frac{1}{D_x}\right\rangle^\epsilon \left\langle D_x\right\rangle^\frac13 \theta^{\mathrm{NL}}\|_{L_{T_0}^\infty L^2}    \lesssim \nu^{\frac56 + 3\delta}.
		\end{aligned}\end{equation*}
	\end{Cor}
	\begin{proof}
		For some $f$, we have that
		\begin{equation*}
			\begin{aligned}
				\left\|e^{c \nu^{\frac{1}{3}} \lambda\left(D_x\right) t}\left\langle D_x\right\rangle\left\langle\frac{1}{D_x}\right\rangle^\epsilon f\right\|_{L_{T_0}^\infty L^2} &\lesssim\left\|\left\langle D_x\right\rangle\left\langle\frac{1}{D_x}\right\rangle^\epsilon f\right\|_{L_{T_0}^\infty L^2} \\
				& \lesssim \left\|\langle k, \xi+k t\rangle^2 \hat{f}_k(\xi)\right\|_{L_{T_0}^\infty L_{k, \xi}^{p^{\prime}}} \left\|\left\langle\frac{1}{k}\right\rangle^\epsilon\langle k, \xi+k t\rangle^{-1}\right\|_{L_{T_0}^\infty L_{k, \xi}^{\frac{2 p}{2-p}}} \\
				& \lesssim\left\|\left\langle D_x, D_y+t D_x\right\rangle^2 f\right\|_{L_{T_0}^\infty L^p},
			\end{aligned}
		\end{equation*}
		where we used that for small $\epsilon<\frac{1}{2}$, if we choose $1<p<2$ sufficiently close to $1$ such that $\frac{1}{p}>\epsilon+\frac{1}{2}$, then we have
		\begin{equation*}
			\left\|\left\langle\frac{1}{k}\right\rangle^\epsilon\langle k, \xi+k t\rangle^{-1}\right\|_{ L_{k, \xi}^{\frac{2 p}{2-p}}} < \infty.
		\end{equation*}
		Then, Proposition \ref{cor} completes the proof.
	\end{proof}

	\begin{Prop} \label{prop2.3}
		For initial data satisfying \eqref{eq:the initial data}, $T_0 = \nu^{-\frac16}$, there exists a small enough $0< \nu_2<1$ such that if $0<\nu<\nu_2$, it holds that
		\begin{equation*}
			\begin{aligned}
				&\| e^{c \nu^{\frac{1}{3}} \lambda(D_x) t} \langle D_x \rangle \langle \tfrac{1}{D_x} \rangle^\epsilon \omega^{\mathrm{NL}} \|_{L^\infty_{[T_0, +\infty]} L^2}  + \|e^{c \nu^{\frac{1}{3}} \lambda\left(D_x\right) t}\langle D_x\rangle\langle\frac{1}{D_x}\rangle^\epsilon \partial_x \nabla \phi^{\mathrm{NL}}\|_{L_{[T_0, +\infty]}^2 L^2} \\
				&\quad + \nu^{-\frac{1}{3} - \delta} \| e^{c \nu^{\frac{1}{3}} \lambda(D_x) t} \langle D_x \rangle \langle \tfrac{1}{D_x} \rangle^\epsilon \langle D_x \rangle^{\frac{1}{3}} \theta^{\mathrm{NL}} \|_{L^\infty_{[T_0, +\infty]} L^2} 
				\lesssim \nu^{\frac{1}{2} + 2\delta}.
			\end{aligned}
		\end{equation*}
	\end{Prop}
	
	\begin{proof}[Proof of Theorem \ref{thm:main}]
		Denoting $\nu_0= \min\{\nu_1, \nu_2\}$ and using  Proposition \ref{lem:linear theta}, Corollary \ref{cor2} and Proposition \ref{prop2.3},  the proof is complete.
	\end{proof}

	\section{Estimates of the linear part} \label{sec.3}
	Denote $\omega^{\mathrm{L}}$ and $\theta^{\mathrm{L}}$ by a pair of solutions that solve the following linear equation:
	\begin{align} \label{eq:linear equ}
		\left\{\begin{array}{l}
			\partial_t \omega^{\mathrm{L}}-\nu \Delta \omega^{\mathrm{L}}+y \partial_x \omega^{\mathrm{L}}  = \partial_x \theta^{\mathrm{L}}  ,  \\
			\partial_t \theta^{\mathrm{L}}  -  \mu \Delta \theta^{\mathrm{L}}+y \partial_x \theta^{\mathrm{L}}  = 0  ,\\
			\left.\omega^{\mathrm{L}} \right|_{t=0}=\omega_{\text {in}}(x, y),\quad \left.\theta^{\mathrm{L}} \right|_{t=0}= \theta_{\text{in}}(x, y).
		\end{array}\right.
	\end{align}
	It is easy to check that the operator $\langle D_x, D_y + tD_x\rangle^a$, $a\in \mathbb{Z}^+$, commutes with the differential operator with variable coefficients $\partial_t+y \partial_x$ and thus
	\begin{align*} 
		\left\{\begin{array}{l}
			(\partial_t -\nu \Delta +y \partial_x) \langle D_x, D_y + tD_x\rangle^a  \omega^{\mathrm{L}}  = \partial_x \langle D_x, D_y + tD_x\rangle^a  \theta^{\mathrm{L}}  ,  \\
			(\partial_t   -  \mu \Delta +  y \partial_x ) \langle D_x, D_y + tD_x\rangle^a  \theta^{\mathrm{L}}  = 0  ,\\
			\left. \langle D_x, D_y + tD_x\rangle^a  \omega^{\mathrm{L}} \right|_{t=0}  =  \langle D_x, D_y \rangle^a  \omega_{\text {in}}(x, y),\\ \left. \langle D_x, D_y + tD_x\rangle^a  \theta^{\mathrm{L}} \right|_{t=0}= \langle D_x, D_y \rangle^a  \theta_{\text{in}}(x, y).
		\end{array}\right.
	\end{align*}
	The solution to the equations above can be represented by the Fourier methods:
	\begin{equation}
		\begin{aligned}\label{eq:solutions of the linear equ}
			\langle k, \xi + kt \rangle^a  \widehat{\theta^{\mathrm{L}} }(t, k, \xi)&  = \langle k, \xi + kt \rangle^a  \widehat{\theta_{\text {in}} }(k, \xi+k t) e^{-\mu \int_0^t|k|^2+|\xi+k(t-\tau)|^2 d \tau},\\
			\langle k, \xi + kt \rangle^a  \widehat{\omega^{\mathrm{L}} }(t, k, \xi)& = \langle k, \xi + kt \rangle^a  \widehat{\omega_{\text {in}} }(k, \xi+kt) e^{-\nu \int_0^t|k|^2+|\xi+k(t-\tau)|^2 d \tau}\\
			&\quad  + \int_0^t ik \langle k, \xi + ks \rangle^a  \widehat{\theta^{\mathrm{L}} }(s, k, \xi) e^{-\nu \int_0^{t-s} |k|^2+|\xi+k(t-\tau)|^2 d \tau} ds     .
		\end{aligned}
	\end{equation}
	When $\nu =\mu$, substituting $\eqref{eq:solutions of the linear equ}_1$ into $\eqref{eq:solutions of the linear equ}_2$ yields that
	\begin{equation}\begin{aligned} \label{eq:solutions of the linear equ2}
			\langle k, \xi + kt \rangle^a  \widehat{\omega^{\mathrm{L}} }(t, k, \xi)& = \langle k, \xi + kt \rangle^a  \widehat{\omega_{\text {in}} }(k, \xi+kt) e^{-\nu \int_0^t|k|^2+|\xi+k(t-\tau)|^2 d \tau}\\
			&\quad  +  e^{-\nu \int_0^t|k|^2+|\xi+k(t-\tau)|^2 d \tau} \int_0^t ik \langle k, \xi + ks \rangle^a   \widehat{\theta_{\text {in}} }(k, \xi+k s) ds.
	\end{aligned}\end{equation}
	
	With the solution above, we can obtain estimates at each time instant in the frequency space.
	
	\begin{Lem} \label{lem:linear Linfty}
		For initial data satisfying \eqref{eq:the initial data}, there exists a universal constant $0< c_0 <\frac{1}{12}$ such that for $2\leq p \leq \infty$ and for any $t>0$, it holds that
		\begin{equation*}\begin{aligned}
				&\|e^{c_0 \nu |k|^2 t^3}  \langle k, \xi + kt \rangle^6\langle\frac{1}{k}\rangle^4 \langle k \rangle \widehat{\theta^{\mathrm{L}} } (t)\|_{L_{k,\xi}^p} \leq \nu^{\frac23 + 2\delta} , \\
				&\|e^{c_0 \nu |k|^2 t^3}  \langle k, \xi + kt \rangle^6\langle\frac{1}{k}\rangle^4 \widehat{\omega^{\mathrm{L}} } (t)\|_{L_{k,\xi}^p}\leq C(c_0) \nu^{\frac13 + \delta}.
		\end{aligned}\end{equation*}
	\end{Lem}
	\begin{proof}
		For $2\leq p \leq \infty$, $t\geq 0$ and $c_0 <\frac{1}{12}$, we deduce from $\eqref{eq:solutions of the linear equ}_1$ and \eqref{eq:the initial data} that
		\begin{equation*}\begin{aligned}
				\|e^{c_0 \nu |k|^2 t^3} & \langle k, \xi + kt \rangle^6\langle\frac{1}{k}\rangle^4 \langle k \rangle \widehat{\theta^{\mathrm{L}} }(t, k, \xi) \|_{L_{k, \xi}^p}\\
				&\leq \| \langle k, \xi + kt \rangle^6\langle\frac{1}{k}\rangle^4 \langle k \rangle \widehat{\theta_{\text {in}} }(k, \xi+k t) \|_{L_{k, \xi}^p} 
				= \| \langle k, \xi \rangle^6\langle\frac{1}{k}\rangle^4 \langle k \rangle \widehat{\theta_{\text {in}} }(k, \xi) \|_{L_{k, \xi}^p} \\
				&\leq  \| \langle k, \xi \rangle^6\langle\frac{1}{k}\rangle^4 \langle k \rangle \widehat{\theta_{\text {in}} } \|_{L^2}^{\frac2p} \| \langle k, \xi \rangle^6\langle\frac{1}{k}\rangle^4 \langle k \rangle \widehat{\theta_{\text {in}} }\|_{L^\infty}^{1- \frac2p} \\
				&\leq \|\langle D_x, D_y\rangle^6\langle\frac{1}{D_x}\rangle^4 \langle D_x \rangle \theta_{\mathrm{in}}\|_{L^2}^{\frac2p} \|\langle D_x, D_y\rangle^6\langle\frac{1}{D_x}\rangle^4 \langle D_x \rangle \theta_{\mathrm{in}}\|_{L^1}^{1- \frac2p} 
				\leq  \nu^{\frac23 + 2\delta},
		\end{aligned}\end{equation*}
		where we used that
		\begin{equation*}\begin{aligned}
				\int_0^t|k|^2+|\xi+k(t-\tau)|^2 d \tau =|k|^2 t+\frac{|k|^2 t^3}{12}+\left|\xi+\frac{k t}{2}\right|^2 t .
		\end{aligned}\end{equation*}
		Similarly, by \eqref{eq:solutions of the linear equ2} and \eqref{eq:the initial data} again, we have
		\begin{align*}
			\|e^{c_0 \nu |k|^2 t^3}  & \langle k, \xi + kt \rangle^6\langle\frac{1}{k}\rangle^4 \widehat{\omega^{\mathrm{L}} }(t, k, \xi) \|_{L_{k, \xi}^p}\\
			&\leq \| \langle k, \xi + kt \rangle^6\langle\frac{1}{k}\rangle^4  \widehat{\omega_{\text {in}} }(k, \xi+k t) \|_{L_{k, \xi}^p} \\
			&\qquad \qquad +  \| e^{ (c_0 - \frac{1}{12} ) \nu |k|^2 t^3}    \int_0^t ik \langle k, \xi + ks \rangle^6\langle\frac{1}{k}\rangle^4  \widehat{\theta_{\text {in}} }(k, \xi+k s) ds   \|_{L_{k, \xi}^p} \\
			&= \| \langle k, \xi \rangle^6\langle\frac{1}{k}\rangle^4  \widehat{\omega_{\text {in}} }(k, \xi) \|_{L_{k, \xi}^p}    +   \|  e^{ (c_0 - \frac{1}{12} ) \nu |k|^2 t^3} ikt \langle k, \xi \rangle^6\langle\frac{1}{k}\rangle^4  \widehat{\theta_{\text {in}} }(k, \xi) \|_{L_{k, \xi}^p}\\
			&\leq   \| \langle k, \xi \rangle^6\langle\frac{1}{k}\rangle^4  \widehat{\omega_{\text {in}} } \|_{L^2}^{\frac2p} \| \langle k, \xi \rangle^6\langle\frac{1}{k}\rangle^4  \widehat{\omega_{\text {in}} } \|_{L^\infty}^{1- \frac2p}  \\
			&\qquad + (\frac{1}{12} - c_0)^{-\frac13} \nu^{-\frac13} \| \langle k, \xi \rangle^6\langle\frac{1}{k}\rangle^4  |k|^\frac13 \widehat{\theta_{\text {in}} }\|_{L^2}^{\frac2p} \| \langle k, \xi \rangle^6\langle\frac{1}{k}\rangle^4  |k|^\frac13\widehat{\theta_{\text {in}} } \|_{L^\infty}^{1- \frac2p}   \\
			& \leq C(c_0) \nu^{\frac13 + \delta},
		\end{align*}
		where we used 
		\begin{equation*}\begin{aligned}
				|k|^\frac23 t e^{ (c_0 - \frac{1}{12} ) \nu |k|^2 t^3} \leq (\frac{1}{12} - c_0)^{-\frac13} \nu^{-\frac13}.
		\end{aligned}\end{equation*}
		The proof is completed.
	\end{proof}

	Analogously, in the frequency space, one can also obtain the following lemma by using the Fourier multiplier $\mathcal{M}_1$.
	\begin{Lem} \label{lem:linear}
		For initial data satisfying \eqref{eq:the initial data} and any $m>0$, there exists a universal constant $0< c_0 < \frac{m^{p-1}}{8(m+\pi)^p}$ such that for $ 1\leq p< +\infty$ and for any $t>0$, it holds that
		\begin{equation*}\begin{aligned} 
				\|  e^{c_0 \nu^\frac13 |k|^\frac23 t} \langle k, & \xi + kt\rangle^6 \langle \frac{1}{k}\rangle^4 \langle |k|^\frac13 \rangle \widehat{\theta^{\mathrm{L}}} \|_{L_t^\infty L^{2p}} \\
				&+ \nu^\frac{1}{6p} \| e^{ c_0 \nu^\frac13 |k|^\frac23 t}    \langle k, \xi + kt\rangle^6 \langle \frac{1}{k}\rangle^4 \langle |k|^\frac13 \rangle |k|^\frac{1}{3p} \widehat{\theta^{\mathrm{L}}} \|_{L_t^{2p} L^{2p}} \lesssim \nu^{\frac{2}{3} + 2\delta} 
		\end{aligned}\end{equation*}
		and
		\begin{equation*}\begin{aligned} 
				\|  e^{c_0 \nu^\frac13 |k|^\frac23 t} \langle k, &\xi + kt\rangle^6 \langle \frac{1}{k}\rangle^4 \widehat{\omega^{\mathrm{L}}} \|_{L_t^\infty L^{2p}} \\
				&+ \nu^\frac{1}{6p} \| e^{ c_0 \nu^\frac13 |k|^\frac23 t}    \langle k, \xi + kt\rangle^6 \langle \frac{1}{k}\rangle^4  |k|^\frac{1}{3p} \widehat{\omega^{\mathrm{L}}} \|_{L_t^{2p} L^{2p}} \lesssim \nu^{\frac{1}{3} +\delta} .
		\end{aligned}\end{equation*}
	\end{Lem}
	\begin{proof}
		\underline{\bf Estimates of $\widehat{\theta^{\mathrm{L}}}$. }
		Denoting that $ \mathcal{M}_1' (k, \xi):= {\mathcal{M}_1} (k, \xi) + m$ and $M:= m+\pi$, it holds that $m<\mathcal{M}_1' (k, \xi)<M$. Applying the Fourier transform for $\eqref{eq:linear equ}_2$, one obtains 
		\begin{align} \label{eq:linear Fourier}
			(\pt- k\pxi) \widehat{\theta^{\mathrm{L}} }(t, k, \xi)  + \mu(k^2 + \xi^2) \widehat{\theta^{\mathrm{L}} }(t, k, \xi)   = 0.
		\end{align}
		Thanks to 
		\begin{align*}
			&\pt\lt( | \sqrt{\mathcal{M}_1'(k, \xi)} e^{c_0 \mu^\frac13 |k|^\frac23 t} \widehat{\theta^{\mathrm{L}} }|^{2p-2} \mathcal{M}_1'(k, \xi) e^{2c_0 \mu^\frac13 |k|^\frac23 t} \widehat{\theta^{\mathrm{L}} } \rt) \\
			&=c_0\mu^\frac13 |k|^\frac23 | \sqrt{\mathcal{M}_1'(k, \xi)} e^{c_0 \mu^\frac13 |k|^\frac23 t} \widehat{\theta^{\mathrm{L}} }|^{2p-2}  \mathcal{M}_1'(k, \xi) e^{2c_0 \mu^\frac13 |k|^\frac23 t} \widehat{\theta^{\mathrm{L}} }\\
			& \quad + (2p-1) \sqrt{\mathcal{M}_1'(k, \xi)} e^{c_0 \mu^\frac13 |k|^\frac23 t} \\
			&\qquad \times | \sqrt{\mathcal{M}_1'(k, \xi)} e^{c_0 \mu^\frac13 |k|^\frac23 t} \widehat{\theta^{\mathrm{L}} }|^{2p-2} \pt( \sqrt{\mathcal{M}_1'(k, \xi)} e^{c_0 \mu^\frac13 |k|^\frac23 t} \widehat{\theta^{\mathrm{L}} }),
		\end{align*}
		\begin{align*}
			&\partial_\xi \lt( | \sqrt{\mathcal{M}_1'(k, \xi)} e^{c_0 \mu^\frac13 |k|^\frac23 t} \widehat{\theta^{\mathrm{L}} }|^{2p-2} \mathcal{M}_1'(k, \xi) e^{2c_0 \mu^\frac13 |k|^\frac23 t} \widehat{\theta^{\mathrm{L}} } \rt) \\
			&=  p \partial_\xi  \mathcal{M}_1'(k, \xi, \mu)  | \sqrt{\mathcal{M}_1'(k, \xi)} e^{c_0 \mu^\frac13 |k|^\frac23 t} \widehat{\theta^{\mathrm{L}} }|^{2p-2}    e^{2c_0 \mu^\frac13 |k|^\frac23 t} \widehat{\theta^{\mathrm{L}} }\\
			& \quad + (2p-1) \mathcal{M}_1'(k, \xi)  | \sqrt{\mathcal{M}_1'(k, \xi)} e^{c_0 \mu^\frac13 |k|^\frac23 t} \widehat{\theta^{\mathrm{L}} }|^{2p-2}   e^{2c_0 \mu^\frac13 |k|^\frac23 t}    \partial_\xi\widehat{\theta^{\mathrm{L}} },
		\end{align*}
		and taking the inner product of $\eqref{eq:linear Fourier}$ with $$2p | \sqrt{\mathcal{M}_1'(k, \xi)} e^{c_0 \mu^\frac13 |k|^\frac23 t} \widehat{\theta^{\mathrm{L}} }|^{2p-2} \mathcal{M}_1'(k, \xi) e^{2c_0 \mu^\frac13 |k|^\frac23 t} \widehat{\theta^{\mathrm{L}} },$$ we have
		\begin{align*}
			&\big( \pt \widehat{\theta^{\mathrm{L}}} \big| 2p | \sqrt{\mathcal{M}_1'(k, \xi)} e^{c_0 \mu^\frac13 |k|^\frac23 t} \widehat{\theta^{\mathrm{L}} }|^{2p-2} \mathcal{M}_1'(k, \xi) e^{2c_0 \mu^\frac13 |k|^\frac23 t} \widehat{\theta^{\mathrm{L}} } \big) \\
			&= \frac{d}{dt} \big(  \widehat{\theta^{\mathrm{L}}} \big| 2p | \sqrt{\mathcal{M}_1'(k, \xi)} e^{c_0 \mu^\frac13 |k|^\frac23 t} \widehat{\theta^{\mathrm{L}} }|^{2p-2} \mathcal{M}_1'(k, \xi, \mu) e^{2c_0 \mu^\frac13 |k|^\frac23 t} \widehat{\theta^{\mathrm{L}} } \big) \\
			& \quad - 2p\big(  \widehat{\theta^{\mathrm{L}}} \big| \pt ( | \sqrt{\mathcal{M}_1'(k, \xi)} e^{c_0 \mu^\frac13 |k|^\frac23 t} \widehat{\theta^{\mathrm{L}} }|^{2p-2} \mathcal{M}_1'(k, \xi) e^{2c_0 \mu^\frac13 |k|^\frac23 t} \widehat{\theta^{\mathrm{L}} } ) \big)\\
			&= 2p \frac{d}{dt} \| | \sqrt{\mathcal{M}_1'(k, \xi)} e^{c_0 \mu^\frac13 |k|^\frac23 t} \widehat{\theta^{\mathrm{L}} }|^p \|_{L^2}^2 - 2p c_0 \mu^\frac13 \| |k|^\frac13 | \sqrt{\mathcal{M}_1'(k, \xi)} e^{c_0 \mu^\frac13 |k|^\frac23 t} \widehat{\theta^{\mathrm{L}} }|^p\|_{L^2}^2 \\
			& \quad - (2p  -1 ) \frac{d}{dt} \| | \sqrt{\mathcal{M}_1'(k, \xi)} e^{c_0 \mu^\frac13 |k|^\frac23 t} \widehat{\theta^{\mathrm{L}} }|^p \|_{L^2}^2 \\
			&=  \frac{d}{dt} \| | \sqrt{\mathcal{M}_1'(k, \xi)} e^{c_0 \mu^\frac13 |k|^\frac23 t} \widehat{\theta^{\mathrm{L}} }|^p \|_{L^2}^2  - 2p c_0 \mu^\frac13 \| |k|^\frac13 | \sqrt{\mathcal{M}_1'(k, \xi)} e^{c_0 \mu^\frac13 |k|^\frac23 t} \widehat{\theta^{\mathrm{L}} }|^p\|_{L^2}^2 ,
		\end{align*}
		\begin{align*}
			\big( k\partial_\xi  \widehat{\theta^{\mathrm{L}}} \big| &2p | \sqrt{\mathcal{M}_1'(k, \xi)} e^{c_0 \mu^\frac13 |k|^\frac23 t} \widehat{\theta^{\mathrm{L}} }|^{2p-2} \mathcal{M}_1'(k, \xi) e^{2c_0 \mu^\frac13 |k|^\frac23 t} \widehat{\theta^{\mathrm{L}} } \big) \\
			&=  - 2p \big( k \widehat{\theta^{\mathrm{L}}} \big| \partial_\xi ( | \sqrt{\mathcal{M}_1'(k, \xi)} e^{c_0 \mu^\frac13 |k|^\frac23 t} \widehat{\theta^{\mathrm{L}} }|^{2p-2} \mathcal{M}_1'(k, \xi) e^{2c_0 \mu^\frac13 |k|^\frac23 t} \widehat{\theta^{\mathrm{L}} } ) \big)\\
			&= -2p^2 \int_{\mathbb{R}^2} k\partial_\xi  \mathcal{M}_1'(k, \xi)    | \sqrt{\mathcal{M}_1'(k, \xi)} |^{2p-2} |e^{c_0 \mu^\frac13 |k|^\frac23 t} \widehat{\theta^{\mathrm{L}} }|^{2p} dkd\xi \\
			&\quad  -  2p(2p-1)\langle k \widehat{\theta^{\mathrm{L}}},   \mathcal{M}_1'(k, \xi)  | \sqrt{\mathcal{M}_1'(k, \xi)} e^{c_0 \mu^\frac13 |k|^\frac23 t} \widehat{\theta^{\mathrm{L}} }|^{2p-2}   e^{2c_0 \mu^\frac13 |k|^\frac23 t}    \partial_\xi\widehat{\theta^{\mathrm{L}} } \rangle
		\end{align*}
		and
		\begin{align*}
			&\big( \mu(k^2 + \xi^2) \widehat{\theta^{\mathrm{L}} }  \big| 2p | \sqrt{\mathcal{M}_1'(k, \xi)} e^{c_0 \mu^\frac13 |k|^\frac23 t} \widehat{\theta^{\mathrm{L}} }|^{2p-2} \mathcal{M}_1'(k, \xi) e^{2c_0 \mu^\frac13 |k|^\frac23 t} \widehat{\theta^{\mathrm{L}} } \big) \\
			&=  2p\mu \|  (k^2 + \xi^2)^\frac12 | \sqrt{\mathcal{M}_1'(k, \xi)} e^{c_0 \mu^\frac13 |k|^\frac23 t} \widehat{\theta^{\mathrm{L}} }|^{p}  \|_{L^2}^2,
		\end{align*}
		which imply
		\begin{align*}
			-\big( k\partial_\xi  \widehat{\theta^{\mathrm{L}}} \big| &2p | \sqrt{\mathcal{M}_1'(k, \xi)} e^{c_0 \mu^\frac13 |k|^\frac23 t} \widehat{\theta^{\mathrm{L}} }|^{2p-2} \mathcal{M}_1'(k, \xi) e^{2c_0 \mu^\frac13 |k|^\frac23 t} \widehat{\theta^{\mathrm{L}} } \big) \\
			&= p \int_{\mathbb{R}^2} k\partial_\xi  \mathcal{M}_1'(k, \xi) | \sqrt{\mathcal{M}_1'(k, \xi)} |^{2p-2} |e^{c_0 \mu^\frac13 |k|^\frac23 t} \widehat{\theta^{\mathrm{L}} }|^{2p} dkd\xi  \\
			&\geq  \frac{pm^{p-1} \mu^\frac13 }{4} \| |k|^\frac13 |e^{c_0 \mu^\frac13 |k|^\frac23 t} \widehat{\theta^{\mathrm{L}} }|^{p} \|_{L^2}^2  -  \frac{pM^{p-1} \mu}{2} \| |\xi| |e^{c_0 \mu^\frac13 |k|^\frac23 t} \widehat{\theta^{\mathrm{L}} }|^{p} \|_{L^2}^2
		\end{align*}
		and
		\begin{align*}
			&\frac{d}{dt} \| | \sqrt{\mathcal{M}_1'(k, \xi)} e^{c_0 \mu^\frac13 |k|^\frac23 t} \widehat{\theta^{\mathrm{L}} }|^p \|_{L^2}^2  - 2p c_0 \mu^\frac13 \| |k|^\frac13 | \sqrt{\mathcal{M}_1'(k, \xi)} e^{c_0 \mu^\frac13 |k|^\frac23 t} \widehat{\theta^{\mathrm{L}} }|^p\|_{L^2}^2 \\
			&+2p\mu \|  (k^2 + \xi^2)^\frac12 | \sqrt{\mathcal{M}_1'(k, \xi)} e^{c_0 \mu^\frac13 |k|^\frac23 t} \widehat{\theta^{\mathrm{L}} }|^{p}  \|_{L^2}^2 \\
			& + p \int_{\mathbb{R}^2} k\partial_\xi  \mathcal{M}_1'(k, \xi) | \sqrt{\mathcal{M}_1'(k, \xi)} |^{2p-2} |e^{c_0 \mu^\frac13 |k|^\frac23 t} \widehat{\theta^{\mathrm{L}} }|^{2p} dkd\xi = 0.
		\end{align*}
		Since $m<\mathcal{M}_1'(k, \xi, \mu)<M $, it holds that
		\begin{align*}
			& \frac{d}{dt} \| | \sqrt{\mathcal{M}_1'(k, \xi)} e^{c_0 \mu^\frac13 |k|^\frac23 t} \widehat{\theta^{\mathrm{L}} }|^p \|_{L^2}^2  + \frac{p\mu^\frac13}{4} (m^{p-1}  - 8c_0 M^p)  \| |k|^\frac13 | e^{c_0 \mu^\frac13 |k|^\frac23 t} \widehat{\theta^{\mathrm{L}} }|^p\|_{L^2}^2 \\
			&+ p m^p\mu \|  |k| |  e^{c_0 \mu^\frac13 |k|^\frac23 t} \widehat{\theta^{\mathrm{L}} }|^{p}  \|_{L^2}^2  + \frac{\mu p}{2}(2m^p-M^{p-1})  \| |\xi| |e^{c_0 \mu^\frac13 |k|^\frac23 t} \widehat{\theta^{\mathrm{L}} }|^{p} \|_{L^2}^2 \leq 0.
		\end{align*}
		Then, thanks to $c_0 < \frac{m^{p-1}}{8M^p} = \frac{m^{p-1}}{8(m+\pi)^p}$, we have 
		\begin{align} \label{eq:est of theta L}
			\frac{d}{dt} \| \sqrt{\mathcal{M}_1'(k, \xi)}   e^{c_0 \mu^\frac13 |k|^\frac23 t} \widehat{\theta^{\mathrm{L}} } \|_{L^{2p}}^{2p}  +  \mu^\frac13 \| |k|^\frac{1}{3p} e^{c_0 \mu^\frac13 |k|^\frac23 t} \widehat{\theta^{\mathrm{L}} }\|_{L^{2p}}^{2p} \leq 0.
		\end{align}
		Since the operator $\langle D_x, D_y + tD_x\rangle^6$ commutes with the differential operator with variable coefficients $\partial_t+y \partial_x$, we get that
		\begin{equation}\begin{aligned} \label{eq:est of linear theta L2}
				\|  &e^{c_0 \mu^\frac13 |k|^\frac23 t} \langle k, \xi + kt\rangle^6 \langle \frac{1}{k}\rangle^4 \langle |k|^\frac13 \rangle \widehat{\theta^{\mathrm{L}}} \|_{L^\infty L^{2p}}^{2p} \\
				&+ \mu^\frac13 \| e^{ c_0 \mu^\frac13 |k|^\frac23 t}    \langle k, \xi + kt\rangle^6 \langle \frac{1}{k}\rangle^4 \langle |k|^\frac13 \rangle |k|^\frac{1}{3p} \widehat{\theta^{\mathrm{L}}} \|_{L^{2p} L^{2p}}^{2p} \lesssim \|  \langle k, \xi\rangle^6 \langle \frac{1}{k}\rangle^4 \langle |k|^\frac13 \rangle \widehat{\theta_{\rm in}} \|_{L^{2p}}^{2p} . 
		\end{aligned}\end{equation}

		\underline{\bf Estimates of  $\widehat{\omega^{\mathrm{L}}}$. }
		Applying Fourier transform for $\eqref{eq:linear equ}_1$, we get that 
		\begin{equation*}\begin{aligned} \label{eq:linear Fourier for omega}
				(\pt- k\pxi) \widehat{\omega^{\mathrm{L}} }(t, k, \xi)  + \nu(k^2 + \xi^2) \widehat{\omega ^{\mathrm{L}} }(t, k, \xi)   = ik \widehat{\theta ^{\mathrm{L}}}.
		\end{aligned}\end{equation*}
		Note that (when $\nu=\mu$)
		\small{\begin{align*}
				&\big( ik  \widehat{\theta ^{\mathrm{L}}} \big| 2p | \sqrt{\mathcal{M}_1'(k, \xi)} e^{c_0 \nu^\frac13 |k|^\frac23 t} \widehat{\omega^{\mathrm{L}} }|^{2p-2} \mathcal{M}_1'(k, \xi) e^{2c_0 \nu^\frac13 |k|^\frac23 t} \widehat{\omega^{\mathrm{L}} } \big)\\
				=& \big( |k|^{\frac{p+1}{3p}} \sqrt{\mathcal{M}_1'(k, \xi) } e^{c_0 \mu^\frac13 |k|^\frac23 t}  \widehat{\theta ^{\mathrm{L}}} \big| |k|^\frac{2p-1}{3p}| \sqrt{\mathcal{M}_1'(k, \xi)} e^{c_0 \nu^\frac13 |k|^\frac23 t} \widehat{\omega^{\mathrm{L}} }|^{2p-2} \sqrt{\mathcal{M}_1'(k, \xi) }  e^{c_0 \nu^\frac13 |k|^\frac23 t} \widehat{\omega^{\mathrm{L}} } \big)\\
				\leq & \| \langle|k|^\frac13 \rangle|k|^\frac{1}{3p}  \sqrt{\mathcal{M}_1'(k, \xi)} e^{c_0 \mu^\frac13 |k|^\frac23 t} \widehat{\theta^{\mathrm{L}} }\|_{L^{2p}}  \| |k|^\frac{1}{3p}  \sqrt{\mathcal{M}_1'(k, \xi)} e^{c_0 \nu^\frac13 |k|^\frac23 t} \widehat{\omega^{\mathrm{L}} }\|_{L^{2p}}^{2p-1}.
		\end{align*}}
		Similarly, as \eqref{eq:est of theta L} we get 
		\begin{align*}
			\frac{d}{dt} \|  \sqrt{\mathcal{M}_1'(k, \xi)} e^{c_0 \nu^\frac13 |k|^\frac23 t} \widehat{\omega^{\mathrm{L}} } \|_{L^{2p}}^{2p}  +  \nu^\frac13  \| |k|^\frac{1}{3p} e^{c_0 \nu^\frac13 |k|^\frac23 t} \widehat{\omega^{\mathrm{L}} }\|_{L^{2p}}^{2p} 
			\lesssim \nu^{-\frac{2p-1}{3}} \| \langle|k|^\frac13 \rangle|k|^\frac{1}{3p}  e^{c_0 \mu^\frac13 |k|^\frac23 t} \widehat{\theta^{\mathrm{L}} }\|_{L^{2p}}^{2p},
		\end{align*}
		provided with $c_0 <  \frac{m^{p-1}}{8(m+\pi)^p}$. Integrating the above inequality and combining with \eqref{eq:est of linear theta L2}, one has
		\begin{equation}\begin{aligned} \label{eq:est of linear omega L2}
				\|  e^{c_0 \nu^\frac13 |k|^\frac23 t} \langle k, \xi + kt\rangle^6 & \langle \frac{1}{k}\rangle^4 \widehat{\omega^{\mathrm{L}}} \|_{L^\infty L^{2p}}^{2p} 
				+ \nu^\frac13 \| e^{ c_0 \nu^\frac13 |k|^\frac23 t}    \langle k, \xi + kt\rangle^6 \langle \frac{1}{k}\rangle^4 |k|^\frac{1}{3p} \widehat{\omega^{\mathrm{L}}} \|_{L^{2p} L^{2p}}^{2p} \\
				&\lesssim \|  \langle k, \xi\rangle^6 \langle \frac{1}{k}\rangle^4  \widehat{\omega_{\rm in}} \|_{L^{2p}}^{2p}  + \nu^{-\frac{2p-1}{3}} \|  \langle k, \xi\rangle^6 \langle \frac{1}{k}\rangle^4 \langle |k|^\frac13 \rangle \widehat{\theta_{\rm in}} \|_{L^{2p}}^{2p}.
		\end{aligned}\end{equation}
		By the interpolation inequality and the Hausdorff-Young inequality, it holds 
		\begin{equation}\begin{aligned} \label{eq:ini omega}
				\|  \langle k, \xi\rangle^6 \langle \frac{1}{k}\rangle^4  \widehat{\omega_{\rm in}} \|_{L^{2p}} &\leq  \|  \langle k, \xi\rangle^6 \langle \frac{1}{k}\rangle^4  \widehat{\omega_{\rm in}} \|_{L^{2}}^\frac1p \|  \langle k, \xi\rangle^6 \langle \frac{1}{k}\rangle^4  \widehat{\omega_{\rm in}} \|_{L^{\infty}}^\frac{p-1}{p} \\
				& \leq  \|  \langle D_x, D_y\rangle^6 \langle \frac{1}{D_x}\rangle^4 \omega_{\rm in} \|_{L^2}^\frac1p \|  \langle D_x, D_y\rangle^6 \langle \frac{1}{D_x}\rangle^4 \omega_{\rm in} \|_{L^1}^\frac{p-1}{p} \leq \nu^{\frac13 + \delta}
		\end{aligned}\end{equation}
		and similarly
		\begin{equation}\begin{aligned} \label{eq:ini theta}
				\|  \langle k, \xi\rangle^6 \langle \frac{1}{k}\rangle^4 \langle |k|^\frac13 \rangle  \widehat{\theta_{\rm in}} \|_{L^{2p}} \leq \nu^{\frac23 + 2\delta}.
		\end{aligned}\end{equation}
		Combining \eqref{eq:est of linear theta L2}, \eqref{eq:est of linear omega L2}, \eqref{eq:ini omega} and \eqref{eq:ini theta}, we complete the proof.
	\end{proof}
	\begin{Rem}
		In Lemma~\ref{lem:linear}, we cannot let \( p \to \infty \), as this would cause the constant \( c \) to vanish for any fixed \( m > 0 \). This issue arises from the limitation in the construction of the multiplier \( \mathcal{M}_1 \). However, in the subsequent analysis, only the case \( p = 1 \) is required. In this setting, we  choose \( m = \pi \), which yields the following Proposition \ref{lem:linear theta}.
	\end{Rem}

	\begin{proof}[Proof of Proposition \ref{lem:linear theta}]
		The proof is similar to Lemma \ref{lem:linear} and we omit it.
	\end{proof}

	\section{Energy estimates in short time scale $t\leq T_0 = \nu^{-\frac16}$} \label{sec.4}

	At the beginning of this section, due to the lack of reverse Hausdorff–Young inequalities, we provide the following estimate for the linear part.

	\begin{Lem}
		For $t\leq T_0 =\nu^{-\frac16}$ and $1 < p < 2$, if the initial data $(\omega_{\mathrm{in}}, \theta_{\mathrm{in}})$ satisfies $(\ref{eq:the initial data})$, there holds 
		\begin{equation*}\begin{aligned}
				\|\langle D_x, D_y+t D_x\rangle^2 \omega^{\mathrm{L}}\|_{L^\infty L^p} + \nu^{-\frac13 - \delta} \|\langle D_x, D_y+t D_x\rangle^2 \langle D_x\rangle \theta^{\mathrm{L}}\|_{ L^\infty L^p}  \lesssim \nu^{\frac13 +\delta}.
		\end{aligned}\end{equation*}
	\end{Lem}
	\begin{proof}
		Applying $\langle D_x, D_y+t D_x\rangle^2\langle D_{x}\rangle$ to $\eqref{eq:linear equ}_2$ and taking the inner product of the resulting equation with $\left|\langle D_x, D_y+t D_x\rangle^2 \langle D_x\rangle  \theta^{\mathrm{L}}\right|^{p-2}\langle D_x, D_y+t D_x\rangle^2 \langle D_x\rangle \theta^{\mathrm{L}}$, we find that
		$$
		\begin{aligned}
			\frac{d}{d t}\|\langle D_x, D_y+t D_x\rangle^2 \langle D_x\rangle \theta^{\mathrm{L}}\|_{L^p} \leq 0,
		\end{aligned}
		$$
		which implies
		\begin{equation*}\begin{aligned}
				\|\langle D_x, D_y+t D_x\rangle^2 \langle D_x\rangle \theta^{\mathrm{L}}\|_{ L^\infty L^p} \leq \|\langle D_x, D_y\rangle^2\langle D_x\rangle\theta_{in}\|_{L^p} \leq \nu^{\frac{2}{3} + 2\delta}.
		\end{aligned}\end{equation*}
		Similarly, we have 
		\begin{equation*}\begin{aligned}
				\frac{d}{d t}\|\langle D_x, D_y+t D_x\rangle^2 \omega^{\mathrm{L}}\|_{L^p} \lesssim \| \langle D_x, D_y+t D_x\rangle^2 \langle D_x\rangle \theta^{\mathrm{L}} \|_{L^p},
		\end{aligned}\end{equation*}
		from which, we deduce for $t\leq \nu^{-\frac16}$ that
		\begin{equation*}\begin{aligned}
				\|\langle D_x, D_y+t D_x\rangle^2 \omega^{\mathrm{L}}\|_{L^\infty L^p} \lesssim  \|\langle D_x, D_y\rangle^2 \omega_{in}\|_{ L^p} + \nu^{-\frac16} \| \langle D_x, D_y+t D_x\rangle^2 \langle D_x\rangle \theta^{\mathrm{L}} \|_{L^\infty L^p} \lesssim \nu^{\frac13 +\delta}.
		\end{aligned}\end{equation*}
	\end{proof}
	Denote that $\omega^{\mathrm{NL}}=\omega-\omega^{\mathrm{L}}$ and $\theta^{\mathrm{NL}}=\theta-\theta^{\mathrm{L}}$, which solve:
	\begin{align} \label{eq:nonlinear part}
		\left\{\begin{array}{l}
			\partial_t \omega^{\mathrm{NL}}-\nu \Delta \omega^{\mathrm{NL}}+y \partial_x \omega^{\mathrm{NL}}+\left(u^{\mathrm{L}}+u^{\mathrm{NL}}\right) \cdot \nabla\left(\omega^{\mathrm{L}}+\omega^{\mathrm{NL}}\right)=\px \theta^{\mathrm{NL}},  \\
			\partial_t \theta^{\mathrm{NL}}-\nu \Delta \theta^{\mathrm{NL}}+y \partial_x \theta^{\mathrm{NL}}+\left(u^{\mathrm{L}}+u^{\mathrm{NL}}\right) \cdot \nabla\left(\theta^{\mathrm{L}}+\theta^{\mathrm{NL}}\right)=0, \\
			u^{\mathrm{L}}=\left(-\partial_y, \partial_x\right)(-\Delta)^{-1} \omega^{\mathrm{L}}, \quad u^{\mathrm{NL}}=\left(-\partial_y, \partial_x\right)(-\Delta)^{-1} \omega^{\mathrm{NL}}, \\
			\left.\omega^{\mathrm{NL}}\right|_{t=0}=0,\quad \left.\theta^{\mathrm{NL}}\right|_{t=0}=0.
		\end{array}\right.
	\end{align}

	Note that, in short time scale $t\leq T_0 = \nu^{-\frac16}$, the  Fourier multipliers is not necessary  for the energy estimates, since the upper bound of $t$ is sufficient for our stability threshold, i.e.,  
	\begin{equation*}
		\begin{aligned}
			\|\langle D_x, D_y+t D_x\rangle^3 \omega^{\mathrm{NL}}\|_{L^\infty L^2}  \lesssim  \  \nu^{-\frac16} \|\langle D_x, D_y+t D_x\rangle^3 \px \theta^{\mathrm{NL}}\|_{L^\infty L^2} + \text{``other terms''.}
		\end{aligned}
	\end{equation*}
    Next we prove it in details.
	
	\subsection{$L^2$-estimate.}
	\begin{Lem} \label{eq:est of theta NL short time}
		For $1<p<2$, it holds
		$$
		\begin{aligned}
			& \frac{d}{d t}\|\langle D_x, D_y+t D_x\rangle^3 \theta^{\mathrm{NL}}\|_{L^2}^2 + \nu \|\langle D_x, D_y+t D_x\rangle^3 \nabla \theta^{\mathrm{NL}}\|_{L^2}^2 \\
			&\lesssim\|\langle D_x, D_y+t D_x\rangle^3 \theta^{\mathrm{NL}}\|_{L^2} \\
			&\quad \times  \left[\langle t\rangle^{-2} \|\langle D_x, D_y+t D_x\rangle^6\langle\frac{1}{D_x}\rangle^4 \omega^{\mathrm{L}}\|_{L^2}
			\|\langle D_x, D_y+t D_x\rangle^6\langle\frac{1}{D_x}\rangle^4 \theta^{\mathrm{L}}\|_{L^2} \rt.\\
			&\lt.\quad \quad +\langle t\rangle^{1+\delta}  \left(\|\omega^{\mathrm{NL}}\|_{L^p}+\|\langle D_x, D_y+t D_x\rangle^3 \omega^{\mathrm{NL}}\|_{L^2}\right) 
			\|\langle D_x, D_y+t D_x\rangle^6\langle\frac{1}{D_x}\rangle^4 \theta^{\mathrm{L}}\|_{L^2} \rt.  \\
			& \left.\quad \quad  + \langle t\rangle^{1+\delta}\|\langle D_x, D_y+t D_x\rangle^3 \theta^{\mathrm{NL}}\|_{L^2} \rt. \\
			&\lt.   \qquad \quad \times \left(\|\langle D_x, D_y+t D_x\rangle^6\langle\frac{1}{D_x}\rangle^4 \omega^{\mathrm{L}}\|_{L^2}  + \|\langle D_x, D_y+t D_x\rangle^3 \omega^{\mathrm{NL}}\|_{L^2} \right) \right].
		\end{aligned}
		$$
	\end{Lem}
	\begin{proof}
		Taking the inner product of $\eqref{eq:nonlinear part}_2$ with $\langle D_x, D_y+t D_x\rangle^6  \theta^{\mathrm{NL}}$, we get that
		\begin{equation}\begin{aligned} \label{eq:energy est of non theta}
				&  \frac12 \frac{d}{ d t}\|\langle D_x, D_y+t D_x\rangle^3 \theta^{\mathrm{NL}}\|_{L^2}^2+\nu\|\langle D_x, D_y+t D_x\rangle^3 \nabla \theta^{\mathrm{NL}}\|_{L^2}^2 \\
				& \quad=-\Re\big(\langle D_x, D_y+t D_x\rangle^3\left[(u^{\mathrm{L}}+u^{\mathrm{NL}}) \cdot \nabla(\theta^{\mathrm{L}}  +  \theta^{\mathrm{NL}})\right]  \big| \langle D_x, D_y+t D_x\rangle^3 \theta^{\mathrm{NL}}\big) .
		\end{aligned}\end{equation}
		
		\underline{\bf Estimate of $u^{\mathrm{L}} \cdot \nabla \theta^{\mathrm{L}}$.}
		First, for the source term $u^{\mathrm{L}} \cdot \nabla \theta^{\mathrm{L}}$, we write by the Fourier transform that
		\begin{equation*}\begin{aligned}
				&\left|\big(\langle D_x, D_y+t D_x\rangle^3(u^{\mathrm{L}} \cdot \nabla \theta^{\mathrm{L}})  \big| \langle D_x, D_y+t D_x\rangle^3 \theta^{\mathrm{NL}}\big)\right|  \\
				&=\left|\int_{\mathbb{R}^4}\langle k, \xi+k t\rangle^6 \hat{\theta}_k^{\mathrm{NL}}(\xi) \frac{\eta(k-l)-l(\xi-\eta)}{|l|^2+|\eta|^2} \hat{\omega}_l^{\mathrm{L}}(\eta) \hat{\theta}_{k-l}^{\mathrm{L}}(\xi-\eta) d k d l d \xi d \eta\right|.
		\end{aligned}\end{equation*}
		Using
		\begin{equation}
			\begin{aligned} \label{eq:0}
				\abs{\eta(k-l)-l(\xi-\eta)} & = \abs{  (\eta+l t)(k-l)-l(\xi-\eta+(k-l) t)  } \\
				& \leq\langle l, \eta+l t\rangle\langle k-l, \xi-\eta+(k-l) t\rangle,
			\end{aligned}
		\end{equation}
		the point-wise inviscid damping estimate
		\begin{equation} \label{eq:the point-wise inviscid damping estimate}
			\frac{1}{|l|^2+|\eta|^2} \lesssim \frac{\langle l, \eta+l t\rangle^2}{|l|^4\langle t\rangle^2} \lesssim\langle t\rangle^{-2}\langle\frac{1}{l}\rangle^4\langle l, \eta+l t\rangle^2,
		\end{equation}
		and
		\begin{equation} \label{eq:temp.1}
			\langle k, \xi+k t\rangle^3 \lesssim\langle l, \eta+l t\rangle^3\langle k-l, \xi-\eta+(k-l) t\rangle^3
		\end{equation}
		one has 
		\begin{equation}\begin{aligned} \label{eq:est of source term}
				& \quad\left|\left(\langle D_x, D_y+t D_x\rangle^3\left(u^{\mathrm{L}} \cdot \nabla \theta^{\mathrm{L}}\right) \mid\langle D_x, D_y+t D_x\rangle^3 \theta^{\mathrm{NL}}\right)\right| \\
				& \lesssim\langle t\rangle^{-2}\|\langle k, \xi+k t\rangle^3 \hat{\theta}_k^{\mathrm{NL}}(\xi)\|_{L_{k, \xi}^2}\|\langle l, \eta+l t\rangle^6\langle\frac{1}{l}\rangle^4 \hat{\omega}_l^{\mathrm{L}}(\eta)\|_{L_{l, \eta}^2} \\
				& \quad \times\|\langle k-l, \xi-\eta+(k-l) t\rangle^{-2}\|_{L_{k-l, \xi - \eta}^2}\|\langle k-l, \xi-\eta+(k-l) t\rangle^6 \hat{\theta}_{k-l}^{\mathrm{L}}(\xi-\eta)\|_{L_{k-l, \xi - \eta}^2} \\
				& \lesssim\langle t\rangle^{-2}\|\langle D_x, D_y+t D_x\rangle^3 \theta^{\mathrm{NL}}\|_{L^2}\|\langle D_x, D_y+t D_x\rangle^6\langle\frac{1}{D_x}\rangle^4 \omega^{\mathrm{L}}\|_{L^2}\|\langle D_x, D_y+t D_x\rangle^6 \theta^{\mathrm{L}}\|_{L^2} .
		\end{aligned}\end{equation}

		\underline{\bf Estimate of $u^{\mathrm{NL}} \cdot \nabla \theta^{\mathrm{L}}$.}
		Applying the Fourier transform to the reaction term $u^{\mathrm{NL}} \cdot \nabla \theta^{\mathrm{L}}$, we have
		$$
		\begin{aligned}
			& \left|\left(\langle D_x, D_y+t D_x\rangle^3(u^{\mathrm{NL}} \cdot \nabla \theta^{\mathrm{L}}) \big| \langle D_x, D_y+t D_x\rangle^3 \theta^{\mathrm{NL}}\right)\right| \\
			= & \left|\int_{\mathbb{R}^4}\langle k, \xi+k t\rangle^6 \hat{\theta}_k^{\mathrm{NL}}(\xi) \frac{\eta(k-l)-l(\xi-\eta)}{|l|^2+|\eta|^2} \hat{\omega}_l^{\mathrm{NL}}(\eta) \hat{\theta}_{k-l}^{\mathrm{L}}(\xi-\eta) d k d l d \xi d \eta\right|.
		\end{aligned}
		$$
		Then, we divide the estimate into two cases.\\
		$\bullet$~Case 1: $|l|^2+|\eta|^2 \geq 1$. In this case, one can use \eqref{eq:temp.1} and
		\begin{align} \label{eq:est of k-l + xi - eta}
			|k-l|+|\xi-\eta| \leq|k-l|+|\xi-\eta + (k-l) t|+|k-l| t \lesssim\langle t\rangle\langle k-l, \xi-\eta+(k-l) t\rangle
		\end{align}
		to prove
		$$
		\begin{aligned}
			& \quad\left|\int_{|l|^2+|\eta|^2 \geq 1}\langle k, \xi+k t\rangle^6 \hat{\theta}_k^{\mathrm{NL}}(\xi) \frac{\eta(k-l)-l(\xi-\eta)}{|l|^2+|\eta|^2} \hat{\omega}_l^{\mathrm{NL}}(\eta) \hat{\theta}_{k-l}^{\mathrm{L}}(\xi-\eta) d k d l d \xi d \eta\right| \\
			& \lesssim\|\langle k, \xi+k t\rangle^3 \hat{\theta}_k^{\mathrm{NL}}(\xi)\|_{L_{k, \xi}^2}\|\langle l, \eta+l t\rangle^3 \hat{\omega}_l^{\mathrm{NL}}(\eta)\|_{L_{l, \eta}^2}\langle t\rangle\|\langle k-l, \xi-\eta+(k-l) t\rangle^{-2}\|_{L_{k-l, \xi-\eta}^2} \\
			& \quad \times\|\langle k-l, \xi-\eta+(k-l) t\rangle^6 \hat{\theta}_{k-l}^{\mathrm{L}}(\xi-\eta)\|_{L_{k-l, \xi-\eta}^2} \\
			& \lesssim\langle t\rangle   \|\langle D_x, D_y+t D_x\rangle^3 \theta^{\mathrm{NL}}\|_{L^2}  \|\langle D_x, D_y+t D_x\rangle^3 \omega^{\mathrm{NL}}\|_{L^2}     \|\langle D_x, D_y+t D_x\rangle^6 \theta^{\mathrm{L}}\|_{L^2}.
		\end{aligned}
		$$
		$\bullet$~Case 2: $|l|^2+|\eta|^2<1$. To overcome the singularity near $\{l, \eta\}=\{0,0\}$, we make use of \begin{align} \label{eq:est of l+eta l t}
			|l|+|\eta+l t| \leq  \langle l, \eta+l t\rangle^{1-\delta}\langle t\rangle^\delta(|l|+|\eta|)^\delta
		\end{align}
		to rewrite \eqref{eq:temp.1} as 
		$$
		\langle k, \xi+k t\rangle^3 \lesssim\langle k-l, \xi-\eta+(k-l) t\rangle^3 +  \langle l, \eta+l t\rangle^{3-\delta}\langle t\rangle^\delta(|l|+|\eta|)^\delta
		$$
		which, combined with \eqref{eq:est of k-l + xi - eta}, implies
		\begin{equation} \label{eq:temp0.1}
		\begin{aligned}
			&\left|\int_{|l|^2+|\eta|^2<1}\langle k, \xi+k t\rangle^6 \hat{\theta}_k^{\mathrm{NL}}(\xi) \frac{\eta(k-l)-l(\xi-\eta)}{|l|^2+|\eta|^2} \hat{\omega}_l^{\mathrm{NL}}(\eta) \hat{\theta}_{k-l}^{\mathrm{L}}(\xi-\eta) d k d l d \xi d \eta\right| \\
			& \lesssim\langle t\rangle\|\langle k, \xi+k t\rangle^3 \hat{\theta}_k^{\mathrm{NL}}(\xi)\|_{L_{k, \xi}^2}\left[\|\hat{\omega}_l^{\mathrm{NL}}(\eta)\|_{L_{l, \eta}^{p^{\prime}}}\|(|l|+|\eta|)^{-1}  \left. \chi\right|_{|l|^2+|\eta|^2<1}\|_{L_{l, \eta}^p}\right. \\
			& \quad \times\|\langle k-l, \xi-\eta+(k-l) t\rangle^4 \hat{\theta}_{k-l}^{\mathrm{L}}(\xi-\eta)\|_{L_{k-l, \xi-\eta}^2}+\langle t\rangle^\delta\|\langle l, \eta+l t\rangle^{3-\delta} \hat{\omega}_l^{\mathrm{NL}}(\eta)\|_{L_{l, \eta}^2} \\
			& \quad \left.\times\|(|l|+|\eta|)^{\delta-1} \left. \chi\right|_{|l|^2+|\eta|^2<1}\|_{L_{l, \eta}^2}\|\langle k-l, \xi-\eta+(k-l) t\rangle \hat{\theta}_{k-l}^{\mathrm{L}}(\xi-\eta)\|_{L_{k-l, \xi-\eta}^2}\right] \\
			& \lesssim\langle t\rangle^{1+\delta}\|\langle D_x, D_y+t D_x\rangle^3 \theta^{\mathrm{NL}}\|_{L^2} \|\langle D_x, D_y+t D_x\rangle^4 \theta^{\mathrm{L}}\|_{L^2}\\
			& \quad \times \left(\|\omega^{\mathrm{NL}}\|_{L^p}+\|\langle D_x, D_y+t D_x\rangle^3 \omega^{\mathrm{NL}}\|_{L^2}\right).
		\end{aligned}   
		\end{equation}
		Combining the two cases above, we arrive at
		\begin{equation}\begin{aligned} \label{eq:est of reaction term}
				& \left|\left(\langle D_x, D_y+t D_x\rangle^3\left(u^{\mathrm{NL}} \cdot \nabla \theta^{\mathrm{L}}\right) \mid\langle D_x, D_y+t D_x\rangle^3 \theta^{\mathrm{NL}}\right)\right| \\
				& \lesssim\langle t\rangle^{1+\delta}\|\langle D_x, D_y+t D_x\rangle^3 \theta^{\mathrm{NL}}\|_{L^2} \|\langle D_x, D_y+t D_x\rangle^6 \theta^{\mathrm{L}}\|_{L^2}\\
				& \quad \times \left(\|\omega^{\mathrm{NL}}\|_{L^p}+\|\langle D_x, D_y+t D_x\rangle^3 \omega^{\mathrm{NL}}\|_{L^2}\right).
		\end{aligned}\end{equation}
		
		\underline{\bf Estimate of $u \cdot \nabla \theta^{\mathrm{NL}}$.} For this term, we employ the classical commutator estimate:
		\begin{align} \label{eq:commutator estimate}
			|\langle k, \xi+k t\rangle^3-\langle k-l, &\xi-\eta+(k-l) t\rangle^3| \no \\
			&\lesssim(|l|+|\eta+l t|)\left(\langle l, \eta+l t\rangle^2+\langle k-l, \xi-\eta+(k-l) t\rangle^2\right).
		\end{align}
		Thanks to
		$$
		\left(u \cdot \nabla\langle D_x, D_y+t D_x\rangle^3 \theta^{\mathrm{NL}}  \big| \langle D_x, D_y+t D_x\rangle^3 \theta^{\mathrm{NL}}\right)=0,
		$$
		we get that
		$$
		\begin{aligned}
			& \left|\left(\langle D_x, D_y+t D_x\rangle^3(u \cdot \nabla \theta^{\mathrm{NL}}) \big| \langle D_x, D_y+t D_x\rangle^3 \theta^{\mathrm{NL}}\right)\right| \\
			= & \lt| \int_{\mathbb{R}^4}\langle k, \xi+k t\rangle^3 \hat{\theta}_k^{\mathrm{NL}}(\xi)\left(\langle k, \xi+k t\rangle^3-\langle k-l, \xi-\eta+(k-l) t\rangle^3\right) \rt. \\
			& \left.\times \frac{\eta(k-l)-l(\xi-\eta)}{|l|^2+|\eta|^2} \hat{\omega}_l(\eta) \hat{\theta}_{k-l}^{\mathrm{NL}}(\xi-\eta) d k d l d \xi d \eta \right| .
		\end{aligned}
		$$
		We then divide the estimate into two cases as above.\\
		$\bullet$~Case 1: $|l|^2+|\eta|^2 \geq 1$.
		By \eqref{eq:est of k-l + xi - eta} and \eqref{eq:commutator estimate}, it holds that
		\begin{equation}
			\begin{aligned} \label{eq:temp.2}
				& \lt|\int_{|l|+|\eta| \geq 1}\langle k, \xi+k t\rangle^3 \hat{\theta}_k^{\mathrm{NL}}(\xi)\left(\langle k, \xi+k t\rangle^3-\langle k-l, \xi-\eta+(k-l) t\rangle^3\right) \rt. \\
				& \left.\quad \times \frac{\eta(k-l)-l(\xi-\eta)}{|l|^2+|\eta|^2} \hat{\omega}_l(\eta) \hat{\theta}_{k-l}^{\mathrm{NL}}(\xi-\eta) d k d l d \xi d \eta \right| \\
				& \lesssim\|\langle k, \xi+k t\rangle^3 \hat{\theta}_k^{\mathrm{NL}}(\xi)\|_{L_{k, t}^2}\|\langle l, \eta+l t\rangle^{3} \hat{\omega}_l(\eta)\|_{L_{l, \eta}^2} \\
				& \quad \times\left(\|\langle l, \eta+l t\rangle^{-2}\|_{L_{l, \eta}^2}+\|\langle k-l, \xi-\eta+(k-l) t\rangle^{-2}\|_{L_{k-l, \xi-\eta}^2}\right) \\
				& \quad \times\langle t\rangle\|\langle k-l, \xi-\eta+(k-l) t\rangle^{3} \hat{\theta}_{k-l}^{\mathrm{NL}}(\xi-\eta)\|_{L_{k-l, \xi-\eta}^2} \\
				& \lesssim\langle t\rangle\|\langle D_x, D_y+t D_x\rangle^3 \theta^{\mathrm{NL}}\|_{L^2}^2\|\langle D_x, D_y+t D_x\rangle^3 \omega\|_{L^2} .
			\end{aligned}
		\end{equation}
		$\bullet$~Case 2: $|l|^2+|\eta|^2<1$. Similarly, to overcome the singularity, we put \eqref{eq:commutator estimate} and \eqref{eq:est of l+eta l t} together to obtain that
		\begin{equation*}\begin{aligned}
				|\langle k, \xi+k t\rangle^3- &\langle k-l, \xi-\eta+(k-l) t\rangle^3|  \\
				& \lesssim \langle l, \eta+l t\rangle^{1-\delta}\langle t\rangle^\delta(|l|+|\eta|)^\delta   \left(\langle l, \eta+l t\rangle^2+\langle k-l, \xi-\eta+(k-l) t\rangle^2\right),
		\end{aligned}\end{equation*}
		which along with \eqref{eq:est of k-l + xi - eta} shows that,
		\begin{equation}
			\begin{aligned} \label{eq:temp.3}
				&  \lt| \int_{|l|+|\eta|<1}\langle k, \xi+k t\rangle^3 \hat{\theta}_k^{\mathrm{NL}}(\xi)\left(\langle k, \xi+k t\rangle^3-\langle k-l, \xi-\eta+(k-l) t\rangle^3\right)\rt. \\
				& \left.\quad \times \frac{\eta(k-l)-l(\xi-\eta)}{|l|^2+|\eta|^2} \hat{\omega}_l(\eta) \hat{\theta}_{k-l}^{\mathrm{NL}}(\xi-\eta) d k d l d \xi d \eta \right| \\
				& \lesssim\|\langle k, \xi+k t\rangle^3 \hat{\theta}_k^{\mathrm{NL}}(\xi)\|_{L_{k, \xi}^2}   \langle t\rangle^\delta      \|( | l | + | \eta | ) ^ { \delta - 1 } \lt. \chi \rt|_{ | l|+|\eta|<1} \|_{L_{l, \eta}^2} \lt(   \|\langle l, \eta+l t\rangle^{3-\delta} \hat{\omega}_l(\eta)\|_{L_{l, \eta}^2}   \rt. \\
				& \lt. \quad \times\langle t\rangle\|\langle k-l, \xi-\eta+(k-l) t\rangle \hat{\theta}_{k-l}^{\mathrm{NL}}(\xi-\eta)\|_{L_{k-l, \xi-\eta}^2}  +  \|\langle l, \eta+l t\rangle^{1-\delta} \hat{\omega}_l(\eta)\|_{L_{l, \eta}^2} \rt.\\
				&\quad \lt.  \times\langle t\rangle\|\langle k-l, \xi-\eta+(k-l) t\rangle^3 \hat{\theta}_{k-l}^{\mathrm{NL}}(\xi-\eta)\|_{L_{k-l, \xi-\eta}^2} \rt)\\
				& \lesssim\langle t\rangle^{1+\delta}\|\langle D_x, D_y+t D_x\rangle^3 \theta^{\mathrm{NL}}\|_{L^2}^2\|\langle D_x, D_y+t D_x\rangle^3 \omega\|_{L^2}.
			\end{aligned}
		\end{equation}
		
		The combination of \eqref{eq:temp.2} and \eqref{eq:temp.3} yields
		\begin{equation}\begin{aligned} \label{eq:est of commutator term}
				& \left|\left(\langle D_x, D_y+t D_x\rangle^3\left(u \cdot \nabla \theta^{\mathrm{NL}}\right) \mid\langle D_x, D_y+t D_x\rangle^3 \theta^{\mathrm{NL}}\right)\right| \\
				\lesssim & \langle t\rangle^{1+\delta}\|\langle D_x, D_y+t D_x\rangle^3 \theta^{\mathrm{NL}}\|_{L^2}^2\|\langle D_x, D_y+t D_x\rangle^3 \omega\|_{L^2}.
		\end{aligned}\end{equation}

		Summing up \eqref{eq:energy est of non theta}, \eqref{eq:est of source term}, \eqref{eq:est of reaction term} and  \eqref{eq:est of commutator term}, we finish the proof.
	\end{proof}
	
	For $\omega^{\mathrm{NL}}$, a similar argument as in the above lemma yields the following.
	\begin{Lem} \label{lem:est of omega L2}
		For $1<p<2$, it holds that
		\begin{equation*}
			\begin{aligned}
				&\frac{d}{d t}\|\langle D_x, D_y +t D_x\rangle^3 \omega^{\mathrm{NL}}\|_{L^2} \\
				&\lesssim\langle t\rangle^{-2} \|\langle D_x, D_y+t D_x\rangle^6\langle\frac{1}{D_x}\rangle^4 \omega^{\mathrm{L}}\|_{L^2}^2    +     \|\langle D_x, D_y+t D_x\rangle^3 \px \theta^{\mathrm{NL}}\|_{L^2}    \\
				&\quad +\langle t\rangle^{1+\delta}( \|\langle D_x, D_y+t D_x\rangle^6\langle\frac{1}{D_x}\rangle^4 \omega^{\mathrm{L}}\|_{L^2} +\|\langle D_x, D_y+t D_x\rangle^3 \omega^{\mathrm{NL}}\|_{L^2}) \\
				&\quad \quad \times(\|\omega^{\mathrm{NL}}\|_{L^p}+\|\langle D_x, D_y+t D_x\rangle^3 \omega^{\mathrm{NL}}\|_{L^2}) .
			\end{aligned}
		\end{equation*}
	\end{Lem}

	For $\px \theta^{\mathrm{NL}}$, the most troublesome term is $\px u \cdot \nabla \theta^{\mathrm{NL}}$, due to the presence of an additional derivative falling on $u$ and the lack of derivative on $\theta^{\mathrm{NL}}$. However, for the short-time scale estimate, it suffices to work with 
	$
	\nu^{1/2 }\|\langle D_x, D_y+t D_x\rangle^3 \nabla \theta^{\mathrm{NL}}\|_{L^2}.
	$
	\begin{Lem} \label{lem:est of pxtheta L2}
		For $1<p<2$, it holds that
		\begin{equation*}\begin{aligned}
				&\frac{d}{d t}\|\langle D_x, D_y+t D_x\rangle^3 \px \theta^{\mathrm{NL}}\|_{L^2}^2+\nu\|\langle D_x, D_y+t D_x\rangle^3 \nabla \px \theta^{\mathrm{NL}}\|_{L^2}^2 \\
				&\lesssim \|\langle D_x, D_y+t D_x\rangle^3 \px \theta^{\mathrm{NL}}\|_{L^2} \bigg[ \|\langle D_x, D_y+t D_x\rangle^6  \langle\frac{1}{l}\rangle^4  \omega^{\mathrm{L}}\|_{L^2} \\
				& \quad  \times \lt( \langle t \rangle^{-2} \|\langle D_x, D_y+t D_x\rangle^6  \langle\frac{1}{l}\rangle^4 \langle \px\rangle \theta^{\mathrm{L}}\|_{L^2}  + \langle t \rangle^{1+ \delta} \|\langle D_x, D_y+t D_x\rangle^3 \px \theta^{\mathrm{NL}}\|_{L^2}  \rt.\\
				&\lt. \quad \quad +   \|\langle D_x, D_y+t D_x\rangle^3 \theta^{\mathrm{NL}}\|_{L^2}  +  \langle t \rangle^{-1} \|\langle D_x, D_y+t D_x\rangle^3 \nabla \theta^{\mathrm{NL}}\|_{L^2} \rt)   \\
				& \quad + \langle t \rangle^{1+ \delta}  \|\langle D_x, D_y+t D_x\rangle^6  \langle\frac{1}{l}\rangle^4 \langle \px\rangle \theta^{\mathrm{L}}\|_{L^2}  \lt(  \|\omega^{\mathrm{NL}}\|_{L^p}   + \|\langle D_x, D_y+t D_x\rangle^3 \omega^{\mathrm{NL}}\|_{L^2} \rt)  \\
				&\quad  +   \|\langle D_x, D_y+t D_x\rangle^3 \omega^{\mathrm{NL}}\|_{L^2} \lt( \langle t \rangle^{1+ \delta}  \|\langle D_x, D_y+t D_x\rangle^3 \px \theta^{\mathrm{NL}}\|_{L^2}   \rt. \\
				&\lt. \qquad +  \langle t \rangle   \|\langle D_x, D_y+t D_x\rangle^3 \theta^{\mathrm{NL}}\|_{L^2} +  \|\langle D_x, D_y+t D_x\rangle^3 \nabla \theta^{\mathrm{NL}}\|_{L^2} \rt)\bigg] .
		\end{aligned}\end{equation*}
	\end{Lem}
	\begin{proof}
		Taking the inner product of $\eqref{eq:nonlinear part}_2$ with $\langle D_x, D_y+t D_x\rangle^6 \px^2 \theta^{\mathrm{NL}}$ gives that
		\begin{equation}\begin{aligned} \label{eq:1}
				&  \frac12 \frac{d}{ d t}\|\langle D_x, D_y+t D_x\rangle^3 \px \theta^{\mathrm{NL}}\|_{L^2}^2+\nu\|\langle D_x, D_y+t D_x\rangle^3 \nabla \px \theta^{\mathrm{NL}}\|_{L^2}^2 \\
				& \quad=-\Re\left(\langle D_x, D_y+t D_x\rangle^3 \px \left[\left(u^{\mathrm{L}}+u^{\mathrm{NL}}\right) \cdot \nabla\left(\theta^{\mathrm{L}}  +  \theta^{\mathrm{NL}}\right)\right] \mid\langle D_x, D_y+t D_x\rangle^3 \px \theta^{\mathrm{NL}}\right) .
		\end{aligned}\end{equation}

		\underline{\bf Estimate of $u \cdot \nabla \px \theta$.} Adapting the argument in Lemma \ref{eq:est of theta NL short time}, we derive
		\begin{equation}
			\begin{aligned} \label{eq:2}
				& \left|\left(\langle D_x, D_y+t D_x\rangle^3(u \cdot \nabla \px\theta )  \big| \langle D_x, D_y+t D_x\rangle^3 \px\theta^{\mathrm{NL}}\right)\right| \\
				&\lesssim\|\langle D_x, D_y+t D_x\rangle^3 \px \theta^{\mathrm{NL}}\|_{L^2} \\
				&\quad \times  \left[\langle t\rangle^{-2} \|\langle D_x, D_y+t D_x\rangle^6\langle\frac{1}{D_x}\rangle^4 \omega^{\mathrm{L}}\|_{L^2}
				\|\langle D_x, D_y+t D_x\rangle^6\langle\frac{1}{D_x}\rangle^4 \px \theta^{\mathrm{L}}\|_{L^2} \rt.\\
				&\lt.\quad \quad +\langle t\rangle^{1+\delta}  \left(\|\omega^{\mathrm{NL}}\|_{L^p}+\|\langle D_x, D_y+t D_x\rangle^3 \omega^{\mathrm{NL}}\|_{L^2}\right) 
				\|\langle D_x, D_y+t D_x\rangle^6\langle\frac{1}{D_x}\rangle^4 \px  \theta^{\mathrm{L}}\|_{L^2} \rt.  \\
				& \left.\quad \quad  + \langle t\rangle^{1+\delta}\|\langle D_x, D_y+t D_x\rangle^3 \px \theta^{\mathrm{NL}}\|_{L^2} \rt. \\
				&\lt.   \qquad \quad \times \left(\|\langle D_x, D_y+t D_x\rangle^6\langle\frac{1}{D_x}\rangle^4 \omega^{\mathrm{L}}\|_{L^2}  + \|\langle D_x, D_y+t D_x\rangle^3 \omega^{\mathrm{NL}}\|_{L^2} \right) \right].
			\end{aligned}
		\end{equation}

		\underline{\bf Estimate of $\px u^{\mathrm{L}} \cdot \nabla  \theta^{\mathrm{L}}$.} Thanks to the $\langle \frac{1}{D_x}\rangle$ derivative in the initial data, the additional $\px$ derivative on $u$ can be absorbed.
		Using
		$$
		\frac{|l|}{|l|^2+|\eta|^2}  \lesssim \frac{\langle l, \eta+l t\rangle^2}{|l|^3\langle t\rangle^2} \lesssim\langle t\rangle^{-2}\langle\frac{1}{l}\rangle^3\langle l, \eta+l t\rangle^2,
		$$
		we get similarly as \eqref{eq:est of source term} that
		\begin{align} \label{eq:3}
			&\no \left|\left(\langle D_x, D_y+t D_x\rangle^3( \px u^{\mathrm{L}} \cdot \nabla \theta^{\mathrm{L}}) \big| \langle D_x, D_y+t D_x\rangle^3 \px\theta^{\mathrm{NL}}\right)\right| \\
			&\no \lesssim\langle t\rangle^{-2}\|\langle D_x, D_y+t D_x\rangle^3 \px\theta^{\mathrm{NL}}\|_{L^2}\|\langle D_x, D_y+t D_x\rangle^6\langle\frac{1}{D_x}\rangle^4 \omega^{\mathrm{L}}\|_{L^2} \\
			& \quad \times\|\langle D_x, D_y+t D_x\rangle^6 \theta^{\mathrm{L}}\|_{L^2} .
		\end{align}
		
		\underline{\bf Estimate of $\px u^{\mathrm{NL}} \cdot \nabla \theta^{\mathrm{L}}$.} For $\px u^{\mathrm{NL}}$, the classical Riesz transform theory allows us to bypass any concern about singularities (see \eqref{eq:temp.4}).
		Writing
		$$
		\begin{aligned}
			& \left|\left(\langle D_x, D_y+t D_x\rangle^3( \px u^{\mathrm{NL}} \cdot \nabla \theta^{\mathrm{L}}) \big| \langle D_x, D_y+t D_x\rangle^3 \px \theta^{\mathrm{NL}}\right)\right| \\
			= & \left|\int_{\mathbb{R}^4}\langle k, \xi+k t\rangle^6 k \hat{\theta}_k^{\mathrm{NL}}(\xi) \frac{\eta(k-l)-l(\xi-\eta)}{|l|^2+|\eta|^2} l\hat{\omega}_l^{\mathrm{NL}}(\eta) \hat{\theta}_{k-l}^{\mathrm{L}}(\xi-\eta) d k d l d \xi d \eta\right|
		\end{aligned}
		$$
		and using \eqref{eq:temp.1}, \eqref{eq:est of k-l + xi - eta} and 
		\begin{equation} \label{eq:temp.4}
			\frac{|l|(|\eta| + |l|)}{l^2   +  \eta^2}  \lesssim 1,
		\end{equation}
		one gets
		\begin{equation}\begin{aligned} \label{eq:4}
				& \left|\left(\langle D_x, D_y+t D_x\rangle^3( \px u^{\mathrm{NL}} \cdot \nabla \theta^{\mathrm{L}}) \big| \langle D_x, D_y+t D_x\rangle^3 \px \theta^{\mathrm{NL}}\right)\right| \\
				& \lesssim\|\langle k, \xi+k t\rangle^3 k\hat{\theta}_k^{\mathrm{NL}}(\xi)\|_{L_{k, \xi}^2}\|\langle l, \eta+l t\rangle^3 \hat{\omega}_l^{\mathrm{NL}}(\eta)\|_{L_{l, \eta}^2}\langle t\rangle\|\langle k-l, \xi-\eta+(k-l) t\rangle^{-2}\|_{L_{k-l, \xi-\eta}^2} \\
				& \quad \times\|\langle k-l, \xi-\eta+(k-l) t\rangle^6 \hat{\theta}_{k-l}^{\mathrm{L}}(\xi-\eta)\|_{L_{k-l, \xi-\eta}^2} \\
				& \lesssim\langle t\rangle   \|\langle D_x, D_y+t D_x\rangle^3 \px \theta^{\mathrm{NL}}\|_{L^2}  \|\langle D_x, D_y+t D_x\rangle^3 \omega^{\mathrm{NL}}\|_{L^2}     \|\langle D_x, D_y+t D_x\rangle^6 \theta^{\mathrm{L}}\|_{L^2}.
		\end{aligned}\end{equation}
		
		\underline{\bf Estimate of $\px u^{\mathrm{L}} \cdot \nabla \theta^{\mathrm{NL}}$.} This term is challenging, since one cannot apply commutator estimates and is forced to work with the less favorable term $\nabla \theta^{\mathrm{NL}}$.
		Using
		\begin{align*} \label{eq:temp1.0}
			\langle k, \xi+k t\rangle^3 \lesssim\langle k-l, \xi-\eta+(k-l) t\rangle^3  +  \langle l, \eta+l t\rangle^3
		\end{align*}
		we have
		\begin{equation}\begin{aligned} \label{eq:temp1}
				& \left|\left(\langle D_x, D_y+t D_x\rangle^3( \px u^{\mathrm{L}} \cdot \nabla \theta^{\mathrm{NL}}) \big| \langle D_x, D_y+t D_x\rangle^3 \px \theta^{\mathrm{NL}}\right)\right| \\
				&=  \left|\int_{\mathbb{R}^4}\langle k, \xi+k t\rangle^6 k \hat{\theta}_k^{\mathrm{NL}}(\xi) \frac{\eta(k-l)-l(\xi-\eta)}{|l|^2+|\eta|^2} l\hat{\omega}_l^{\mathrm{L}}(\eta) \hat{\theta}_{k-l}^{\mathrm{NL}}(\xi-\eta) d k d l d \xi d \eta\right|. \\
				&\lesssim  \int_{\mathbb{R}^4}\langle k, \xi+k t\rangle^3 \lt|k \hat{\theta}_k^{\mathrm{NL}}(\xi)  \frac{\eta(k-l)-l(\xi-\eta)}{|l|^2+|\eta|^2} l \rt| \langle l, \eta+l t\rangle^3 \lt| \hat{\omega}_l^{\mathrm{L}}(\eta)   \hat{\theta}_{k-l}^{\mathrm{NL}}(\xi-\eta) \rt| d k d l d \xi d \eta \\
				&\quad + \int_{\mathbb{R}^4}\langle k, \xi+k t\rangle^3 \lt| k \hat{\theta}_k^{\mathrm{NL}}(\xi) \frac{\eta(k-l)-l(\xi-\eta)}{|l|^2+|\eta|^2} l\hat{\omega}_l^{\mathrm{L}}(\eta) \rt| \\
				& \qquad \times  \langle k-l, \xi-\eta+(k-l) t\rangle^3 \lt| \hat{\theta}_{k-l}^{\mathrm{NL}}(\xi-\eta) \rt|  d k d l d \xi d \eta
				=:     I_{k-l} +I_{l}.
		\end{aligned}\end{equation}
		For $I_{k-l}$, by \eqref{eq:est of k-l + xi - eta} and
		\begin{equation}\begin{aligned} \label{eq:temp1.5}
				\frac{|l|(|l| + |\eta|)}{l^2   +  \eta^2}  \lesssim \frac{\langle l, \eta+l t\rangle}{(|l|^2+|\eta|^2)^\frac12} \lesssim \frac{\langle l, \eta+l t\rangle^2}{|l|^2\langle t\rangle} \lesssim\langle t\rangle^{-1}\langle\frac{1}{l}\rangle^2\langle l, \eta+l t\rangle^2,
		\end{aligned}\end{equation}
		we get that
		\begin{equation}\begin{aligned} \label{eq:temp2}
				I_{k-l}^1 \lesssim & \|\langle k, \xi+k t\rangle^3 k\hat{\theta}_k^{\mathrm{NL}}(\xi)\|_{L_{k, \xi}^2}\|\langle l, \eta+l t\rangle^5 \langle\frac{1}{l}\rangle^2 \hat{\omega}_l^{\mathrm{L}}(\eta)\|_{L_{l, \eta}^2}  \|\langle k-l, \xi-\eta+(k-l) t\rangle^{-2}\|_{L_{k-l, \xi-\eta}^2} \\
				& \quad \times\|\langle k-l, \xi-\eta+(k-l) t\rangle^3 \hat{\theta}_{k-l}^{\mathrm{NL}}(\xi-\eta)\|_{L_{k-l, \xi-\eta}^2} \\
				\lesssim &   \|\langle D_x, D_y+t D_x\rangle^3 \px \theta^{\mathrm{NL}}\|_{L^2}  \|\langle D_x, D_y+t D_x\rangle^6  \langle\frac{1}{l}\rangle^4  \omega^{\mathrm{L}}\|_{L^2}     \|\langle D_x, D_y+t D_x\rangle^3 \theta^{\mathrm{NL}}\|_{L^2}.
		\end{aligned}\end{equation}
		For $I_{l}$, by \eqref{eq:temp1.5}, we infer that
		\begin{equation}\begin{aligned} \label{eq:temp3}
				I_{l}^1 \lesssim & \|\langle k, \xi+k t\rangle^3 k\hat{\theta}_k^{\mathrm{NL}}(\xi)\|_{L_{k, \xi}^2}   \langle t \rangle^{-1} \|\langle l, \eta+l t\rangle^4 \langle\frac{1}{l}\rangle^2 \hat{\omega}_l^{\mathrm{L}}(\eta)\|_{L_{l, \eta}^2}  \|\langle l, \eta + l t\rangle^{-2}\|_{L_{l, \eta}^2} \\
				& \quad \times\|\langle k-l, \xi-\eta+(k-l) t\rangle^3 \left(|k-l|+|\xi-\eta|\right) \hat{\theta}_{k-l}^{\mathrm{NL}}(\xi-\eta)\|_{L_{k-l, \xi-\eta}^2} \\
				\lesssim & \langle t \rangle^{-1}  \|\langle D_x, D_y+t D_x\rangle^3 \px \theta^{\mathrm{NL}}\|_{L^2}  \|\langle D_x, D_y+t D_x\rangle^6  \langle\frac{1}{l}\rangle^4  \omega^{\mathrm{L}}\|_{L^2}     \|\langle D_x, D_y+t D_x\rangle^3 \nabla \theta^{\mathrm{NL}}\|_{L^2}.
		\end{aligned}\end{equation}
		From \eqref{eq:temp1}, \eqref{eq:temp2} and \eqref{eq:temp3}, we get that
		\begin{align} \label{eq:5}
			&\no\left|\left(\langle D_x, D_y+t D_x\rangle^3( \px u^{\mathrm{L}} \cdot \nabla \theta^{\mathrm{NL}}) \big| \langle D_x, D_y+t D_x\rangle^3 \px \theta^{\mathrm{NL}}\right)\right|\\
			&\no\lesssim \|\langle D_x, D_y+t D_x\rangle^3 \px \theta^{\mathrm{NL}}\|_{L^2}  \|\langle D_x, D_y+t D_x\rangle^6  \langle\frac{1}{l}\rangle^4  \omega^{\mathrm{L}}\|_{L^2}  \\
			&\quad \times   \left(\|\langle D_x, D_y+t D_x\rangle^3 \theta^{\mathrm{NL}}\|_{L^2}  + \langle t \rangle^{-1} \|\langle D_x, D_y+t D_x\rangle^3 \nabla \theta^{\mathrm{NL}}\|_{L^2} \right).
		\end{align}
		
		\underline{\bf Estimate of $\px u^{\mathrm{NL}} \cdot \nabla \theta^{\mathrm{NL}}$.}
		Similarly as \eqref{eq:temp1}, we infer that
		\begin{equation}\begin{aligned} \label{eq:temp4}
				& \left|\left(\langle D_x, D_y+t D_x\rangle^3\left( \px u^{\mathrm{NL}} \cdot \nabla \theta^{\mathrm{NL}}\right) \mid\langle D_x, D_y+t D_x\rangle^3 \px \theta^{\mathrm{NL}}\right)\right| \\
				&=  \left|\int_{\mathbb{R}^4}\langle k, \xi+k t\rangle^6 k \hat{\theta}_k^{\mathrm{NL}}(\xi) \frac{\eta(k-l)-l(\xi-\eta)}{|l|^2+|\eta|^2} l\hat{\omega}_l^{\mathrm{NL}}(\eta) \hat{\theta}_{k-l}^{\mathrm{NL}}(\xi-\eta) d k d l d \xi d \eta\right|. \\
				&\lesssim  \int_{\mathbb{R}^4}\left|\langle k, \xi+k t\rangle^3 k \hat{\theta}_k^{\mathrm{NL}}(\xi) \frac{\eta(k-l)-l(\xi-\eta)}{|l|^2+|\eta|^2} l \langle l, \eta+l t\rangle^3 \hat{\omega}_l^{\mathrm{NL}}(\eta)   \hat{\theta}_{k-l}^{\mathrm{NL}}(\xi-\eta) \right|d k d l d \xi d \eta \\
				& \quad + \int_{\mathbb{R}^4}\left|\langle k, \xi+k t\rangle^3 k \hat{\theta}_k^{\mathrm{NL}}(\xi) \frac{\eta(k-l)-l(\xi-\eta)}{|l|^2+|\eta|^2} l\hat{\omega}_l^{\mathrm{NL}}(\eta) \rt.\\
				&\lt. \qquad  \times  \langle k-l, \xi-\eta+(k-l) t\rangle^3 \hat{\theta}_{k-l}^{\mathrm{NL}}(\xi-\eta)\right| d k d l d \xi d \eta  =:   I_{k-l}' +I_{l}'.
		\end{aligned}\end{equation}
		For $I_{k-l}'$, using \eqref{eq:temp.4} and \eqref{eq:est of k-l + xi - eta},
		we get that
		\begin{equation}\begin{aligned}  \label{eq:temp5}
				I_{k-l}' \lesssim & \|\langle k, \xi+k t\rangle^3 k\hat{\theta}_k^{\mathrm{NL}}(\xi)\|_{L_{k, \xi}^2}\|\langle l, \eta+l t\rangle^3  \hat{\omega}_l^{\mathrm{NL}}(\eta)\|_{L_{l, \eta}^2}  \|\langle k-l, \xi-\eta+(k-l) t\rangle^{-2}\|_{L_{k-l, \xi-\eta}^2} \\
				& \quad \times \langle t\rangle \|\langle k-l, \xi-\eta+(k-l) t\rangle^3 \hat{\theta}_{k-l}^{\mathrm{NL}}(\xi-\eta)\|_{L_{k-l, \xi-\eta}^2} \\
				\lesssim & \langle t\rangle  \|\langle D_x, D_y+t D_x\rangle^3 \px \theta^{\mathrm{NL}}\|_{L^2}  \|\langle D_x, D_y+t D_x\rangle^3   \omega^{\mathrm{NL}}\|_{L^2}     \|\langle D_x, D_y+t D_x\rangle^3 \theta^{\mathrm{NL}}\|_{L^2}.
		\end{aligned}\end{equation}
		For $I_{l}'$, by \eqref{eq:temp.4} again, there holds that
		\begin{equation}\begin{aligned} \label{eq:temp6}
				I_{l}' \lesssim & \|\langle k, \xi+k t\rangle^3 k\hat{\theta}_k^{\mathrm{NL}}(\xi)\|_{L_{k, \xi}^2}  \|\langle l, \eta+l t\rangle^2 \hat{\omega}_l^{\mathrm{NL}}(\eta)\|_{L_{l, \eta}^2}  \|\langle l, \eta + l t\rangle^{-2}\|_{L_{l, \eta}^2} \\
				& \quad \times\|\langle k-l, \xi-\eta+(k-l) t\rangle^3 \left(|k-l|+|\xi-\eta|\right) \hat{\theta}_{k-l}^{\mathrm{NL}}(\xi-\eta)\|_{L_{k-l, \xi-\eta}^2} \\
				\lesssim &   \|\langle D_x, D_y+t D_x\rangle^3 \px \theta^{\mathrm{NL}}\|_{L^2}  \|\langle D_x, D_y+t D_x\rangle^3   \omega^{\mathrm{NL}}\|_{L^2}     \|\langle D_x, D_y+t D_x\rangle^3 \nabla \theta^{\mathrm{NL}}\|_{L^2}.
		\end{aligned}\end{equation}
		Collecting \eqref{eq:temp4}, \eqref{eq:temp5} and \eqref{eq:temp6}, we infer that
		\begin{align} \label{eq:6}
			&\no\left|\left(\langle D_x, D_y+t D_x\rangle^3( \px u^{\mathrm{NL}} \cdot \nabla \theta^{\mathrm{NL}}) \big| \langle D_x, D_y+t D_x\rangle^3 \px \theta^{\mathrm{NL}}\right)\right| \\
			&\no\lesssim \|\langle D_x, D_y+t D_x\rangle^3 \px \theta^{\mathrm{NL}}\|_{L^2}  \|\langle D_x, D_y+t D_x\rangle^3  \omega^{\mathrm{NL}}\|_{L^2}  \\
			&\quad \times   \left(  \langle t \rangle \|\langle D_x, D_y+t D_x\rangle^3 \theta^{\mathrm{NL}}\|_{L^2}  +  \|\langle D_x, D_y+t D_x\rangle^3 \nabla \theta^{\mathrm{NL}}\|_{L^2} \right).
		\end{align}
		
		At last, we complete the proof by summing up \eqref{eq:1}, \eqref{eq:2}, \eqref{eq:3}, \eqref{eq:4}, \eqref{eq:5} and \eqref{eq:6}.
	\end{proof}
	
	\subsection{$L^p$-estimate and proof of Proposition \ref{cor}.}
	\begin{Lem} \label{lem:temp1}
		For any $1<p<2$, it holds that
		$$
		\begin{aligned}
			& \frac{d}{d t}\|\langle D_x, D_y+t D_x\rangle^2 \theta^{\mathrm{NL}}\|_{L^p} \\
			&\lesssim   \langle t\rangle^{-2}  \|\langle D_x, D_y+t D_x\rangle^6 \langle\frac{1}{D_x}\rangle^4 \omega^{\mathrm{L}}\|_{L^2}  \|\langle D_x, D_y+t D_x\rangle^6 \langle\frac{1}{D_x}\rangle^4 \theta^{\mathrm{L}}\|_{L^2} \\
			& \quad+\langle t\rangle\left[ \|\langle D_x, D_y+t D_x\rangle^2 \omega^{\mathrm{NL}}\|_{L^p}\|\langle D_x, D_y+t D_x\rangle^6 \langle\frac{1}{D_x}\rangle^4 \theta^{\mathrm{L}}\|_{L^2}\right. \\
			& \left.\qquad \quad+(\|\langle D_x, D_y+t D_x\rangle^5 \langle\frac{1}{D_x}\rangle^4 \omega^{\mathrm{L}}\|_{L^p}+\|\langle D_x, D_y+t D_x\rangle^2 \omega^{\mathrm{NL}}\|_{L^p})\rt. \\
			&\lt. \qquad \qquad\times \|\langle D_x, D_y+t D_x\rangle^3 \theta^{\mathrm{NL}}\|_{L^2} \right].
		\end{aligned}
		$$
	\end{Lem}
	\begin{proof}
		Applying $\langle D_x, D_y+t D_x\rangle^2$ to $\eqref{eq:nonlinear part}_2$ and taking the inner product of the resulting equation with $\left|\langle D_x, D_y+t D_x\rangle^2 \theta^{\mathrm{NL}}\right|^{p-2}\langle D_x, D_y+t D_x\rangle^2 \theta^{\mathrm{NL}}$, we find that
		$$
		\begin{aligned}
			&  \frac1p \frac{d}{ d t}\|\langle D_x, D_y+t D_x\rangle^2 \theta^{\mathrm{NL}}\|_{L^p}^p \\
			\leq & -\Re\lt(\langle D_x, D_y+t D_x\rangle^2\lt(u \cdot \nabla \theta\rt)   \Big| \lt|\langle D_x, D_y+t D_x\rangle^2 \theta^{\mathrm{NL}}\rt|^{p-2}\langle D_x, D_y+t D_x\rangle^2 \theta^{\mathrm{NL}}\right).
		\end{aligned}
		$$
		By Hölder inequality, we have
		$$
		\begin{aligned}
			& \left|  \Re\lt(\langle D_x, D_y+t D_x\rangle^2\lt(u \cdot \nabla \theta\rt)   \Big| \lt|\langle D_x, D_y+t D_x\rangle^2 \theta^{\mathrm{NL}}\rt|^{p-2}\langle D_x, D_y+t D_x\rangle^2 \theta^{\mathrm{NL}}\right)  \right| \\
			\leq & \|\langle D_x, D_y+t D_x\rangle^2(u \cdot \nabla \theta)\|_{L^p}\|\langle D_x, D_y+t D_x\rangle^2 \theta^{\mathrm{NL}}\|_{L^p}^{p-1}.
		\end{aligned}
		$$
		and therefore
		$$
		\frac{d}{d t}\|\langle D_x, D_y+t D_x\rangle^2 \theta^{\mathrm{NL}}\|_{L^p} \leq\|\langle D_x, D_y+t D_x\rangle^2(u \cdot \nabla \theta)\|_{L^p}.
		$$

		\underline{\bf Estimate of $u^{\mathrm{L}} \cdot \nabla \theta^{\mathrm{L}}$.}
		For the source term $u^{\mathrm{L}} \cdot \nabla \theta^{\mathrm{L}}$, one gets
		$$
		\begin{aligned}
			& \|\langle D_x, D_y+t D_x\rangle^2\left(u^{\mathrm{L}} \cdot \nabla \theta^{\mathrm{L}}\right)\|_{L^p} \\
			\lesssim &  \|\langle D_x, D_y+t D_x\rangle^3 \Delta^{-1} \omega^{\mathrm{L}}\|_{L^2}\|\langle D_x, D_y+t D_x\rangle^3 \theta^{\mathrm{L}}\|_{L^{\frac{2 p}{2-p}}} \\
			\lesssim & \langle t\rangle^{-2} \|\langle D_x, D_y+t D_x\rangle^5\langle\frac{1}{D_x}\rangle^4 \omega^{\mathrm{L}}\|_{L^2}\|\langle D_x, D_y+t D_x\rangle^5 \theta^{\mathrm{L}}\|_{L^2},
		\end{aligned}
		$$
		where we used
		\begin{align} \label{eq:trick0}
			\nabla^\perp f \cdot \nabla g = \left(\begin{array}{c}
				-(\py + t\px)f \\
				\px f
			\end{array}\right) \cdot \left(\begin{array}{c}
				\px g \\
				(\py + t\px) g
			\end{array}\right)
		\end{align}
		and
		\begin{align} \label{eq:trick}
			\begin{aligned}
				& \|\langle D_x, D_y+t D_x\rangle^3 f\|_{L^{\frac{2 p}{2-p}}} \lesssim\|\langle k, \xi+k t\rangle^3 \hat{f}_k(\xi)\|_{L_{k, \xi}^{\frac{2 p}{3 p-2}}} \\
				& \lesssim  \|\langle k, \xi+k t\rangle^5 \hat{f}_k(\xi)\|_{L_{k, \xi}^2}\|\langle k, \xi+k t\rangle^{-2}\|_{L_{k, \xi}^{p^{\prime}}} \lesssim\|\langle D_x, D_y+t D_x\rangle^5 f\|_{L^2},
			\end{aligned}
		\end{align}
		for $1<p<2$ and some sufficiently smooth $f$ and $g$.
		
		\underline{\bf Estimate of $u^{\mathrm{NL}} \cdot \nabla \theta^{\mathrm{L}}$ and $u \cdot \nabla \theta^{\mathrm{NL}}$.}
		Using \eqref{eq:est of k-l + xi - eta} and the following classical estimate of Biot-Savart law:
		\begin{equation} \label{eq:Biot-Savart law}
			\|\nabla^{\perp}(-\Delta)^{-1} f\|_{L^{\frac{2 p}{2-p}}} \lesssim\|f\|_{L^p},
		\end{equation}
		we have that
		$$
		\begin{aligned}
			& \|\langle D_x, D_y+t D_x\rangle^2\left(u^{\mathrm{NL}} \cdot \nabla \theta^{\mathrm{L}}  +u \cdot \nabla \theta^{\mathrm{NL}}\right)\|_{L^p} \\
			& \lesssim\|\langle D_x, D_y+t D_x\rangle^2 u^{\mathrm{NL}}\|_{L^{\frac{2 p}{2-p}}}\|\langle D_x, D_y+t D_x\rangle^2 \nabla \theta^{\mathrm{L}}\|_{L^2} \\
			& \quad+\|\langle D_x, D_y+t D_x\rangle^2 u\|_{L^{\frac{2 p}{2-p}}}\|\langle D_x, D_y+t D_x\rangle^2 \nabla \theta^{\mathrm{NL}}\|_{L^2} \\
			& \lesssim\langle t\rangle\left[ \|\langle D_x, D_y+t D_x\rangle^2 \omega^{\mathrm{NL}}\|_{L^p}\|\langle D_x, D_y+t D_x\rangle^3 \theta^{\mathrm{L}}\|_{L^2}\right. \\
			& \left.\quad+(\|\langle D_x, D_y+t D_x\rangle^2 \omega^{\mathrm{L}}\|_{L^p}+\|\langle D_x, D_y+t D_x\rangle^2 \omega^{\mathrm{NL}}\|_{L^p})\|\langle D_x, D_y+t D_x\rangle^3 \theta^{\mathrm{NL}}\|_{L^2}\right] .
		\end{aligned}
		$$
		
		Adding all inequalities above together completes the proof.
	\end{proof}
	Similarly, for $\omega$ and $\px \theta$, we can use the analogous methods and state the estimates as follows.
	\begin{Lem} \label{lem:temp2}
		For any $1<p<2$, if the initial data satisfies \eqref{eq:the initial data}, then one has
		$$
		\begin{aligned}
			& \frac{d}{d t}\|\langle D_x, D_y+t D_x\rangle^2 \omega^{\mathrm{NL}}\|_{L^p} \\
			&\lesssim \langle t\rangle^{-2} \|\langle D_x, D_y+t D_x\rangle^6 \langle\frac{1}{D_x}\rangle^4 \omega^{\mathrm{L}}\|_{L^2}^2 +   \|\langle D_x, D_y+t D_x\rangle^2 \px \theta^{\mathrm{NL}} \|_{L^p}  \\
			& \quad+\langle t\rangle\left[ \|\langle D_x, D_y+t D_x\rangle^2 \omega^{\mathrm{NL}}\|_{L^p}\|\langle D_x, D_y+t D_x\rangle^6 \langle\frac{1}{D_x}\rangle^4 \omega^{\mathrm{L}}\|_{L^2}\right. \\
			& \left.\quad\quad \quad+\left(\|\langle D_x, D_y+t D_x\rangle^5 \langle\frac{1}{D_x}\rangle^4 \omega^{\mathrm{L}}\|_{L^p}+\|\langle D_x, D_y+t D_x\rangle^2 \omega^{\mathrm{NL}}\|_{L^p}\right) \rt. \\
			& \lt.\qquad\qquad\qquad\times\|\langle D_x, D_y+t D_x\rangle^3 \omega^{\mathrm{NL}}\|_{L^2}\right].
		\end{aligned}
		$$
	\end{Lem}
	
	\begin{Lem} \label{lem:temp3}
		For any $1<p<2$, if the initial data satisfies \eqref{eq:the initial data}, then one has
		$$
		\begin{aligned}
			& \frac{d}{d t}\|\langle D_x, D_y + t D_x\rangle^2 \px \theta^{\mathrm{NL}}\|_{L^p}  \\ &\lesssim \langle t\rangle^{-2} \|\langle D_x, D_y+t D_x\rangle^6\langle\frac{1}{D_x}\rangle^4 \omega^{\mathrm{L}}\|_{L^2}                 \|\langle D_x, D_y+t D_x\rangle^6\langle\frac{1}{D_x}\rangle^4 \theta^{\mathrm{L}}\|_{L^2} \\
			&\quad +\langle t \rangle\left[  \|\langle D_x, D_y+t D_x\rangle^6\langle\frac{1}{D_x}\rangle^4 \theta^{\mathrm{L}}\|_{L^2} \rt.\\
			&\lt. \qquad \qquad  \times \left(\|\langle D_x, D_y+t D_x\rangle^2 \omega^{\mathrm{NL}}\|_{L^p}  + \|\langle D_x, D_y+t D_x\rangle^3 \omega^{\mathrm{NL}}\|_{L^2} \right)  \rt.\\
			&\lt.\quad\quad \quad  + (   \|\langle D_x, D_y+t D_x\rangle^5 \langle\frac{1}{D_x}\rangle^4 \omega^{\mathrm{L}}\|_{L^p} +  \|\langle D_x, D_y+t D_x\rangle^6\langle\frac{1}{D_x}\rangle^4 \omega^{\mathrm{L}}\|_{L^2} \rt. \\
			&\lt. \quad\quad \qquad + \|\langle D_x, D_y+t D_x\rangle^3 \omega^{\mathrm{NL}}\|_{L^2} + \|\langle D_x, D_y+t D_x\rangle^2 \omega^{\mathrm{NL}}\|_{L^p}  ) \rt.\\
			&\lt. \qquad \qquad \times \|\langle D_x, D_y+t D_x\rangle^3 \langle\px\rangle \theta^{\mathrm{NL}}\|_{L^2}  \rt] .
		\end{aligned}
		$$
	\end{Lem}
	\begin{proof}
		By $\eqref{eq:nonlinear part}_2$, we have
		$$
		\begin{aligned}
			& \frac1p \frac{d}{ d t}\|\langle D_x, D_y+t D_x\rangle^2 \px \theta^{\mathrm{NL}}\|_{L^p}^p \\
			\leq & -\Re\lt(\langle D_x, D_y+t D_x\rangle^2 \px \lt(u \cdot \nabla \theta\rt)   \Big| \lt|\langle D_x, D_y+t D_x\rangle^2  \px \theta^{\mathrm{NL}}\rt|^{p-2}\langle D_x, D_y+t D_x\rangle^2 \px \theta^{\mathrm{NL}}\right),
		\end{aligned}
		$$
		which implies
		$$
		\frac{d}{d t}\|\langle D_x, D_y+t D_x\rangle^2 \px \theta^{\mathrm{NL}}\|_{L^p} \lesssim \|\langle D_x, D_y+t D_x\rangle^2 \px (u \cdot \nabla \theta)\|_{L^p}.
		$$
		
		\underline{\bf Estimate of $u \cdot \nabla \px \theta$.} Following a similar reasoning to that in Lemma \ref{lem:temp1}, with appropriate adjustments, we arrive at
		$$
		\begin{aligned}
			& \|\langle D_x, D_y+t D_x\rangle^2\left(u  \cdot \nabla \px \theta \right)\|_{L^p} \\
			&\lesssim   \langle t\rangle^{-2}  \|\langle D_x, D_y+t D_x\rangle^6 \langle\frac{1}{D_x}\rangle^4 \omega^{\mathrm{L}}\|_{L^2}  \|\langle D_x, D_y+t D_x\rangle^6 \langle\frac{1}{D_x}\rangle^4  \theta^{\mathrm{L}}\|_{L^2} \\
			& \quad+\langle t\rangle\left[ \|\langle D_x, D_y+t D_x\rangle^2 \omega^{\mathrm{NL}}\|_{L^p}\|\langle D_x, D_y+t D_x\rangle^6 \langle\frac{1}{D_x}\rangle^4  \theta^{\mathrm{L}}\|_{L^2}\right. \\
			& \left.\qquad \quad+(\|\langle D_x, D_y+t D_x\rangle^5 \langle\frac{1}{D_x}\rangle^4 \omega^{\mathrm{L}}\|_{L^p}+\|\langle D_x, D_y+t D_x\rangle^2 \omega^{\mathrm{NL}}\|_{L^p})\rt. \\
			&\lt. \qquad \qquad\times \|\langle D_x, D_y+t D_x\rangle^3\px \theta^{\mathrm{NL}}\|_{L^2} \right].
		\end{aligned}
		$$
		
		\underline{\bf Estimate of $ \px u^{\mathrm{L}} \cdot \nabla \theta^{\mathrm{L}}$.}
		By \eqref{eq:trick0} and \eqref{eq:trick}, one has
		$$
		\begin{aligned}
			& \|\langle D_x, D_y+t D_x\rangle^2\left( \px u^{\mathrm{L}} \cdot \nabla  \theta^{\mathrm{L}}\right)\|_{L^p} \\
			&\lesssim  \|\langle D_x, D_y+t D_x\rangle^4 \Delta^{-1} \omega^{\mathrm{L}}\|_{L^2}\|\langle D_x, D_y+t D_x\rangle^3  \theta^{\mathrm{L}}\|_{L^{\frac{2 p}{2-p}}} \\
			&\lesssim  \langle t\rangle^{-2}\|\langle D_x, D_y+t D_x\rangle^6 \langle\frac{1}{D_x}\rangle^4 \omega^{\mathrm{L}}\|_{L^2}\|\langle D_x, D_y+t D_x\rangle^5  \theta^{\mathrm{L}}\|_{L^2}.
		\end{aligned}
		$$
		
		\underline{\bf Estimate of $ \px u^{\mathrm{NL}} \cdot \nabla \theta $ and $\px u^{\mathrm{L}} \cdot \nabla \theta ^{\mathrm{NL}} $.}
		Thank to \eqref{eq:trick}, \eqref{eq:Biot-Savart law}, \eqref{eq:est of k-l + xi - eta} and \eqref{eq:temp.4}, we have
		$$
		\begin{aligned}
			& \|\langle D_x, D_y+t D_x\rangle^2\left( \px u^{\mathrm{NL}} \cdot \nabla  \theta^{\mathrm{L}}+ \px u^{\mathrm{L}} \cdot \nabla   \theta^{\mathrm{NL}}\right)\|_{L^p} \\
			& \lesssim\|\langle D_x, D_y+t D_x\rangle^2 \px u^{\mathrm{NL}}\|_{L^2}\|\langle D_x, D_y+t D_x\rangle^2 \nabla  \theta^{\mathrm{L}}\|_{L^\frac{2 p}{2-p}} \\
			& \quad+\|\langle D_x, D_y+t D_x\rangle^2 \px u^{\mathrm{L}}\|_{L^{\frac{2 p}{2-p}}}\|\langle D_x, D_y+t D_x\rangle^2 \nabla  \theta^{\mathrm{NL}}\|_{L^2} \\
			& \lesssim  \langle t\rangle \left(  \|\langle D_x, D_y+t D_x\rangle^2 \omega^{\mathrm{NL}}\|_{L^2}\|\langle D_x, D_y+t D_x\rangle^5 \theta^{\mathrm{L}}\|_{L^2}\right. \\
			& \left.\quad+ \|\langle D_x, D_y+t D_x\rangle^3 \omega^{\mathrm{L}}\|_{L^p}   \|\langle D_x, D_y+t D_x\rangle^3  \theta^{\mathrm{NL}}\|_{L^2}\right) 
		\end{aligned}
		$$
		and
		$$
		\begin{aligned}
			& \|\langle D_x, D_y+t D_x\rangle^2\left( \px u^{\mathrm{NL}} \cdot \nabla  \theta^{\mathrm{NL}} \right)\|_{L^p} \\
			& \lesssim\|\langle D_x, D_y+t D_x\rangle \px u^{\mathrm{NL}}\|_{L^\frac{2 p}{2-p}}\|\langle D_x, D_y+t D_x\rangle^2 \nabla  \theta^{\mathrm{NL}}\|_{L^2} \\
			& \quad+\|\langle D_x, D_y+t D_x\rangle^2 \px u^{\mathrm{NL}}\|_{L^2}\| \nabla  \theta^{\mathrm{NL}}\|_{L^\frac{2 p}{2-p}} \\
			& \lesssim  \langle t\rangle (  \|\langle D_x, D_y+t D_x\rangle^2 \omega^{\mathrm{NL}}\|_{L^p}+ \|\langle D_x, D_y+t D_x\rangle^2 \omega^{\mathrm{NL}}\|_{L^2} ) \|\langle D_x, D_y+t D_x\rangle^3 \theta^{\mathrm{NL}}\|_{L^2}.
		\end{aligned}
		$$
		
		In all, the proof is completed.
	\end{proof}

	To conclude this section, we present the proof of Proposition \ref{cor}.
	
	\begin{proof}[Proof of Proposition \ref{cor}] For $t\leq T_0=\nu^{-\frac16}$, we get by Lemma \ref{lem:est of omega L2} that
		\begin{equation*}
			\begin{aligned}
				& \|\langle D_x, D_y+t D_x\rangle^3 \omega^{\mathrm{NL}}\|_{L^\infty L^2}\\ &\lesssim  \nu^{\frac23 + 2\delta}  +    \nu^{-\frac16} \|\langle D_x, D_y+t D_x\rangle^3 \px \theta^{\mathrm{NL}}\|_{L^\infty L^2}    \\
				& \quad + \nu^{-\frac13 -\frac\delta6 }(  \nu^{\frac13 + \delta} +\|\langle D_x, D_y+t D_x\rangle^3 \omega^{\mathrm{NL}}\|_{L^\infty L^2}) \\
				&\qquad \times 
				(\|\omega^{\mathrm{NL}}\|_{L^\infty L^p}+\|\langle D_x, D_y+t D_x\rangle^3 \omega^{\mathrm{NL}}\|_{L^\infty L^2}) ,
			\end{aligned}
		\end{equation*}
		by Lemma \ref{eq:est of theta NL short time} that
		$$
		\begin{aligned}
			& \|\langle D_x, D_y+t D_x\rangle^3 \theta^{\mathrm{NL}}\|_{L^\infty L^2}^2 + \nu \|\langle D_x, D_y+t D_x\rangle^3 \nabla \theta^{\mathrm{NL}}\|_{L^2 L^2}^2 \\ &\lesssim\|\langle D_x, D_y+t D_x\rangle^3 \theta^{\mathrm{NL}}\|_{L^\infty L^2} \lt[ \nu^{1 + 3\delta} \rt. \\
			& \lt. \qquad
			+ \nu^{\frac13 + \frac{11}{6}\delta} (\|\omega^{\mathrm{NL}}\|_{L^\infty L^p}+\|\langle D_x, D_y+t D_x\rangle^3 \omega^{\mathrm{NL}}\|_{L^\infty L^2}) 
			\rt.  \\
			& \left.\quad \quad  + \nu^{-\frac13 -\frac\delta6 } \|\langle D_x, D_y+t D_x\rangle^3 \theta^{\mathrm{NL}}\|_{L^\infty L^2}(\nu^{\frac13 + \delta}  + \|\langle D_x, D_y+t D_x\rangle^3 \omega^{\mathrm{NL}}\|_{L^\infty L^2} ) \rt],
		\end{aligned}
		$$
		by Lemma \ref{lem:est of pxtheta L2} that
		\begin{equation*}\begin{aligned} 
				&  \|\langle D_x, D_y+t D_x\rangle^3 \px \theta^{\mathrm{NL}}\|_{L^\infty L^2} \\
				&\lesssim  \nu^{\frac13 + \delta} \lt( \nu^{\frac23 + 2\delta} + \nu^{-\frac13 -\frac\delta6 } \|\langle D_x, D_y+t D_x\rangle^3 \px \theta^{\mathrm{NL}}\|_{L^\infty L^2}  \rt.\\
				&\lt. \qquad \quad +   \nu^{-\frac16}\|\langle D_x, D_y+t D_x\rangle^3 \theta^{\mathrm{NL}}\|_{L^\infty L^2}  +   \|\langle D_x, D_y+t D_x\rangle^3 \nabla \theta^{\mathrm{NL}}\|_{L^2 L^2} \rt)   \\
				& \quad + \nu^{\frac13 + \frac{11}{6}\delta}  \lt(  \|\omega^{\mathrm{NL}}\|_{L^\infty L^p}   + \|\langle D_x, D_y+t D_x\rangle^3 \omega^{\mathrm{NL}}\|_{L^\infty L^2} \rt)  \\
				& \quad  +   \|\langle D_x, D_y+t D_x\rangle^3 \omega^{\mathrm{NL}}\|_{L^\infty L^2} \lt( \nu^{-\frac13 -\frac\delta6 } \|\langle D_x, D_y+t D_x\rangle^3 \px \theta^{\mathrm{NL}}\|_{L^\infty L^2}   \rt. \\
				&\lt.\qquad \qquad \quad + \nu^{-\frac13}  \|\langle D_x, D_y+t D_x\rangle^3 \theta^{\mathrm{NL}}\|_{L^\infty L^2} + \nu^{-\frac{1}{12}} \|\langle D_x, D_y+t D_x\rangle^3 \nabla \theta^{\mathrm{NL}}\|_{L^2 L^2} \rt) ,
		\end{aligned}\end{equation*}
		by Lemma \ref{lem:temp2} that
		$$
		\begin{aligned}
			&\|\langle D_x, D_y+t D_x\rangle^2 \omega^{\mathrm{NL}}\|_{L^\infty L^p} \\
			&\lesssim  \nu^{\frac23 + 2\delta} +   \nu^{-\frac16}\|\langle D_x, D_y+t D_x\rangle^2 \px \theta^{\mathrm{NL}} \|_{L^\infty L^p}  \\
			& \quad+\nu^{-\frac13} \left[   \nu^{\frac13+ \delta}  \|\langle D_x, D_y+t D_x\rangle^2 \omega^{\mathrm{NL}}\|_{L^\infty L^p} \right. \\
			& \left.\qquad \qquad +(\nu^{\frac13+ \delta} +\|\langle D_x, D_y+t D_x\rangle^2 \omega^{\mathrm{NL}}\|_{L^\infty L^p})\|\langle D_x, D_y+t D_x\rangle^3 \omega^{\mathrm{NL}}\|_{L^\infty L^2}\right]\\
			&\quad +\nu^{-\frac16}\|\langle D_x, D_y+t D_x\rangle^3 \partial_{x}\theta^{\mathrm{NL}}\|_{L^\infty L^p},
		\end{aligned}
		$$
		by Lemma \ref{lem:temp1} that
		$$
		\begin{aligned}
			&\|\langle D_x, D_y+t D_x\rangle^2 \theta^{\mathrm{NL}}\|_{L^\infty L^p} \\
			& \lesssim  \nu^{1 + 3\delta} 
			+  \nu^{-\frac13}  \left[   \nu^{\frac23 + 2 \delta}\|\langle D_x, D_y+t D_x\rangle^2 \omega^{\mathrm{NL}}\|_{L^\infty L^p}  \right. \\
			& \left. \qquad \qquad \quad+(\nu^{\frac13 + \delta}+\|\langle D_x, D_y+t D_x\rangle^2 \omega^{\mathrm{NL}}\|_{L^\infty L^p})\|\langle D_x, D_y+t D_x\rangle^3 \theta^{\mathrm{NL}}\|_{L^\infty L^2}\right],
		\end{aligned}
		$$
		and by Lemma \ref{lem:temp3} that
		$$
		\begin{aligned}
			& \|\langle D_x, D_y + t D_x\rangle^2 \px \theta^{\mathrm{NL}}\|_{L^\infty L^p}  \\ &\lesssim \nu^{1 + 3\delta}  +  \nu^{-\frac13}\lt[  \nu^{\frac23 + 2\delta} (\|\langle D_x, D_y+t D_x\rangle^2 \omega^{\mathrm{NL}}\|_{L^\infty L^p}  + \|\langle D_x, D_y+t D_x\rangle^3 \omega^{\mathrm{NL}}\|_{L^\infty L^2} )  \rt.\\
			&\lt. \qquad \qquad + (   \nu^{\frac13 +\delta} + \|\langle D_x, D_y+t D_x\rangle^3 \omega^{\mathrm{NL}}\|_{L^\infty L^2} + \|\langle D_x, D_y+t D_x\rangle^2 \omega^{\mathrm{NL}}\|_{L^\infty L^p}  ) \rt. \\
			&\lt. \qquad \qquad \quad\quad \times \|\langle D_x, D_y+t D_x\rangle^3 \langle\px\rangle \theta^{\mathrm{NL}}\|_{L^\infty L^2}  \rt] .
		\end{aligned}
		$$
		By bootstrap arguments, there exist a small enough $0< \nu_1<1$ such that if $0<\nu<\nu_1$, \eqref{eq:temp7} holds. The proof is completed.
	\end{proof}

	\section{Estimate of nonlinear part in long time scale $t\geq  T_0 = \nu^{-\frac16}$} \label{sec.5}
	In the last section, we will do energy estimates with the multiplier $\langle\frac{1}{D_x}\rangle^\epsilon$ and the Fourier multiplier $\mathcal{M}$ we construct in Section \ref{Sec.2}. The multiplier $\langle\frac{1}{D_x}\rangle^\epsilon$ helps to handle singularities, while the multiplier $\mathcal{M}$ is devoted to extracting the stabilizing mechanisms and dealing with echo cascades.
	
	\subsection{Estimate of $\theta^{\mathrm{NL}}$}
	\begin{Prop} \label{lem:est of thetanl}
		There exist $\frac{1-\delta}{2}<\epsilon<\frac{1}{2}$, $0<\kappa<\frac{2\epsilon}{1-\delta}-1$ and $c<c(\delta,\epsilon)$ small enough such that if the initial data $(w_{in}, \theta_{in})$ satisfies \eqref{eq:the initial data}, then for any $  T \geq T_{0} = \nu^{-\frac{1}{6}} $, it holds that
		\footnotesize{\begin{align*} 
				& \| e^{c \nu^{\frac{1}{3}} \lambda\left(D_x\right) t}\langle D_x\rangle\langle\frac{1}{D_x}\rangle^\epsilon \langle D_x \rangle^\frac13  \theta^{\mathrm{NL}}\|_{L_{[T_0, T]}^\infty L^2}^2+\nu\|e^{c \nu^{\frac{1}{3}} \lambda\left(D_x\right) t}\langle D_x\rangle\langle\frac{1}{D_x}\rangle^\epsilon \nabla \langle D_x \rangle^\frac13  \theta^{\mathrm{NL}}\|_{L_{[T_0, T]}^2 L^2}^2 \\
				& \quad+ \nu^{\frac{1}{3}}\|e^{c \nu^{\frac{1}{3}} \lambda\left(D_x\right) t}\langle D_x\rangle\langle\frac{1}{D_x}\rangle^\epsilon\left|D_x\right|^{\frac{1}{3}}  \langle D_x \rangle ^\frac13 \theta^{\mathrm{NL}}\|_{L_{[T_0, T]}^2 L^2}^2  
				\\
				& \quad+\|e^{c \nu^{\frac{1}{3}} \lambda\left(D_x\right) t}\langle D_x\rangle\langle\frac{1}{D_x}\rangle^\epsilon \sqrt{\Upsilon\left(t, D_x, D_y\right)} \langle D_x \rangle^\frac13  \theta^{\mathrm{NL}}\|_{L_{[T_0, T]}^2 L^2}^2 \\
				\lesssim &  \| e^{c \nu^{\frac{1}{3}} \lambda\left(D_x\right) t}\langle D_x\rangle\langle\frac{1}{D_x}\rangle^\epsilon \langle D_x \rangle^\frac13  \theta^{\mathrm{NL}}(T_0)\|_{L^2}^2 \\	&  +\nu^\frac16 \|e^{c \nu^{\frac{1}{3}} \lambda\left(D_x\right) t}\langle D_x\rangle  \langle \frac{1}{D_x}\rangle^\epsilon \langle D_x \rangle^\frac13  \theta^{\mathrm{NL}}\|_{L_{[T_0, T]}^\infty L^2}    \|e^{c \nu^{\frac{1}{3}} \lambda\left(D_x\right) t}\langle D_x, D_y + tD_x\rangle^5\langle\frac{1}{D_x}\rangle^4 \omega^{\mathrm{L}} \|_{L_{[T_0, T]}^\infty L^2} \\
				&\quad \times  \|e^{c \nu^{\frac{1}{3}} \lambda\left(D_x\right) t}\langle D_x, D_y + tD_x\rangle^5\langle\frac{1}{D_x}\rangle^4 \langle D_x \rangle^\frac13  \theta^{\mathrm{L}} \|_{L_{[T_0, T]}^\infty L^2} \\
				&+ \nu^\frac14 \|e^{c \nu^{\frac{1}{3}} \lambda\left(D_x\right) t}\langle D_x\rangle  \langle \frac{1}{D_x}\rangle^\epsilon \langle D_x \rangle^\frac13  \theta^{\mathrm{NL}}\|_{L_{[T_0, T]}^\infty L^2}    \|e^{c \nu^{\frac{1}{3}} \lambda\left(D_x\right) t}\langle D_x, D_y + tD_x\rangle^5\langle\frac{1}{D_x}\rangle^4 \omega^{\mathrm{L}} \|_{L_{[T_0, T]}^\infty L^2} \\
				&\quad \times \|e^{c \nu^{\frac{1}{3}} \lambda\left(D_x\right) t} \nabla \langle D_x \rangle^\frac13  \theta^{\mathrm{NL}} \|_{L_{[T_0, T]}^2 L^2} \\
				& +    \|e^{c \nu^{\frac{1}{3}} \lambda\left(D_x\right) t}\langle D_x, D_y+t D_x\rangle^5\langle\frac{1}{D_x}\rangle^4 \omega^{\mathrm{L}}\|_{L_{[T_0, T]}^\infty L^2} \\
				& \quad \times ( \nu^\frac1{12} \|e^{c \nu^{\frac{1}{3}} \lambda\left(D_x\right) t}\langle D_x\rangle\langle\frac{1}{D_x}\rangle^\epsilon \langle D_x \rangle^\frac13 \theta^{\mathrm{NL}}\|_{L_{[T_0, T]}^\infty L^2}  \| e^{c \nu^{\frac{1}{3}} \lambda\left(D_x\right) t} \langle D_x \rangle \langle\frac{1}{D_x}\rangle^\epsilon
				\left|D_x\right|^{\frac{1}{3}}  \langle D_x \rangle ^\frac13 \theta^{\mathrm{NL}} \|_{L_{[T_0, T]}^2 L^2} \\
				& \qquad + \nu^\frac16  \|e^{c \nu^{\frac{1}{3}} \lambda\left(D_x\right) t}\langle D_x\rangle\langle\frac{1}{D_x}\rangle^\epsilon\left|D_x\right|^{\frac{1}{3}}  \langle D_x \rangle^\frac13  \theta^{\mathrm{NL}}\|_{L_{[T_0, T]}^2 L^2}^{\frac{3}{2}}   \|e^{c \nu^{\frac{1}{3}} \lambda\left(D_x\right) t}\langle D_x\rangle\langle\frac{1}{D_x}\rangle^\epsilon \nabla \langle D_x \rangle^\frac13  \theta^{\mathrm{NL}}\|_{L_{[T_0, T]}^2 L^2}^{\frac{1}{2}})\\
				&+ ( \|e^{c \nu^{\frac{1}{3}} \lambda\left(D_x\right) t}\langle D_x\rangle  \langle\frac{1}{D_x}\rangle^\epsilon \langle D_x \rangle^\frac13  \theta^{\mathrm{NL}}\|_{L_{[T_0, T]}^\infty L^2}\|e^{c \nu^{\frac{1}{3}} \lambda\left(D_x\right) t}\langle D_x\rangle \langle\frac{1}{D_x}\rangle^\epsilon \partial_x \nabla \phi^{\mathrm{NL}}\|_{L_{[T_0, T]}^2 L^2}  \\
				& \quad + \|e^{c \nu^{\frac{1}{3}} \lambda\left(D_x\right) t}\langle D_x\rangle    \langle\frac{1}{D_x}\rangle^\epsilon    \left|D_x\right|^{\frac{1}{3}} \langle D_x \rangle^\frac13  \theta^{\mathrm{NL}}\|_{L_{[T_0, T]}^2 L^2}\|e^{c \nu^{\frac{1}{3}} \lambda\left(D_x\right) t}\langle D_x\rangle\langle\frac{1}{D_x}\rangle^\epsilon \omega^{\mathrm{NL}}\|_{L_{[T_0, T]}^\infty L^2} ) \\
				& \qquad \times\|e^{c \nu^{\frac{1}{3}} \lambda\left(D_x\right) t}\langle D_x, D_y+t D_x\rangle^5\langle\frac{1}{D_x}\rangle^4\left|D_x\right|^{\frac{1}{3}} \langle D_x \rangle^\frac13  \theta^{\mathrm{L}}\|_{L_{[T_0, T]}^2 L^2} \\
				&  + \nu^{-\frac{1}{3}-\frac{\delta}{6}}\|e^{c \nu^{\frac{1}{3}} \lambda\left(D_x\right)}\langle D_x\rangle\langle\frac{1}{D_x}\rangle^\epsilon \partial_x \nabla \phi^{\mathrm{NL}}\|_{L_{[T_0, T]}^2 L^2}\|e^{c_0 \nu\left|D_x\right|^2 t^3}\langle D_x, D_y+t D_x\rangle^5\langle\frac{1}{D_x}\rangle^4 \theta^{\mathrm{L}}\|_{L_{[T_0, T]}^\infty (L^1 \cap L^2)} \\
				&\quad  \times(\|e^{c \nu^{\frac{1}{3}} \lambda\left(D_x\right) t}\langle D_x\rangle\langle\frac{1}{D_x}\rangle^\epsilon \sqrt{\Upsilon} \langle D_x \rangle^\frac13  \theta^{\mathrm{NL}}\|_{L_{[T_0, T]}^2 L^2}+\nu^{\frac{1}{6}}\|e^{c \nu^{\frac{1}{3}} \lambda\left(D_x\right) t}\langle D_x\rangle\langle\frac{1}{D_x}\rangle^\epsilon |D_x|^\frac13 \langle D_x \rangle^{\frac{1}{3}} \theta^{\mathrm{NL}}\|_{L_{[T_0, T]}^2 L^2}) \\
				&+\|e^{c \nu^{\frac{1}{3}} \lambda\left(D_x\right) t}\langle D_x\rangle    \langle \frac{1}{D_x}\rangle^\epsilon  \left|D_x\right|^{\frac{1}{3}}\langle D_x \rangle^\frac13 \theta^{\mathrm{NL}}\|_{L_{[T_0, T]}^2 L^2}^\frac32 \|e^{c \nu^{\frac{1}{3}} \lambda\left(D_x\right) t}\langle D_x\rangle\langle\frac{1}{D_x}\rangle^\epsilon \omega^{\mathrm{NL}}\|_{L_{[T_0, T]}^\infty L^2} \\
				& \quad \times\|e^{c \nu^{\frac{1}{3}} \lambda\left(D_x\right) t}\langle D_x\rangle\langle\frac{1}{D_x}\rangle^\epsilon \nabla  \langle D_x \rangle^{\frac{1}{3}} \theta^{\mathrm{NL}}\|_{L_{[T_0, T]}^2 L^2}^\frac12 \\
				&+ \|e^{c \nu^{\frac{1}{3}} \lambda\left(D_x\right) t}\langle D_x\rangle    \langle \frac{1}{D_x}\rangle^\epsilon  \langle D_x \rangle^{\frac{1}{3}} \theta^{\mathrm{NL}}\|_{L_{[T_0, T]}^\infty L^2} \|e^{c \nu^{\frac{1}{3}} \lambda\left(D_x\right) t}\langle D_x\rangle\langle\frac{1}{D_x}\rangle^\epsilon \px \nabla \phi^{\mathrm{NL}}\|_{L_{[T_0, T]}^2 L^2} \\
				& \quad \times\|e^{c \nu^{\frac{1}{3}} \lambda\left(D_x\right) t}\langle D_x\rangle\langle\frac{1}{D_x}\rangle^\epsilon \nabla \langle D_x \rangle^{\frac{1}{3}} \theta^{\mathrm{NL}}\|_{L_{[T_0, T]}^2 L^2} \\
				&+\|e^{c \nu^{\frac{1}{3}} \lambda\left(D_x\right) t}\langle D_x\rangle    \langle \frac{1}{D_x}\rangle^\epsilon  \left|D_x\right|^{\frac{1}{3}} \langle D_x \rangle^{\frac{1}{3}} \theta^{\mathrm{NL}}\|_{L_{[T_0, T]}^2 L^2}^2 \|e^{c \nu^{\frac{1}{3}} \lambda\left(D_x\right) t}\langle D_x\rangle\langle\frac{1}{D_x}\rangle^\epsilon \omega^{\mathrm{NL}}\|_{L_{[T_0, T]}^\infty L^2}  
				.
		\end{align*} }
	\end{Prop}
	
	\begin{Rem}
For the choice of the parameters such as $\epsilon, \kappa$ and $\alpha$ which appears below, one can find them in \cite{LLZ2025}.
	\end{Rem}

	\begin{proof}
		Taking the inner product of $\eqref{eq:nonlinear part}_2$ with $\mathcal{M} e^{2 c \nu^{\frac{1}{3}} \lambda\left(D_x\right) t}\langle D_x\rangle^2\langle\frac{1}{D_x}\rangle^{2 \epsilon} \langle D_x \rangle^\frac23  \theta^{\mathrm{NL}}$, we infer that
		\small{	$$
			\begin{aligned}
				\frac{d}{d t} \| \sqrt{\mathcal{M}} & e^{c \nu^{\frac{1}{3}} \lambda\left(D_x\right) t}\langle D_x\rangle \langle\frac{1}{D_x}\rangle^\epsilon \langle D_x \rangle^\frac13  \theta^{\mathrm{NL}}\|_{L^2}^2-2 c \nu^{\frac{1}{3}}\| \sqrt{\mathcal{M}} e^{c \nu^{\frac{1}{3}} \lambda\left(D_x\right) t}\langle D_x\rangle\langle\frac{1}{D_x}\rangle^\epsilon \sqrt{\lambda\left(D_x\right)} \langle D_x \rangle^\frac13  \theta^{\mathrm{NL}} \|_{L^2}^2 \\
				& +2 \nu\|\sqrt{\mathcal{M}} e^{c \nu^{\frac{1}{3}} \lambda\left(D_x\right) t}\langle D_x\rangle\langle\frac{1}{D_x}\rangle^\epsilon \nabla \langle D_x \rangle^\frac13  \theta^{\mathrm{NL}}\|_{L^2}^2 \\
				& +\int_{\mathbb{R}^2}\left(-\partial_t+k \partial_{\xi}\right) \mathcal{M}(k, \xi) e^{2 c \nu^{\frac{1}{3}} \lambda(k) t}\langle k\rangle^2\langle\frac{1}{k}\rangle^{2 \epsilon} \langle k \rangle^\frac23 \left|\hat{\theta}_k^{\mathrm{NL}}(\xi)\right|^2 d k d \xi \\
				= & -2 \Re\left((u^{\mathrm{L}}+u^{\mathrm{NL}}) \cdot \nabla(\theta^{\mathrm{L}}+\theta^{\mathrm{NL}}) \big| \mathcal{M} e^{2 c \nu^{\frac{1}{3}} \lambda\left(D_x\right) t}\langle D_x\rangle^2\langle\frac{1}{D_x}\rangle^{2 \epsilon}  \langle D_x \rangle^\frac23  \theta^{\mathrm{NL}}\right).
			\end{aligned}
			$$}
		Thanks to \eqref{eq:bound of M}, \eqref{eq:m3} and \eqref{eq:m1m2}, we use the fact that $\lambda(k) \leq|k|^{\frac{2}{3}}$ to deduce that there exist some $0<c<C(C_\kappa)$ such that
		$$
		\begin{aligned}
			& \frac{d}{d t}\|\sqrt{\mathcal{M}} e^{c \nu^{\frac{1}{3}} \lambda\left(D_x\right) t}\langle D_x\rangle\langle\frac{1}{D_x}\rangle^\epsilon \langle D_x \rangle^\frac13  \theta^{\mathrm{NL}}\|_{L^2}^2+\nu\|e^{c \nu^{\frac{1}{3}} \lambda\left(D_x\right) t}\langle D_x\rangle\langle\frac{1}{D_x}\rangle^\epsilon \nabla \langle D_x \rangle^\frac13  \theta^{\mathrm{NL}}\|_{L^2}^2 \\
			& \quad+\frac{\nu^{\frac{1}{3}}}{16}\|e^{c \nu^{\frac{1}{3}} \lambda\left(D_x\right) t}\langle D_x\rangle\langle\frac{1}{D_x}\rangle^\epsilon\left|D_x\right|^{\frac{1}{3}}  \langle D_x \rangle ^\frac13 \theta^{\mathrm{NL}}\|_{L^2}^2  
			\\
			& \quad+\|e^{c \nu^{\frac{1}{3}} \lambda\left(D_x\right) t}\langle D_x\rangle\langle\frac{1}{D_x}\rangle^\epsilon \sqrt{\Upsilon\left(t, D_x, D_y\right)} \langle D_x \rangle^\frac13  \theta^{\mathrm{NL}}\|_{L^2}^2 \\
			& \leq 2\left|\Re\left((u^{\mathrm{L}}+u^{\mathrm{NL}}) \cdot \nabla(\theta^{\mathrm{L}}+\theta^{\mathrm{NL}}) \big| \mathcal{M} e^{2 c \nu^{\frac{1}{3}} \lambda\left(D_x\right) t}\langle D_x\rangle^2\langle\frac{1}{D_x}\rangle^{2 \epsilon}  \langle D_x \rangle^\frac23   \theta^{\mathrm{NL}}\right)\right|.
		\end{aligned}
		$$
		Since 
		\begin{align*}
			(fg| \langle D_x \rangle^\frac23 h) &\lesssim   \big((1+|k|^2)^\frac16 \widehat{fg}| (1+|k|^2)^\frac16 \hat{h} \big) \\
			&\lesssim \int_{\mathbb{R}^2}  |k-l|^\frac13|\hat{f}(k-l) \hat{g}(l)| (1+|k|^2)^\frac16 |\hat{h}(k)| dkdl \\
			&\quad  +  \int_{\mathbb{R}^2} (1 + |l|^2)^\frac16|\hat{f}(k-l) \hat{g}(l)| (1+|k|^2)^\frac16 |\hat{h}(k)| dkdl
		\end{align*}
        holds for some $f,g$ and $h$, we formally decompose $\langle D_x \rangle^\frac13 \big[(u^{\mathrm{L}}+u^{\mathrm{NL}}) \cdot \nabla(\theta^{\mathrm{L}}+\theta^{\mathrm{NL}}) \big] $ into eight terms and proceed to estimate them one by one, which does not affect the upper bound estimates above.

		\underline{\bf Step I: Estimate of  $u^{\mathrm{L}} \cdot  \nabla \langle D_x \rangle^\frac13 \theta^{\mathrm{L}}$.}
		Applying the Fourier transformation, there holds that
		\begin{equation}\begin{aligned} \label{eq:temp8}
				& \left|\left(u^{\mathrm{L}} \cdot \nabla \langle D_x \rangle^\frac13 \theta^{\mathrm{L}} \big| \mathcal{M} e^{2 c \nu^{\frac{1}{3}} \lambda(D_x) t}\langle D_x\rangle^2\langle\frac{1}{D_x}\rangle^{2 \epsilon} \langle D_x \rangle^\frac13 \theta^{\mathrm{NL}}\right)\right| \\
				= & \left|\int_{\mathbb{R}^4} \mathcal{M}(t, k, \xi) e^{2 c \nu^{\frac{1}{3}} \lambda(k) t}\langle k\rangle^2\langle\frac{1}{k}\rangle^{2 \epsilon} \langle k \rangle^\frac13  \hat{\theta }_k^{\mathrm{NL}}(\xi) \langle k-l \rangle^\frac13 \rt. \\
				& \lt.\qquad \times \frac{\eta(k-l)-l(\xi-\eta)}{|l|^2+|\eta|^2} \hat{\omega}_l^{\mathrm{L}}(\eta) \hat{\theta }_{k-l}^{\mathrm{L}}(\xi-\eta) d k d l d \xi d \eta\right|.
		\end{aligned}\end{equation}
		We then consider two cases for the frequency $k$: high and low frequencies.\\
		$\bullet$~Case 1: $|k| \geq 1$.
		Using \eqref{eq:0}, \eqref{eq:the point-wise inviscid damping estimate} and 
		\begin{equation*} 
			\langle k \rangle \langle \frac1k \rangle^\epsilon \lesssim \langle l, \eta+l t\rangle\langle k-l, \xi-\eta+(k-l) t\rangle,
		\end{equation*}
		we get that
		\begin{equation} \label{eq:temp9}
			\begin{aligned}
				&  \left|\int_{|k| \geq  1} \mathcal{M}(t, k, \xi) e^{2 c \nu^{\frac{1}{3}} \lambda(k) t}\langle k\rangle^2\langle\frac{1}{k}\rangle^{2 \epsilon} \langle k \rangle^\frac13  \hat{\theta }_k^{\mathrm{NL}}(\xi) \langle k-l \rangle^\frac13 \rt. \\
				& \lt. \qquad\qquad\qquad\qquad  \times \frac{\eta(k-l)-l(\xi-\eta)}{|l|^2+|\eta|^2} \hat{\omega}_l^{\mathrm{L}}(\eta) \hat{\theta }_{k-l}^{\mathrm{L}}(\xi-\eta) d k d l d \xi d \eta\right| \\
				& \lesssim\|e^{c \nu^{\frac{1}{3}} \lambda(k) t}\langle k\rangle \langle \frac1k \rangle^\epsilon \langle k \rangle^\frac13  \hat{\theta}_k^{\mathrm{NL}}(\xi)\|_{L_{k, \xi}^2}\langle t\rangle^{-2}\|e^{c \nu^{\frac{1}{3}} \lambda(l) t}\langle l, \eta+l t\rangle^4\langle\frac{1}{l}\rangle^4 \hat{\omega}_l^{\mathrm{L}}(\eta)\|_{L_{l, \eta}^2} \\
				& \quad \times\|\langle k-l, \xi-\eta+(k-l) t\rangle^{-2}\|_{L_{k-l, \xi-\eta}^2} \\
				& \quad \times\|e^{c \nu^{\frac{1}{3}} \lambda(k-l) t}\langle k-l, \xi-\eta+(k-l) t\rangle^4 \langle k-l \rangle^\frac13  \hat{\theta}_{k-l}^{\mathrm{L}}(\xi-\eta)\|_{L_{k-l, \xi-\eta}^2} \\
				& \lesssim\langle t\rangle^{-2} \|e^{c \nu^{\frac{1}{3}} \lambda\left(D_x\right) t}\langle D_x\rangle  \langle \frac{1}{D_x}\rangle^\epsilon \langle D_x \rangle^\frac13  \theta^{\mathrm{NL}}\|_{L^2}    \|e^{c \nu^{\frac{1}{3}} \lambda\left(D_x\right) t}\langle D_x, D_y + tD_x\rangle^4\langle\frac{1}{D_x}\rangle^4 \omega^{\mathrm{L}} \|_{L^2} \\
				&\quad \times \|e^{c \nu^{\frac{1}{3}} \lambda\left(D_x\right) t}\langle D_x, D_y + tD_x\rangle^4\langle\frac{1}{D_x}\rangle^4 \langle D_x \rangle^\frac13   \theta^{\mathrm{L}} \|_{L^2}.
			\end{aligned}
		\end{equation}
		$\bullet$~Case 2:  $|k| < 1$. By \eqref{eq:0}, \eqref{eq:the point-wise inviscid damping estimate} and 
		\begin{equation*} 
			\langle k \rangle \lesssim \langle l, \eta+l t\rangle\langle k-l, \xi-\eta+(k-l) t\rangle,
		\end{equation*}
		we have that
		\begin{align}  \label{eq:temp10}
			&\no \left|\int_{|k| < 1}  \mathcal{M}(t, k, \xi) e^{2 c \nu^{\frac{1}{3}} \lambda(k) t}\langle k\rangle^2\langle\frac{1}{k}\rangle^{2 \epsilon} \langle k \rangle^\frac13  \hat{\theta }_k^{\mathrm{NL}}(\xi) \langle k-l \rangle^\frac13 \rt. \\
			&\no \lt. \qquad\qquad\qquad\qquad\qquad\qquad  \times \frac{\eta(k-l)-l(\xi-\eta)}{|l|^2+|\eta|^2} \hat{\omega}_l^{\mathrm{L}}(\eta) \hat{\theta }_{k-l}^{\mathrm{L}}(\xi-\eta) d k d l d \xi d \eta\right| \\
			& \no\lesssim \| \langle\frac{1}{k}\rangle^{ \epsilon} \|_{L_k^2 ([-1, 1])}\|e^{c \nu^{\frac{1}{3}} \lambda(k) t}\langle k\rangle \langle \frac1k \rangle^\epsilon \langle k \rangle^\frac13  \hat{\theta}_k^{\mathrm{NL}}(\xi)\|_{L_{k, \xi}^2}\langle t\rangle^{-2}\|e^{c \nu^{\frac{1}{3}} \lambda(l) t}\langle l, \eta+l t\rangle^4\langle\frac{1}{l}\rangle^4 \hat{\omega}_l^{\mathrm{L}}(\eta)\|_{L_{l, \eta}^2} \\
			&\no \quad \times\|\langle k-l, \xi-\eta+(k-l) t\rangle^{-1}\|_{L_{ \xi-\eta}^2} \\
			&\no \quad \times\|e^{c \nu^{\frac{1}{3}} \lambda(k-l) t}\langle k-l, \xi-\eta+(k-l) t\rangle^3 \langle k-l \rangle^\frac13  \hat{\theta}_{k-l}^{\mathrm{L}}(\xi-\eta)\|_{L_{k-l, \xi-\eta}^2} \\
			&\no \lesssim\langle t\rangle^{-2} \|e^{c \nu^{\frac{1}{3}} \lambda\left(D_x\right) t}\langle D_x\rangle  \langle \frac{1}{D_x}\rangle^\epsilon \langle D_x \rangle^\frac13  \theta^{\mathrm{NL}}\|_{L^2}    \|e^{c \nu^{\frac{1}{3}} \lambda\left(D_x\right) t}\langle D_x, D_y + tD_x\rangle^4\langle\frac{1}{D_x}\rangle^4 \omega^{\mathrm{L}} \|_{L^2} \\
			&\quad \times \|e^{c \nu^{\frac{1}{3}} \lambda\left(D_x\right) t}\langle D_x, D_y + tD_x\rangle^4\langle\frac{1}{D_x}\rangle^4 \langle D_x \rangle^\frac13  \theta^{\mathrm{L}} \|_{L^2}.
		\end{align}
		Hence, \eqref{eq:temp8}, \eqref{eq:temp9} and \eqref{eq:temp10} gives
		\begin{equation}\begin{aligned} \label{eq:sum1}
				&\left|\left(u^{\mathrm{L}} \cdot \nabla \langle D_x \rangle^\frac13 \theta^{\mathrm{L}} \big| \mathcal{M} e^{2 c \nu^{\frac{1}{3}} \lambda\left(D_x\right) t}\langle D_x\rangle^2\langle\frac{1}{D_x}\rangle^{2 \epsilon} \langle D_x \rangle^\frac13 \theta^{\mathrm{NL}}\right)\right| \\
				\lesssim	& \langle t\rangle^{-2} \|e^{c \nu^{\frac{1}{3}} \lambda\left(D_x\right) t}\langle D_x\rangle  \langle \frac{1}{D_x}\rangle^\epsilon \langle D_x \rangle^\frac13  \theta^{\mathrm{NL}}\|_{L^2}    \|e^{c \nu^{\frac{1}{3}} \lambda\left(D_x\right) t}\langle D_x, D_y + tD_x\rangle^4\langle\frac{1}{D_x}\rangle^4 \omega^{\mathrm{L}} \|_{L^2} \\
				&\quad \times \|e^{c \nu^{\frac{1}{3}} \lambda\left(D_x\right) t}\langle D_x, D_y + tD_x\rangle^4\langle\frac{1}{D_x}\rangle^4 \langle D_x \rangle^\frac13  \theta^{\mathrm{L}} \|_{L^2}.
		\end{aligned}\end{equation}

		\underline{\bf Step II: Estimate of $ |D_x|^\frac13 u^{\mathrm{L}} \cdot  \nabla  \theta^{\mathrm{L}}$.} 
		Write
		\begin{equation*}\begin{aligned} \label{eq:temp11}
				& \left|\left( |D_x|^\frac13 u^{\mathrm{L}} \cdot \nabla  \theta^{\mathrm{L}} \big| D_x\rangle^2\langle\frac{1}{D_x}\rangle^{2 \epsilon} |D_x|^\frac13 \theta^{\mathrm{NL}}\right)\right| \\
				= & \left|\int_{\mathbb{R}^4} \mathcal{M}(t, k, \xi) e^{2 c \nu^{\frac{1}{3}} \lambda(k) t}\langle k\rangle^2\langle\frac{1}{k}\rangle^{2 \epsilon} |k|^\frac13  \hat{\theta }_k^{\mathrm{NL}}(\xi) |l|^\frac13 \rt. \\
				& \lt. \quad \times  \frac{\eta(k-l)-l(\xi-\eta)}{|l|^2+|\eta|^2} \hat{\omega}_l^{\mathrm{L}}(\eta) \hat{\theta }_{k-l}^{\mathrm{L}}(\xi-\eta) d k d l d \xi d \eta\right|.
		\end{aligned}\end{equation*}
		\eqref{eq:the point-wise inviscid damping estimate} is substituted by 
		\begin{equation*} 
			\frac{|l|^\frac13}{|l|^2+|\eta|^2} \lesssim \frac{\langle l, \eta+l t\rangle^2}{|l|^{\frac{11}{3}}\langle t\rangle^2} \lesssim\langle t\rangle^{-2}\langle\frac{1}{l}\rangle^4\langle l, \eta+l t\rangle^2.
		\end{equation*}
		Then, similarly as \eqref{eq:sum1}, we have
		\begin{equation*}\begin{aligned} \label{eq:sum2}
				&\left|\left( |D_x|^\frac13 u^{\mathrm{L}} \cdot \nabla  \theta^{\mathrm{L}} \left\lvert\, \mathcal{M} e^{2 c \nu^{\frac{1}{3}} \lambda\left(D_x\right) t}\langle D_x\rangle^2\langle\frac{1}{D_x}\rangle^{2 \epsilon} |D_x|^\frac13 \theta^{\mathrm{NL}}\right.\right)\right| \\
				& \lesssim\langle t\rangle^{-2} \|e^{c \nu^{\frac{1}{3}} \lambda\left(D_x\right) t}\langle D_x\rangle  \langle \frac{1}{D_x}\rangle^\epsilon |D_x|^\frac13  \theta^{\mathrm{NL}}\|_{L^2}    \|e^{c \nu^{\frac{1}{3}} \lambda\left(D_x\right) t}\langle D_x, D_y + tD_x\rangle^4 \langle\frac{1}{D_x}\rangle^4 \omega^{\mathrm{L}} \|_{L^2} \\
				&\quad \times \|e^{c \nu^{\frac{1}{3}} \lambda\left(D_x\right) t}\langle D_x, D_y + tD_x\rangle^4\langle\frac{1}{D_x}\rangle^4   \theta^{\mathrm{L}} \|_{L^2}.
		\end{aligned}\end{equation*}

		\underline{\bf Step III: Estimate of $  u^{\mathrm{L}} \cdot  \nabla \langle D_x \rangle^\frac13 \theta^{\mathrm{NL}}$.} Due to the different decay behaviors of $u_1^{\mathrm{L}} \px    $ and $u_2^{\mathrm{L}} \py   $, we consider them separately.
		
		\underline{ Step III.1: Estimate of $u_2^{\mathrm{L}} \py \langle D_x \rangle^\frac13 \theta^{\mathrm{NL}}  $.} 
		The Fourier transform shows that
		\begin{equation}\begin{aligned} \label{eq:temp14}
				& \left|\left(  u_2^{\mathrm{L}} \py \langle D_x \rangle^\frac13 \theta^{\mathrm{NL}} \left\lvert\, \mathcal{M} e^{2 c \nu^{\frac{1}{3}} \lambda\left(D_x\right) t}\langle D_x\rangle^2\langle\frac{1}{D_x}\rangle^{2 \epsilon} \langle D_x \rangle^\frac13 \theta^{\mathrm{NL}}\right.\right)\right| \\
				= & \left|\int_{\mathbb{R}^4} \mathcal{M}(t, k, \xi) e^{2 c \nu^{\frac{1}{3}} \lambda(k) t}\langle k\rangle^2\langle\frac{1}{k}\rangle^{2 \epsilon} \langle k \rangle^\frac13  \hat{\theta }_k^{\mathrm{NL}}(\xi) \langle k-l \rangle^\frac13 \frac{l(\xi-\eta)}{|l|^2+|\eta|^2} \hat{\omega}_l^{\mathrm{L}}(\eta) \hat{\theta }_{k-l}^{\mathrm{NL}}(\xi-\eta) d k d l d \xi d \eta\right|.
		\end{aligned}\end{equation}
		$\bullet$~Case 1:   $|k| \geq  1$. From $\langle \frac1k \rangle^\epsilon\langle k \rangle \lesssim \langle k- l \rangle \langle  l, \eta + lt \rangle $ and 
		\begin{equation} \label{eq:trick2}
			l(|l|^2+|\eta|^2)^{-1}   \lesssim\langle t\rangle^{-2}\langle\frac{1}{l}\rangle^3 \langle l, \eta+l t\rangle^2,
		\end{equation}
		there holds
		\begin{equation} \label{eq:temp14.5}
			\begin{aligned}
				& \left|\int_{|k| \geq 1} \mathcal{M}(t, k, \xi) e^{2 c \nu^{\frac{1}{3}} \lambda(k) t}\langle k\rangle^2\langle\frac{1}{k}\rangle^{2 \epsilon} \langle k \rangle^\frac13  \hat{\theta }_k^{\mathrm{NL}}(\xi) \langle k-l \rangle^\frac13 \frac{l(\xi-\eta)}{|l|^2+|\eta|^2} \hat{\omega}_l^{\mathrm{L}}(\eta) \hat{\theta }_{k-l}^{\mathrm{NL}}(\xi-\eta) d k d l d \xi d \eta\right| \\
				& \lesssim\|e^{c \nu^{\frac{1}{3}} \lambda(k) t}\langle k\rangle\langle\frac{1}{k}\rangle^\epsilon \langle k \rangle^\frac13  \hat{\theta}_k^{\mathrm{NL}}(\xi)\|_{L_{k, \xi}^2}\langle t\rangle^{-2}\|\langle l, \eta+l t\rangle^{-2}\|_{L_{l, \eta}^2}\|e^{c \nu^{\frac{1}{3}} \lambda(l) t}\langle l, \eta+l t\rangle^5\langle\frac{1}{l}\rangle^3 \hat{\omega}_l^{\mathrm{L}}(\eta)\|_{L_{l, \eta}^2} \\
				& \quad \times\|e^{c \nu^{\frac{1}{3}} \lambda(k-l) t}\langle k-l\rangle(\xi-\eta) \langle k-l \rangle^\frac13 \hat{\theta}_{k-l}^{\mathrm{NL}}(\xi-\eta)\|_{L_{k-l, \xi-\eta}^2} \\
				& \lesssim\langle t\rangle^{-2}\|e^{c \nu^{\frac{1}{3}} \lambda\left(D_x\right) t}\langle D_x\rangle\langle\frac{1}{D_x}\rangle^\epsilon \langle D_x \rangle^\frac13  \theta^{\mathrm{NL}}\|_{L^2}\|e^{c \nu^{\frac{1}{3}} \lambda\left(D_x\right) t}\langle D_x, D_y+t D_x\rangle^5\langle\frac{1}{D_x}\rangle^3 \omega^{\mathrm{L}}\|_{L^2} \\
				& \quad \times\|e^{c \nu^{\frac{1}{3}} \lambda\left(D_x\right) t}\langle D_x\rangle \partial_y   \langle D_x \rangle^\frac13 \theta^{\mathrm{NL}}\|_{L^2}.
			\end{aligned}
		\end{equation}
		$\bullet$~Case 2:   $|k| < 1$. According to \eqref{eq:trick2} and $\langle k \rangle^2  \lesssim 1$, we arrive
		\begin{equation} \label{eq:temp15}
			\begin{aligned}
				& \left|\int_{|k| < 1}  \mathcal{M}(t, k, \xi) e^{2 c \nu^{\frac{1}{3}} \lambda(k) t}\langle k\rangle^2\langle\frac{1}{k}\rangle^{2 \epsilon} \langle k \rangle^\frac13  \hat{\theta }_k^{\mathrm{NL}}(\xi) \langle k-l \rangle^\frac13 \frac{l(\xi-\eta)}{|l|^2+|\eta|^2} \hat{\omega}_l^{\mathrm{L}}(\eta) \hat{\theta }_{k-l}^{\mathrm{NL}}(\xi-\eta)   d k d l d \xi d \eta\right| \\
				& \lesssim \| \langle\frac{1}{k}\rangle^{ \epsilon} \|_{L_k^2 ([-1, 1])}\|e^{c \nu^{\frac{1}{3}} \lambda(k) t}\langle k\rangle \langle \frac1k \rangle^\epsilon \langle k \rangle^\frac13  \hat{\theta}_k^{\mathrm{NL}}(\xi)\|_{L_{k, \xi}^2}\langle t\rangle^{-2}\|e^{c \nu^{\frac{1}{3}} \lambda(l) t}\langle l, \eta+l t\rangle^3\langle\frac{1}{l}\rangle^3 \hat{\omega}_l^{\mathrm{L}}(\eta)\|_{L_{l, \eta}^2} \\
				& \quad \times\|\langle \eta + lt\rangle^{-1}\|_{L_{ \eta}^2} 
				\|e^{c \nu^{\frac{1}{3}} \lambda(k-l) t}  |\xi - \eta| \langle k-l \rangle^\frac13  \hat{\theta}_{k-l}^{\mathrm{NL}}(\xi-\eta)\|_{L_{k-l, \xi-\eta}^2} \\
				& \lesssim\langle t\rangle^{-2} \|e^{c \nu^{\frac{1}{3}} \lambda\left(D_x\right) t}\langle D_x\rangle  \langle \frac{1}{D_x}\rangle^\epsilon \langle D_x \rangle^\frac13  \theta^{\mathrm{NL}}\|_{L^2}    \|e^{c \nu^{\frac{1}{3}} \lambda\left(D_x\right) t}\langle D_x, D_y + tD_x\rangle^3\langle\frac{1}{D_x}\rangle^3 \omega^{\mathrm{L}} \|_{L^2} \\
				&\quad \times \|e^{c \nu^{\frac{1}{3}} \lambda\left(D_x\right) t} \py \langle D_x \rangle^\frac13  \theta^{\mathrm{NL}} \|_{L^2}.
			\end{aligned}
		\end{equation}
		Adding \eqref{eq:temp14}, \eqref{eq:temp14.5} and \eqref{eq:temp15} together, we infer that
		\begin{equation}\begin{aligned} \label{eq:temp15.5}
				&\left|\left(  u_2^{\mathrm{L}} \py \langle D_x \rangle^\frac13 \theta^{\mathrm{NL}} \big| \mathcal{M} e^{2 c \nu^{\frac{1}{3}} \lambda\left(D_x\right) t}\langle D_x\rangle^2\langle\frac{1}{D_x}\rangle^{2 \epsilon} \langle D_x \rangle^\frac13 \theta^{\mathrm{NL}}\right)\right|\\
				&\lesssim \langle t\rangle^{-2} \|e^{c \nu^{\frac{1}{3}} \lambda\left(D_x\right) t}\langle D_x\rangle  \langle \frac{1}{D_x}\rangle^\epsilon \langle D_x \rangle^\frac13  \theta^{\mathrm{NL}}\|_{L^2}    \|e^{c \nu^{\frac{1}{3}} \lambda\left(D_x\right) t}\langle D_x, D_y + tD_x\rangle^ 5 \langle\frac{1}{D_x}\rangle^3 \omega^{\mathrm{L}} \|_{L^2} \\
				&\quad \times \|e^{c \nu^{\frac{1}{3}} \lambda\left(D_x\right) t} \py \langle D_x \rangle^\frac13  \theta^{\mathrm{NL}} \|_{L^2}.
		\end{aligned}\end{equation}
		
		\underline{ Step III.2: Estimate of $u_1^{\mathrm{L}} \px \langle D_x \rangle^\frac13 \theta^{\mathrm{NL}}  $.} 
		Note that
		\begin{equation}\begin{aligned} \label{eq:temp16}
				& \left|\left(  u_1^{\mathrm{L}} \px \langle D_x \rangle^\frac13 \theta^{\mathrm{NL}} \big| \mathcal{M} e^{2 c \nu^{\frac{1}{3}} \lambda\left(D_x\right) t}\langle D_x\rangle^2\langle\frac{1}{D_x}\rangle^{2 \epsilon} \langle D_x \rangle^\frac13 \theta^{\mathrm{NL}}\right)\right| \\
				= & \left|\int_{\mathbb{R}^4} \mathcal{M}(t, k, \xi) e^{2 c \nu^{\frac{1}{3}} \lambda(k) t}\langle k\rangle^2\langle\frac{1}{k}\rangle^{2 \epsilon} \langle k \rangle^\frac13  \hat{\theta }_k^{\mathrm{NL}}(\xi) \langle k-l \rangle^\frac13 \frac{\eta(k-l)}{|l|^2+|\eta|^2} \hat{\omega}_l^{\mathrm{L}}(\eta) \hat{\theta }_{k-l}^{\mathrm{NL}}(\xi-\eta) d k d l d \xi d \eta\right|.
		\end{aligned}\end{equation}
		We decompose the integration domain into two parts: the \emph{$k$-near-diagonal region}, where $|k - l| \leq 2|k|$, and the \emph{$k$-off-diagonal region}, where $|k - l| > 2|k|$. \\
		$\bullet$~Case 1:  $|k-l|\leq 2|k|$. Using $\langle k \rangle \langle \frac1k\rangle^\epsilon \lesssim \langle k- l\rangle \langle  l, \eta + lt\rangle \langle \frac1{k-l}\rangle^\epsilon$, $|k-l|^\frac13 \lesssim |k|^\frac13 $ and
		\begin{equation} \label{eq:trick3}
			\eta  (|l|^2+|\eta|^2)^{-1}   \lesssim\langle t\rangle^{-1}\langle\frac{1}{l}\rangle^2 \langle l, \eta+l t\rangle,
		\end{equation}
		one has
		\begin{align} \label{eq:temp17}
			&\no \left|\int_{|k-l| \leq 2|k|} \mathcal{M}(t, k, \xi) e^{2 c \nu^{\frac{1}{3}} \lambda(k) t}\langle k\rangle^2\langle\frac{1}{k}\rangle^{2 \epsilon} \langle k \rangle^\frac13  \hat{\theta }_k^{\mathrm{NL}}(\xi) \langle k-l \rangle^\frac13 \frac{\eta(k-l)}{|l|^2+|\eta|^2} \hat{\omega}_l^{\mathrm{L}}(\eta) \hat{\theta }_{k-l}^{\mathrm{NL}}(\xi-\eta) d k d l d \xi d \eta\right| \\
			&\no \lesssim\|e^{c \nu^{\frac{1}{3}} \lambda(k) t}\langle k\rangle\langle\frac{1}{k}\rangle^\epsilon|k|^{\frac{1}{3}} \langle k \rangle^\frac13 \hat{\theta}_k^{\mathrm{NL}}(\xi)\|_{L_{k, \xi}^2}\langle t\rangle^{-1}\|\langle l, \eta+l t\rangle^{-2}\|_{L_{l, \eta}^2}\|e^{c \nu^{\frac{1}{3}} \lambda(l) t}\langle l, \eta+l t\rangle^4\langle\frac{1}{l}\rangle^2 \hat{\omega}_l^{\mathrm{L}}(\eta)\|_{L_{l, \eta}^2} \\
			&\no \quad \times\|e^{c \nu^{\frac{1}{3}} \lambda(k-l) t}\langle k-l\rangle\langle\frac{1}{k-l}\rangle^\epsilon|k-l|^\frac23 \langle k-l \rangle^\frac13 \hat{\theta}_{k-l}^{\mathrm{NL}}(\xi-\eta)\|_{L_{k-l, \xi-\eta}^2} \\
			&\no \lesssim\|e^{c \nu^{\frac{1}{3}} \lambda(k) t}\langle k\rangle\langle\frac{1}{k}\rangle^\epsilon|k|^{\frac{1}{3}}  \langle k \rangle^\frac13 \hat{\theta}_k^{\mathrm{NL}}(\xi)\|_{L_{k, \xi}^2}\langle t\rangle^{-1}\|\langle l, \eta+l t\rangle^{-2}\|_{L_{l, \eta}^2}\|e^{c \nu^{\frac{1}{3}} \lambda(l) t}\langle l, \eta+l t\rangle^4\langle\frac{1}{l}\rangle^2 \hat{\omega}_l^{\mathrm{L}}(\eta)\|_{L_{l, \eta}^2} \\
			&\no \quad \times\|e^{c \nu^{\frac{1}{3}} \lambda(k-l) t}\langle k-l\rangle\langle\frac{1}{k-l}\rangle^\epsilon|k-l|^\frac13  \langle k-l \rangle^\frac13 \hat{\theta}_{k-l}^{\mathrm{NL}}(\xi-\eta)\|_{L_{k-l, \xi-\eta}^2}^\frac12 \\
			&\no\quad \times \|e^{c \nu^{\frac{1}{3}} \lambda(k-l) t}\langle k-l\rangle\langle\frac{1}{k-l}\rangle^\epsilon|k-l| \langle k-l \rangle^\frac13 \hat{\theta}_{k-l}^{\mathrm{NL}}(\xi-\eta)\|_{L_{k-l, \xi-\eta}^2}^\frac12  \\
			&\no \lesssim \langle t\rangle^{-1} \|e^{c \nu^{\frac{1}{3}} \lambda\left(D_x\right) t}\langle D_x, D_y+t D_x\rangle^4\langle\frac{1}{D_x}\rangle^2 \omega^{\mathrm{L}}\|_{L^2}  \|e^{c \nu^{\frac{1}{3}} \lambda\left(D_x\right) t}\langle D_x\rangle\langle\frac{1}{D_x}\rangle^\epsilon\left|D_x\right|^{\frac{1}{3}} \langle D_x \rangle^\frac13 \theta^{\mathrm{NL}}\|_{L^2}^{\frac{3}{2}} \\
			& \quad \times  \|e^{c \nu^{\frac{1}{3}} \lambda\left(D_x\right) t}\langle D_x\rangle\langle\frac{1}{D_x}\rangle^\epsilon \px \langle D_x \rangle^\frac13  \theta^{\mathrm{NL}}\|_{L^2}^{\frac{1}{2}} .
		\end{align}
		$\bullet$~Case 2:  $2|k|<|k-l|$. In this case, it follows from $|l| \sim |k-l|$ that $\langle k\rangle \lesssim\langle l\rangle$. By $\langle k \rangle^2 \lesssim \langle  l, \eta + lt\rangle \langle k-l \rangle$,  $|k-l| \lesssim \langle  l, \eta + lt\rangle |k-l|^\frac13 $ and \eqref{eq:trick3}, we have
		\begin{equation}\begin{aligned} \label{eq:temp18}
				&\left|\int_{2|k|<|k-l|} \mathcal{M}(t, k, \xi) e^{2 c \nu^{\frac{1}{3}} \lambda(k) t}\langle k\rangle^2\langle\frac{1}{k}\rangle^{2 \epsilon} \langle k \rangle^\frac13  \hat{\theta }_k^{\mathrm{NL}}(\xi) \langle k-l \rangle^\frac13 \frac{\eta(k-l)}{|l|^2+|\eta|^2} \hat{\omega}_l^{\mathrm{L}}(\eta) \hat{\theta }_{k-l}^{\mathrm{NL}}(\xi-\eta) d k d l d \xi d \eta\right| \\
				&\lesssim\|\langle k\rangle^{-1}\langle\frac{1}{k}\rangle^\epsilon\|_{L_k^2}\|e^{c \nu^{\frac{1}{3}} \lambda(k) t}\langle k\rangle\langle\frac{1}{k}\rangle^\epsilon \langle k \rangle^\frac13 \hat{\theta}_k^{\mathrm{NL}}(\xi)\|_{L_{k, \xi}^2}\langle t\rangle^{-1}\|e^{c \nu^{\frac{1}{3}} \lambda(l) t}\langle l, \eta+l t\rangle^4\langle\frac{1}{l}\rangle^2 \hat{\omega}_l^{\mathrm{L}}(\eta)\|_{L_{l, \eta}^2} \\
				&\quad \times\|\langle\eta+l t\rangle^{-1}\|_{L_\eta^2}\|e^{c \nu^{\frac{1}{3}} \lambda(k-l) t} \langle k-l \rangle   |k-l|^{\frac{1}{3}} \langle k-l \rangle^\frac13  \hat{\theta}_{k-l}^{\mathrm{NL}}(\xi-\eta)\|_{L_{k-l, \xi-\eta}^2} \\
				& \lesssim \langle t\rangle^{-1}\|e^{c \nu^{\frac{1}{3}} \lambda\left(D_x\right) t}\langle D_x\rangle\langle\frac{1}{D_x}\rangle^\epsilon \langle D_x \rangle^\frac13 \theta^{\mathrm{NL}}\|_{L^2}\|e^{c \nu^{\frac{1}{3}} \lambda\left(D_x\right) t}\langle D_x, D_y+t D_x\rangle^4\langle\frac{1}{D_x}\rangle^2 \omega^{\mathrm{L}}\|_{L^2} \\
				& \quad \times  \| e^{c \nu^{\frac{1}{3}} \lambda\left(D_x\right) t} \langle D_x \rangle
				\left|D_x\right|^{\frac{1}{3}}  \langle D_x \rangle^\frac13  \theta^{\mathrm{NL}} \|_{L^2} .
		\end{aligned}\end{equation}
		It follows from \eqref{eq:temp16}, \eqref{eq:temp17} and \eqref{eq:temp18} that
		\begin{equation*}\begin{aligned}
				&\left|\left(  u_1^{\mathrm{L}} \px \langle D_x \rangle^\frac13 \theta^{\mathrm{NL}} \big| \mathcal{M} e^{2 c \nu^{\frac{1}{3}} \lambda\left(D_x\right) t}\langle D_x\rangle^2\langle\frac{1}{D_x}\rangle^{2 \epsilon} \langle D_x \rangle^\frac13 \theta^{\mathrm{NL}}\right)\right|\\
				&\lesssim \langle t\rangle^{-1}     \|e^{c \nu^{\frac{1}{3}} \lambda\left(D_x\right) t}\langle D_x, D_y+t D_x\rangle^4\langle\frac{1}{D_x}\rangle^2 \omega^{\mathrm{L}}\|_{L^2} \\
				& \quad \times (\|e^{c \nu^{\frac{1}{3}} \lambda\left(D_x\right) t}\langle D_x\rangle\langle\frac{1}{D_x}\rangle^\epsilon \langle D_x \rangle^\frac13 \theta^{\mathrm{NL}}\|_{L^2}  \| e^{c \nu^{\frac{1}{3}} \lambda\left(D_x\right) t} \langle D_x \rangle
				\left|D_x\right|^{\frac{1}{3}} \langle D_x \rangle^\frac13 \theta^{\mathrm{NL}} \|_{L^2} \\
				&\qquad + \|e^{c \nu^{\frac{1}{3}} \lambda\left(D_x\right) t}\langle D_x\rangle\langle\frac{1}{D_x}\rangle^\epsilon\left|D_x\right|^{\frac{1}{3}}  \langle D_x \rangle^\frac13 \theta^{\mathrm{NL}}\|_{L^2}^{\frac{3}{2}}   \|e^{c \nu^{\frac{1}{3}} \lambda\left(D_x\right) t}\langle D_x\rangle\langle\frac{1}{D_x}\rangle^\epsilon \px \langle D_x \rangle ^\frac13  \theta^{\mathrm{NL}}\|_{L^2}^{\frac{1}{2}}),
		\end{aligned}\end{equation*}
		which along with \eqref{eq:temp15.5} implies
		\begin{equation}\begin{aligned} \label{eq:sum3}
				&\left|\left(  u^{\mathrm{L}} \cdot\nabla \langle D_x \rangle^\frac13 \theta^{\mathrm{NL}} \big| \mathcal{M} e^{2 c \nu^{\frac{1}{3}} \lambda\left(D_x\right) t}\langle D_x\rangle^2\langle\frac{1}{D_x}\rangle^{2 \epsilon} \langle D_x \rangle^\frac13 \theta^{\mathrm{NL}}\right)\right| \\
				&\lesssim \langle t\rangle^{-2} \|e^{c \nu^{\frac{1}{3}} \lambda\left(D_x\right) t}\langle D_x\rangle  \langle \frac{1}{D_x}\rangle^\epsilon \langle D_x \rangle^\frac13  \theta^{\mathrm{NL}}\|_{L^2}    \|e^{c \nu^{\frac{1}{3}} \lambda\left(D_x\right) t}\langle D_x, D_y + tD_x\rangle^ 5 \langle\frac{1}{D_x}\rangle^3 \omega^{\mathrm{L}} \|_{L^2} \\
				&\quad \times \|e^{c \nu^{\frac{1}{3}} \lambda\left(D_x\right) t} \py \langle D_x \rangle^\frac13  \theta^{\mathrm{NL}} \|_{L^2} \\
				&\quad +  \langle t\rangle^{-1}     \|e^{c \nu^{\frac{1}{3}} \lambda\left(D_x\right) t}\langle D_x, D_y+t D_x\rangle^4\langle\frac{1}{D_x}\rangle^2 \omega^{\mathrm{L}}\|_{L^2} \\
				& \quad \times (\|e^{c \nu^{\frac{1}{3}} \lambda\left(D_x\right) t}\langle D_x\rangle\langle\frac{1}{D_x}\rangle^\epsilon \langle D_x \rangle^\frac13 \theta^{\mathrm{NL}}\|_{L^2}  \| e^{c \nu^{\frac{1}{3}} \lambda\left(D_x\right) t} \langle D_x \rangle
				\left|D_x\right|^{\frac{1}{3}}  \langle D_x \rangle ^\frac13 \theta^{\mathrm{NL}} \|_{L^2} \\
				&\qquad + \|e^{c \nu^{\frac{1}{3}} \lambda\left(D_x\right) t}\langle D_x\rangle\langle\frac{1}{D_x}\rangle^\epsilon\left|D_x\right|^{\frac{1}{3}}  \langle D_x \rangle^\frac13  \theta^{\mathrm{NL}}\|_{L^2}^{\frac{3}{2}}   \|e^{c \nu^{\frac{1}{3}} \lambda\left(D_x\right) t}\langle D_x\rangle\langle\frac{1}{D_x}\rangle^\epsilon \px \langle D_x \rangle^\frac13  \theta^{\mathrm{NL}}\|_{L^2}^{\frac{1}{2}}).
		\end{aligned}\end{equation}

		\underline{\bf Step IV: Estimate of $ |D_x|^\frac13 u^{\mathrm{L}} \cdot  \nabla  \theta^{\mathrm{NL}}$.} 
		Writing \eqref{eq:trick2} and \eqref{eq:trick3} as $|l|^\frac43(|l|^2+|\eta|^2)^{-1}   \lesssim\langle t\rangle^{-2}\langle\frac{1}{l}\rangle^3 \langle l, \eta+l t\rangle^2$ and $ |l|^\frac13 \eta  (|l|^2+|\eta|^2)^{-1}   \lesssim\langle t\rangle^{-1}\langle\frac{1}{l}\rangle^2 \langle l, \eta+l t\rangle$, respectively, 
		we infer similarly as \eqref{eq:sum3} that
		\begin{equation*}\begin{aligned} \label{eq:sum4}
				&\left|\left( |D_x|^\frac13  u^{\mathrm{L}} \cdot\nabla \theta^{\mathrm{NL}} \big| \mathcal{M} e^{2 c \nu^{\frac{1}{3}} \lambda\left(D_x\right) t}\langle D_x\rangle^2\langle\frac{1}{D_x}\rangle^{2 \epsilon} |D_x|^\frac13 \theta^{\mathrm{NL}} \right)\right| \\
				&\lesssim \langle t\rangle^{-2} \|e^{c \nu^{\frac{1}{3}} \lambda\left(D_x\right) t}\langle D_x\rangle  \langle \frac{1}{D_x}\rangle^\epsilon |D_x|^\frac13  \theta^{\mathrm{NL}}\|_{L^2}    \|e^{c \nu^{\frac{1}{3}} \lambda\left(D_x\right) t}\langle D_x, D_y + tD_x\rangle^ 5 \langle\frac{1}{D_x}\rangle^3 \omega^{\mathrm{L}} \|_{L^2} \\
				&\quad \times \|e^{c \nu^{\frac{1}{3}} \lambda\left(D_x\right) t} \py  \theta^{\mathrm{NL}} \|_{L^2} \\
				&\quad +  \langle t\rangle^{-1}     \|e^{c \nu^{\frac{1}{3}} \lambda\left(D_x\right) t}\langle D_x, D_y+t D_x\rangle^4\langle\frac{1}{D_x}\rangle^2 \omega^{\mathrm{L}}\|_{L^2} \\
				& \quad \times \left(\|e^{c \nu^{\frac{1}{3}} \lambda\left(D_x\right) t}\langle D_x\rangle\langle\frac{1}{D_x}\rangle^\epsilon |D_x|^\frac13 \theta^{\mathrm{NL}}\|_{L^2} ^2  + \|e^{c \nu^{\frac{1}{3}} \lambda\left(D_x\right) t}\langle D_x\rangle\langle\frac{1}{D_x}\rangle^\epsilon\left|D_x\right|^{\frac{1}{3}} \theta^{\mathrm{NL}}\|_{L^2} \rt.\\
				&\qquad \quad \lt.\times \|e^{c \nu^{\frac{1}{3}} \lambda\left(D_x\right) t}\langle D_x\rangle\langle\frac{1}{D_x}\rangle^\epsilon\left|D_x\right|^{\frac{2}{3}} \theta^{\mathrm{NL}}\|_{L^2}^{\frac{1}{2}}   \|e^{c \nu^{\frac{1}{3}} \lambda\left(D_x\right) t}\langle D_x\rangle\langle\frac{1}{D_x}\rangle^\epsilon \px |D_x|^\frac13  \theta^{\mathrm{NL}}\|_{L^2}^{\frac{1}{2}}\right).
		\end{aligned}\end{equation*}

		\underline{\bf Step V: Estimate of $  u^{\mathrm{NL}} \cdot  \nabla \langle D_x \rangle^\frac13 \theta^{\mathrm{L}}$.} For this term, we decompose it into two parts as follows:
		\begin{equation}\begin{aligned} \label{eq:I1I2}
				& \left|\left(u^{\mathrm{NL}} \cdot \nabla \langle D_x \rangle^\frac13 \theta^{\mathrm{L}} \big| \mathcal{M} e^{2 c \nu^{\frac{1}{3}} \lambda\left(D_x\right) t}\langle D_x\rangle^2\langle\frac{1}{D_x}\rangle^{2 \epsilon} \langle D_x \rangle^\frac13 \theta^{\mathrm{NL}} \right)\right| \\
				= & \left|\int_{\mathbb{R}^4} \mathcal{M}(t, k, \xi) e^{2 c \nu^{\frac{1}{3}} \lambda(k) t}\langle k\rangle^2\langle\frac{1}{k}\rangle^{2 \epsilon} \langle k \rangle^\frac13\hat{\theta}_k^{\mathrm{NL}}(\xi) \rt. \\
				&\lt. \times \langle k-l \rangle^\frac13  \frac{\eta(k-l)-l(\xi-\eta)}{|l|^2+|\eta|^2} \hat{\omega}_l^{\mathrm{NL}}(\eta) \hat{\theta}_{k-l}^{\mathrm{L}}(\xi-\eta) d k d l d \xi d \eta\right| \\
				\leq & I_1+I_2,
		\end{aligned}\end{equation}
		where $I_1$ and $I_2$ are defined by
		$$
		\begin{gathered}
			I_1 :=\bigg| \int_{\mathbb{R}^4} \mathcal{M}(t, k, \xi) e^{2 c \nu^{\frac{1}{3}} \lambda(k) t}\langle k\rangle^2\langle\frac{1}{k}\rangle^{2 \epsilon} \langle k \rangle^\frac13\hat{\theta}_k^{\mathrm{NL}}(\xi) \langle k-l \rangle^\frac13  \frac{\eta(k-l)-l(\xi-\eta+t(k-l))}{|l|^2+|\eta|^2} \hat{\omega}_l^{\mathrm{NL}}(\eta) \\
			\times \hat{\theta}_{k-l}^{\mathrm{L}}(\xi-\eta) d k d l d \xi d \eta \bigg|, \\
			I_2 :=\left|\int_{\mathbb{R}^4} \mathcal{M}(t, k, \xi) e^{2 c \nu^{\frac{1}{3}} \lambda(k) t}\langle k\rangle^2\langle\frac{1}{k}\rangle^{2 \epsilon} \langle k \rangle^\frac13\hat{\theta}_k^{\mathrm{NL}}(\xi) \langle k-l \rangle^\frac13  \frac{t l(k-l)}{|l|^2+|\eta|^2} \hat{\omega}_l^{\mathrm{NL}}(\eta) \hat{\theta}_{k-l}^{\mathrm{L}}(\xi-\eta) d k d l d \xi d \eta\right| .
		\end{gathered}
		$$ 
		
		\underline{ Step V.1: Estimate of $I_1$.}       
		We classify the analysis into three cases according to the magnitude of the frequency $k$ and the proximity between $k$ and $l$. \\
		$\bullet$~Case 1:   $|k-l| \leq 2|l|$ and $|k|>1$.  It follows from $\langle \frac1k \rangle^{2\epsilon} \lesssim 1$, 
		\begin{align}
			&\langle k \rangle \lesssim \langle k-l, \xi-\eta+(k-l) t\rangle \langle l \rangle, \label{eq:trick4}\\
			& |k-l|+|\xi-\eta+t(k-l)| \leq\langle k-l, \xi-\eta+(k-l) t\rangle, \label{eq:trick5}\\
			& |l|^{-1} \lesssim |k-l|^{-1} \lesssim \langle\frac{1}{k-l}\rangle^2 |k-l|^{\frac{1}{3}} \label{eq:trick6},
		\end{align}
		that
		\begin{align} \label{eq:temp27}
			\no& \left| \int_{|k-l| \leq 2|l|,\,|k|>1} \mathcal{M}(t, k, \xi) e^{2 c \nu^{\frac{1}{3}} \lambda(k) t}\langle k\rangle^2\langle\frac{1}{k}\rangle^{2 \epsilon} \langle k \rangle^\frac13\hat{\theta}_k^{\mathrm{NL}}(\xi) \langle k-l \rangle^\frac13  \frac{\eta(k-l)-l(\xi-\eta+t(k-l))}{|l|^2+|\eta|^2}\right.\\\no
			& \lt. \times \hat{\omega}_l^{\mathrm{NL}}(\eta) \hat{\theta}_{k-l}^{\mathrm{L}}(\xi-\eta) d k d l d \xi d \eta \rt| \\\nonumber
			& \lesssim\|e^{c \nu^{\frac{1}{3}} \lambda(k) t}\langle k\rangle \langle k \rangle^\frac13  \hat{\theta}_k^{\mathrm{NL}}(\xi)\|_{L_{k, \xi}^2}\|e^{c \nu^{\frac{1}{3}} \lambda(l) t}\langle l\rangle \frac{l}{|l|+|\eta|} \hat{\omega}_l^{\mathrm{NL}}(\eta)\|_{L_{l, \eta}^2}\|\langle k-l, \xi-\eta+(k-l) t\rangle^{-2}\|_{L_{k-l, \xi-\eta}^2} \\\no
			& \quad \times\|e^{c \nu^{\frac{1}{3}} \lambda(k-l) t}\langle k-l, \xi-\eta+(k-l) t\rangle^4\langle\frac{1}{k-l}\rangle^2|k-l|^{\frac{1}{3}} \langle k-l \rangle^\frac13 \hat{\theta}_{k-l}^{\mathrm{L}}(\xi-\eta)\|_{L_{k-l, \xi-\eta}^2} \\\no
			& \lesssim   \|e^{c \nu^{\frac{1}{3}} \lambda\left(D_x\right) t}\langle D_x\rangle  \langle D_x \rangle^\frac13  \theta^{\mathrm{NL}}\|_{L^2}\|e^{c \nu^{\frac{1}{3}} \lambda\left(D_x\right) t}\langle D_x\rangle \partial_x \nabla \phi^{\mathrm{NL}}\|_{L^2} \\
			& \quad \times\|e^{c \nu^{\frac{1}{3}} \lambda\left(D_x\right) t}\langle D_x, D_y+t D_x\rangle^4\langle\frac{1}{D_x}\rangle^2\left|D_x\right|^{\frac{1}{3}}  \langle D_x \rangle^\frac13  \theta^{\mathrm{L}}\|_{L^2}.
		\end{align}
		$\bullet$~Case 2:   $|k-l| \leq 2|l|$ and $|k|\leq 1$. Using \eqref{eq:trick4}, \eqref{eq:trick5} and \eqref{eq:trick6} again, we have 
		\begin{align} \label{eq:temp28}
			&\no \left| \int_{|k-l| \leq 2|l|,|k| \leq 1} \mathcal{M}(t, k, \xi) e^{2 c \nu^{\frac{1}{3}} \lambda(k) t}\langle k\rangle^2\langle\frac{1}{k}\rangle^{2 \epsilon} \langle k \rangle^\frac13\hat{\theta}_k^{\mathrm{NL}}(\xi) \langle k-l \rangle^\frac13  \frac{\eta(k-l)-l(\xi-\eta+t(k-l))}{|l|^2+|\eta|^2}\right. \\
			\no& \lt.  \times \hat{\omega}_l^{\mathrm{NL}}(\eta) \hat{\theta}_{k-l}^{\mathrm{L}}(\xi-\eta) d k d l d \xi d \eta \rt| \\
			&\no \lesssim\|\langle\frac{1}{k}\rangle^\epsilon\|_{L_k^2([-1,1])}\|e^{c \nu^{\frac{1}{3}} \lambda(k) t}\langle k\rangle\langle\frac{1}{k}\rangle^\epsilon \langle k \rangle^\frac13  \hat{\theta}_k^{\mathrm{NL}}(\xi)\|_{L_{k, \xi}^2}\|e^{c \nu^{\frac{1}{3}} \lambda(l) t}\langle l\rangle \frac{l}{|l|+|\eta|} \hat{\omega}_l^{\mathrm{NL}}(\eta)\|_{L_{l, \eta}^2} \\
			&\no \quad\times\|\langle\xi-\eta+t(k-l)\rangle^{-1}\|_{L_{\xi-\eta}^2} \\
			&\no \quad\times\|e^{c \nu^{\frac{1}{3}} \lambda(k-l) t}\langle k-l, \xi-\eta+(k-l) t\rangle^3\langle\frac{1}{k-l}\rangle^2|k-l|^{\frac{1}{3}} \langle k-l \rangle ^\frac13 \hat{\theta}_{k-l}^{\mathrm{L}}(\xi-\eta)\|_{L_{k-l, \xi-\eta}^2} \\
			&\no \lesssim\|e^{c \nu^{\frac{1}{3}} \lambda\left(D_x\right) t}\langle D_x\rangle\langle\frac{1}{D_x}\rangle^\epsilon \langle D_x \rangle^\frac13 \theta^{\mathrm{NL}}\|_{L^2}\|e^{c \nu^{\frac{1}{3}} \lambda\left(D_x\right) t}\langle D_x\rangle \partial_x \nabla \phi^{\mathrm{NL}}\|_{L^2} \\
			& \quad \times\|e^{c \nu^{\frac{1}{3}} \lambda\left(D_x\right) t}\langle D_x, D_y+t D_x\rangle^3\langle\frac{1}{D_x}\rangle^2\left|D_x\right|^{\frac{1}{3}}  \langle D_x \rangle^\frac13 \theta^{\mathrm{L}}\|_{L^2}.
		\end{align}
		$\bullet$~Case 3:   $|k-l| > 2|l|$. In this \emph{$l$-near-diagonal region}, it follows that $|k| \sim |k-l|$. Then, by \eqref{eq:trick5}, $\langle k\rangle \lesssim \langle k-l, \xi-\eta+(k-l) t\rangle$, $\langle\frac{1}{k}\rangle^{2 \epsilon}\lesssim \langle\frac{1}{k - l}\rangle^{2\epsilon} \lesssim |k|^\frac13 |k-l|^\frac13 \langle\frac{1}{k - l}\rangle^{2} $, we obtain 
		\begin{align} \label{eq:temp29}
			\no& \left\lvert\, \int_{2|l|<|k-l|} \mathcal{M}(t, k, \xi) e^{2 c \nu^{\frac{1}{3}} \lambda(k) t}\langle k\rangle^2\langle\frac{1}{k}\rangle^{2 \epsilon} \langle k \rangle^\frac13\hat{\theta}_k^{\mathrm{NL}}(\xi) \langle k-l \rangle^\frac13  \frac{\eta(k-l)-l(\xi-\eta+t(k-l))}{|l|^2+|\eta|^2}\right.\no \\\no
			& \times \hat{\omega}_l^{\mathrm{NL}}(\eta) \hat{\theta}_{k-l}^{\mathrm{L}}(\xi-\eta) d k d l d \xi d \eta \mid \\\no
			& \lesssim\|e^{c \nu^{\frac{1}{3}} \lambda(k) t}\langle k\rangle|k|^{\frac{1}{3}} \langle k \rangle^\frac13  \hat{\theta}_k^{\mathrm{NL}}(\xi)\|_{L_{k, \xi}^2}\|e^{c \nu^{\frac{1}{3}} \lambda(l) t}\langle l\rangle\langle\frac{1}{l}\rangle^\epsilon \hat{\omega}_l^{\mathrm{NL}}(\eta)\|_{L_{l, \eta}^2}\|\langle l\rangle^{-1}\langle\frac{1}{l}\rangle^{-\epsilon}\| \frac{1}{|l|+|\eta|}\|_{L_\eta^2}\|_{L_l^2} \\\no
			& \quad \times\|e^{c \nu^{\frac{1}{3}} \lambda(k-l) t}\langle k-l, \xi-\eta+(k-l) t\rangle^2\langle\frac{1}{k-l}\rangle^2|k-l|^{\frac{1}{3}} \langle k-l \rangle^\frac13 \hat{\theta}_{k-l}^{\mathrm{L}}(\xi-\eta)\|_{L_{k-l, \xi-\eta}^2} \\\no
			& \lesssim\|e^{c \nu^{\frac{1}{3}} \lambda\left(D_x\right) t}\langle D_x\rangle\left|D_x\right|^{\frac{1}{3}} \langle D_x \rangle^\frac13  \theta^{\mathrm{NL}}\|_{L^2}\|e^{c \nu^{\frac{1}{3}} \lambda\left(D_x\right) t}\langle D_x\rangle\langle\frac{1}{D_x}\rangle^\epsilon \omega^{\mathrm{NL}}\|_{L^2} \\
			& \quad \times\|e^{c \nu^{\frac{1}{3}} \lambda\left(D_x\right) t}\langle D_x, D_y+t D_x\rangle^2\langle\frac{1}{D_x}\rangle^2\left|D_x\right|^{\frac{1}{3}} \langle D_x \rangle^\frac13  \theta^{\mathrm{L}}\|_{L^2},
		\end{align}
		where we used that for $\epsilon>0$,
		$$
		\|\langle l\rangle^{-1}\langle\frac{1}{l}\rangle^{-\epsilon}\| \frac{1}{|l|+|\eta|}\|_{L_\eta^2}\|_{L_l^2} \lesssim\|\langle l\rangle^{-1}\langle\frac{1}{l}\rangle^{-\epsilon}|l|^{-\frac{1}{2}}\|_{L_l^2} \lesssim 1 .
		$$
		Combining \eqref{eq:temp27}, \eqref{eq:temp28} and \eqref{eq:temp29}, it holds that
		\begin{equation}\begin{aligned}  \label{eq:est of I1}
				I_1 &\lesssim  \lt( \|e^{c \nu^{\frac{1}{3}} \lambda\left(D_x\right) t}\langle D_x\rangle  \langle\frac{1}{D_x}\rangle^\epsilon \langle D_x \rangle^\frac13  \theta^{\mathrm{NL}}\|_{L^2}\|e^{c \nu^{\frac{1}{3}} \lambda\left(D_x\right) t}\langle D_x\rangle \partial_x \nabla \phi^{\mathrm{NL}}\|_{L^2} \rt. \\
				&\lt. \qquad + \|e^{c \nu^{\frac{1}{3}} \lambda\left(D_x\right) t}\langle D_x\rangle\left|D_x\right|^{\frac{1}{3}} \langle D_x \rangle^\frac13  \theta^{\mathrm{NL}}\|_{L^2}\|e^{c \nu^{\frac{1}{3}} \lambda\left(D_x\right) t}\langle D_x\rangle\langle\frac{1}{D_x}\rangle^\epsilon \omega^{\mathrm{NL}}\|_{L^2} \rt) \\
				& \quad \times\|e^{c \nu^{\frac{1}{3}} \lambda\left(D_x\right) t}\langle D_x, D_y+t D_x\rangle^4\langle\frac{1}{D_x}\rangle^2\left|D_x\right|^{\frac{1}{3}} \langle D_x \rangle^\frac13  \theta^{\mathrm{L}}\|_{L^2}
		\end{aligned}\end{equation}
		
	    \underline{ Step V.2: Estimate of $I_2$.}  
		For the real reaction term $I_2$, we will employ the multiplier $\mathcal{M}_3$ to address this challenge. \\
		$\bullet$~Case 1:   $|k-l| \leq 2|k|$.  From \eqref{eq:trick4} and $$\langle\frac{1}{k}\rangle^\epsilon \lesssim\langle\frac{1}{k-l}\rangle^\epsilon  \lesssim |k|^{\frac13 \delta} \langle \frac{1}{k-l} \rangle,$$ there holds that
		\small{	\begin{equation}\begin{aligned} \label{eq:im}
					& \left|\int_{|k-l| \leq 2|k|}  \mathcal{M}(t, k, \xi) e^{2 c \nu^{\frac{1}{3}} \lambda(k) t}\langle k\rangle^2\langle\frac{1}{k}\rangle^{2 \epsilon} \langle k \rangle^\frac13\hat{\theta}_k^{\mathrm{NL}}(\xi) \langle k-l \rangle^\frac13  \frac{t l(k-l)}{|l|^2+|\eta|^2} \hat{\omega}_l^{\mathrm{NL}}(\eta) \hat{\theta}_{k-l}^{\mathrm{L}}(\xi-\eta)   d k d l d \xi d \eta\right| \\
					& \lesssim\|e^{c \nu^{\frac{1}{3}} \lambda(l) t}\langle l\rangle\langle\frac{1}{l}\rangle^\epsilon \frac{l}{\sqrt{|l|^2+|\eta|^2}} \hat{\omega}_l^{\mathrm{NL}}(\eta)\|_{L_{l, \eta}^2} \\
					& \quad \times\|t e^{c \nu^{\frac{1}{3}} \lambda(k-l) t}\langle k-l, \xi-\eta+(k-l) t\rangle^3\langle\frac{1}{k-l}\rangle  \langle k-l \rangle^\frac13 \hat{\theta}_{k-l}^{\mathrm{L}}(\xi-\eta)\|_{L_{k-l, \xi-\eta}^2} \\
					& \quad \times\left(\int_{\mathbb{R}^4} e^{2 c \nu^{\frac{1}{3}} \lambda(k) t}\langle k\rangle^2\langle\frac{1}{k}\rangle^{2 \epsilon}|k |^{\frac{2}{3}\delta}  \frac{\langle\frac{1}{l}\rangle^{-2 \epsilon}}{\left(|l|^2+|\eta|^2\right)\langle k-l, \xi-\eta+(k-l) t\rangle^2}\left|\langle k \rangle^\frac13 \hat{\theta}_k^{\mathrm{NL}}(\xi)\right|^2 d k d l d \xi d \eta\right)^{\frac{1}{2}} .
		\end{aligned}\end{equation} }
		Then, on the one hand, it follows that
		$$
		\begin{aligned}
			& \|t e^{c \nu^{\frac{1}{3}} \lambda(k-l) t}\langle k-l, \xi-\eta+(k-l) t\rangle^3\langle\frac{1}{k-l}\rangle \langle k-l \rangle^\frac13  \hat{\theta}_{k-l}^{\mathrm{L}}(\xi-\eta)\|_{L_{k-l, \xi-\eta}^2} \\
			&\lesssim  \nu^{-\frac{1}{3}}\|c \nu^{\frac{1}{3}}|k-l|^{\frac{2}{3}} t e^{c \nu^{\frac{1}{3}}|k-l|^\frac23  t}\langle k-l, \xi-\eta+(k-l) t\rangle^3\langle\frac{1}{k-l}\rangle^2 \hat{\theta}_{k-l}^{\mathrm{L}}(\xi-\eta)\|_{L_{k-l, \xi-\eta}^2} \\
			&\lesssim \nu^{-\frac{1}{3}}\|e^{c_0 \nu|k-l|^2 t^3}\langle k-l, \xi-\eta+(k-l) t\rangle^3\langle\frac{1}{k-l}\rangle^2 \hat{\theta}_{k-l}^{\mathrm{L}}(\xi-\eta)\|_{L_{k-l, \xi-\eta}^2}
		\end{aligned}
		$$
		On the other hand, it follows from the convolution properties of the Poisson kernel (e.g., see  (2-12) in \cite{WZ2023}) and the trick in \cite{LLZ2025} that
		$$
		\begin{aligned}
			& \int_{\mathbb{R}^2} \frac{\langle\frac{1}{l}\rangle^{-2 \epsilon}}{\left(|l|^2+|\eta|^2\right)\langle k-l, \xi-\eta+(k-l) t\rangle^2} d \eta d l \\
			&\lesssim  \int_{\mathbb{R}}\langle\frac{1}{l}\rangle^{-2 \epsilon} \frac{1}{|l|} \frac{1+|l|+|k-l|}{(1+|k-l|)\left(\langle l, k-l\rangle^2+|\xi+t(k-l)|^2\right)} d l \\
			&\lesssim  \left(\int_{\mathbb{R}}\langle\frac{1}{l}\rangle^{-\frac{1}{3}-\kappa} \frac{1}{|l|^{\frac{1}{3}}}\frac{1+|l|+|k-l|}{(1+|k-l|)\left(\langle l, k-l\rangle^2+|\xi+t(k-l)|^2\right)} d l\right)^{1-\delta}\left(\int_{\mathbb{R}}\langle\frac{1}{l}\rangle^{-\alpha} \frac{1}{|l|}\langle l\rangle^{-1} d l\right)^\delta \\
			&\lesssim  \Upsilon(t, k, \xi)^{1-\delta},
		\end{aligned}
		$$
		where $\alpha=\frac{2 \epsilon-(1-\delta)(1+\kappa)}{\delta}$ satisfies
		$(1-\delta)(-1-\kappa)+\delta(-\alpha)=-2 \epsilon$ 
		and one can choose $0<\kappa<\frac{2 \epsilon}{1-\delta}-1$ so that $\alpha>0$, when $\epsilon>\frac{1-\delta}{2}$. According to the two estimates above, we arrive at
		$$
		\begin{aligned}
			& \left|\int_{|k-l| \leq 2|k|} \mathcal{M}(t, k, \xi) e^{2 c \nu^{\frac{1}{3}} \lambda(k) t}\langle k\rangle^2\langle\frac{1}{k}\rangle^{2 \epsilon} \langle k \rangle^\frac13\hat{\theta}_k^{\mathrm{NL}}(\xi) \langle k-l \rangle^\frac13  \frac{t l(k-l)}{|l|^2+|\eta|^2} \hat{\omega}_l^{\mathrm{NL}}(\eta) \hat{\theta}_{k-l}^{\mathrm{L}}(\xi-\eta) d k d l d \xi d \eta\right| \\
			& \begin{aligned}
				& \lesssim\|e^{c \nu^{\frac{1}{3}} \lambda(l) t}\langle l\rangle\langle\frac{1}{l}\rangle^\epsilon \frac{l}{\sqrt{|l|^2+|\eta|^2}} \hat{\omega}_l^{\mathrm{NL}}(\eta)\|_{L_{l, \eta}^2} \\
				&\quad \times \nu^{-\frac{1}{3}}\|e^{c_0 \nu|k-l|^2 t^3}\langle k-l, \xi-\eta+(k-l) t\rangle^3\langle\frac{1}{k-l}\rangle^2    \hat{\theta}_{k-l}^{\mathrm{L}}(\xi-\eta)\|_{L_{k-l, \epsilon-\eta}^2} \\
				& \quad \times\|e^{c \nu^{\frac{1}{3}} \lambda(k) t}\langle k\rangle\langle\frac{1}{k}\rangle^\epsilon|k|^{\frac{\delta}{3}} \Upsilon(t, k, \xi)^{\frac{1-\delta}{2}} \langle k \rangle^\frac13 \hat{\theta}_k^{\mathrm{NL}}(\xi)\|_{L_{k, \epsilon}^2} \\
				& \lesssim \nu^{-\frac{1}{3} }\|e^{c \nu^{\frac{1}{3}} \lambda\left(D_x\right)}\langle D_x\rangle\langle\frac{1}{D_x}\rangle^\epsilon \partial_x \nabla \phi^{\mathrm{NL}}\|_{L^2}\|e^{c_0 \nu\left|D_x\right|^2 t^3}\langle D_x, D_y+t D_x\rangle^3\langle\frac{1}{D_x}\rangle^2 \omega^{\mathrm{L}}\|_{L^2} \\
				& \quad \times\left( \|e^{c \nu^{\frac{1}{3}} \lambda\left(D_x\right) t}\langle D_x\rangle\langle\frac{1}{D_x}\rangle^\epsilon\left|D_x\right|^{\frac{1}{3}}\langle D_x \rangle^\frac{1}{3} \theta^{\mathrm{NL}}\|_{L^2}\right)^\delta \| e^{c \nu^\frac13 \lambda\left(D_x\right) t} 
				\langle D_x\rangle\langle\frac{1}{D_x}\rangle^{\varepsilon} \sqrt{\Upsilon } \langle D_x \rangle^\frac13 \theta^{\mathrm{NL}} \|_{L^2}^{1-\delta} .
			\end{aligned}
		\end{aligned}
		$$
		$\bullet$~Case 2:   $|k-l| > 2|k|$. In such a case, thanks to 
		\begin{align}
			&\langle k \rangle^2 \lesssim \langle k-l, \xi-\eta+(k-l) t\rangle \langle l \rangle , \label{eq:trick7}\\
			& 1 \lesssim \langle \frac{1}{k-l} \rangle \langle \frac1l \rangle^{-\frac{1+ \kappa}{2} + \epsilon},  \label{eq:trick8}
		\end{align}
		it holds 
		\begin{eqnarray*}
			&& \left|\int_{2|k|<|k-l|}\mathcal{M}(t, k, \xi) e^{2 c \nu^{\frac{1}{3}} \lambda(k) t}\langle k\rangle^2\langle\frac{1}{k}\rangle^{2 \epsilon} \langle k \rangle^\frac13\hat{\theta}_k^{\mathrm{NL}}(\xi)\rt.\\
			&&\lt. \qquad\qquad  \qquad \times  \langle k-l \rangle^\frac13  \frac{t l(k-l)}{|l|^2+|\eta|^2} \hat{\omega}_l^{\mathrm{NL}}(\eta) \hat{\theta}_{k-l}^{\mathrm{L}}(\xi-\eta) d k d l d \xi d \eta\right| \\
			& \lesssim&\|\langle k\rangle^{-1}\langle\frac{1}{k}\rangle^\epsilon\|_{L_k^2}\|e^{c \nu^{\frac{1}{3}} \lambda(l) t}\langle l\rangle\langle\frac{1}{l}\rangle^\epsilon \frac{l}{\sqrt{|l|^2+|\eta|^2}} \hat{\omega}_l^{\mathrm{NL}}(\eta)\|_{L_{l, \eta}^2}\|\langle\xi-\eta+t(k-l)\rangle^{-1}\|_{L_{\xi}^2} \\
			&& \times\|t e^{c \nu^{\frac{1}{3}} \lambda(k-l) t}\langle k-l, \xi-\eta+(k-l) t\rangle^4\langle\frac{1}{k-l}\rangle \langle k-l \rangle^\frac13 \hat{\theta}_{k-l}^{\mathrm{L}}(\xi-\eta)\|_{L_{k-l, \xi-\eta}^{\infty}} \\
			&& \times\left(\int_{\mathbb{R}^4} e^{2 c \nu^3 \lambda(k) t}\langle k\rangle^2\langle\frac{1}{k}\rangle^{2 \epsilon} \frac{\langle\frac{1}{l}\rangle^{-1-\kappa}}{\left(|l|^2+|\eta|^2\right)\langle k-l, \xi-\eta+(k-l) t\rangle^2}\left| \langle k \rangle^\frac13 \hat{\theta}_k^{\mathrm{NL}}(\xi)\right|^2 d k d l d \xi d \eta\right)^{\frac{1}{2}} \\
			& \lesssim&\|e^{c \nu^{\frac{1}{3}} \lambda(l) t}\langle l\rangle\langle\frac{1}{l}\rangle^\epsilon \frac{l}{\sqrt{|l|^2+|\eta|^2}}   \hat{\omega}_l^{\mathrm{NL}}(\eta)\|_{L_{l, \eta}^2}  \|e^{c \nu^{\frac{1}{3}} \lambda(k) t}\langle k\rangle\langle\frac{1}{k}\rangle^\epsilon \sqrt{\Upsilon(t, k, \xi)} \langle k \rangle^\frac13  \hat{\theta}_k^{\mathrm{NL}}(\xi)\|_{L_{k, \epsilon}^2} \\
			&& \times \nu^{-\frac{1}{3}}\|e^{c_0 \nu|k-l|^2 t^3}\langle k-l, \xi-\eta+(k-l) t\rangle^4\langle\frac{1}{k-l}\rangle^2  \hat{\theta}_{k-l}^{\mathrm{L}}(\xi-\eta)\|_{L_{k-l, \xi-\eta}^{\infty}} \\
			& \lesssim& \nu^{-\frac{1}{3}}\|e^{c \nu^{\frac{1}{3}} \lambda\left(D_x\right)}\langle D_x\rangle\langle\frac{1}{D_x}\rangle^\epsilon \partial_x \nabla \phi^{\mathrm{NL}}\|_{L^2}\|e^{c_0 \nu\left|D_x\right|^2 t^3}\langle D_x, D_y+t D_x\rangle^4\langle\frac{1}{D_x}\rangle^2 \theta^{\mathrm{L}}\|_{L^1} \\
			&& \times\|e^{c \nu^{\frac{1}{3}} \lambda\left(D_x\right) t}\langle D_x\rangle\langle\frac{1}{D_x}\rangle^\epsilon \sqrt{\Upsilon} \langle D_x \rangle^\frac13  \theta^{\mathrm{NL}}\|_{L^2} .
		\end{eqnarray*} 

		Combining the above two cases, we deduce that
		$$
		\begin{aligned}
			& I_2 \lesssim \nu^{-\frac{1}{3}-\frac{\delta}{6}}\|e^{c \nu^{\frac{1}{3}} \lambda\left(D_x\right)}\langle D_x\rangle\langle\frac{1}{D_x}\rangle^\epsilon \partial_x \nabla \phi^{\mathrm{NL}}\|_{L^2}\|e^{c_0 \nu\left|D_x\right|^2 t^3}\langle D_x, D_y+t D_x\rangle^4\langle\frac{1}{D_x}\rangle^2 \theta^{\mathrm{L}}\|_{L^1 \cap L^2} \\
			& \times\left(\|e^{c \nu^{\frac{1}{3}} \lambda\left(D_x\right) t}\langle D_x\rangle\langle\frac{1}{D_x}\rangle^\epsilon \sqrt{\Upsilon} \langle D_x \rangle^\frac13  \theta^{\mathrm{NL}}\|_{L^2}+\nu^{\frac{1}{6}}\|e^{c \nu^{\frac{1}{3}} \lambda\left(D_x\right) t}\langle D_x\rangle\langle\frac{1}{D_x}\rangle^\epsilon |D_x|^\frac13 \langle D_x \rangle^{\frac{1}{3}} \theta^{\mathrm{NL}}\|_{L^2}\right),
		\end{aligned}
		$$
		which, combined \eqref{eq:est of I1} and \eqref{eq:I1I2}, gives that
		\begin{equation*}\begin{aligned}
				&\left|\left(u^{\mathrm{NL}} \cdot \nabla \langle D_x \rangle^\frac13 \theta^{\mathrm{L}} \big| \mathcal{M} e^{2 c \nu^{\frac{1}{3}} \lambda\left(D_x\right) t}\langle D_x\rangle^2\langle\frac{1}{D_x}\rangle^{2 \epsilon} \langle D_x \rangle^\frac13 \theta^{\mathrm{NL}}\right)\right| \\
				\lesssim &  \lt( \|e^{c \nu^{\frac{1}{3}} \lambda\left(D_x\right) t}\langle D_x\rangle  \langle\frac{1}{D_x}\rangle^\epsilon \langle D_x \rangle^\frac13  \theta^{\mathrm{NL}}\|_{L^2}\|e^{c \nu^{\frac{1}{3}} \lambda\left(D_x\right) t}\langle D_x\rangle \partial_x \nabla \phi^{\mathrm{NL}}\|_{L^2} \rt. \\
				&\lt. \quad + \|e^{c \nu^{\frac{1}{3}} \lambda\left(D_x\right) t}\langle D_x\rangle\left|D_x\right|^{\frac{1}{3}} \langle D_x \rangle^\frac13  \theta^{\mathrm{NL}}\|_{L^2}\|e^{c \nu^{\frac{1}{3}} \lambda\left(D_x\right) t}\langle D_x\rangle\langle\frac{1}{D_x}\rangle^\epsilon \omega^{\mathrm{NL}}\|_{L^2} \rt) \\
				&  \times\|e^{c \nu^{\frac{1}{3}} \lambda\left(D_x\right) t}\langle D_x, D_y+t D_x\rangle^4\langle\frac{1}{D_x}\rangle^2\left|D_x\right|^{\frac{1}{3}} \langle D_x \rangle^\frac13  \theta^{\mathrm{L}}\|_{L^2} \\
				&  + \nu^{-\frac{1}{3}-\frac{\delta}{6}}\|e^{c \nu^{\frac{1}{3}} \lambda\left(D_x\right)}\langle D_x\rangle\langle\frac{1}{D_x}\rangle^\epsilon \partial_x \nabla \phi^{\mathrm{NL}}\|_{L^2}\|e^{c_0 \nu\left|D_x\right|^2 t^3}\langle D_x, D_y+t D_x\rangle^4\langle\frac{1}{D_x}\rangle^2 \theta^{\mathrm{L}}\|_{L^1 \cap L^2} \\
				&\quad  \times\left(\|e^{c \nu^{\frac{1}{3}} \lambda\left(D_x\right) t}\langle D_x\rangle\langle\frac{1}{D_x}\rangle^\epsilon \sqrt{\Upsilon} \langle D_x \rangle^\frac13  \theta^{\mathrm{NL}}\|_{L^2}+\nu^{\frac{1}{6}}\|e^{c \nu^{\frac{1}{3}} \lambda\left(D_x\right) t}\langle D_x\rangle\langle\frac{1}{D_x}\rangle^\epsilon |D_x|^\frac13 \langle D_x \rangle^{\frac{1}{3}} \theta^{\mathrm{NL}}\|_{L^2}\right).
		\end{aligned}\end{equation*}

		\underline{\bf Step VI: Estimate of $  |D_x|^\frac13 u^{\mathrm{NL}} \cdot  \nabla  \theta^{\mathrm{L}}$.} We divide this term into two parts as follows:
		$$
		\begin{aligned}
			& \left|\left( |D_x|^\frac13 u^{\mathrm{NL}} \cdot \nabla  \theta^{\mathrm{L}} \big| \mathcal{M} e^{2 c \nu^{\frac{1}{3}} \lambda\left(D_x\right) t}\langle D_x\rangle^2\langle\frac{1}{D_x}\rangle^{2 \epsilon} |D_x|^\frac13 \theta^{\mathrm{NL}}\right)\right| \\
			= & \left|\int_{\mathbb{R}^4} \mathcal{M}(t, k, \xi) e^{2 c \nu^{\frac{1}{3}} \lambda(k) t}\langle k\rangle^2\langle\frac{1}{k}\rangle^{2 \epsilon} |k|^\frac13\hat{\theta}_k^{\mathrm{NL}}(\xi) |l|^\frac13  \frac{\eta(k-l)-l(\xi-\eta)}{|l|^2+|\eta|^2} \hat{\omega}_l^{\mathrm{NL}}(\eta) \hat{\theta}_{k-l}^{\mathrm{L}}(\xi-\eta) d k d l d \xi d \eta\right| \\
			\leq & I_3+I_4,
		\end{aligned}
		$$
		where
		$$
		\begin{gathered}
			I_3 :=\left\lvert\, \int_{\mathbb{R}^4} \mathcal{M}(t, k, \xi) e^{2 c \nu^{\frac{1}{3}} \lambda(k) t}\langle k\rangle^2\langle\frac{1}{k}\rangle^{2 \epsilon} |k|^\frac13\hat{\theta}_k^{\mathrm{NL}}(\xi) |l|^\frac13  \frac{\eta(k-l)-l(\xi-\eta+t(k-l))}{|l|^2+|\eta|^2} \hat{\omega}_l^{\mathrm{NL}}(\eta)\right. \\
			\times \hat{\theta}_{k-l}^{\mathrm{L}}(\xi-\eta) d k d l d \xi d \eta \mid \\
			I_4 :=\left|\int_{\mathbb{R}^4} \mathcal{M}(t, k, \xi) e^{2 c \nu^{\frac{1}{3}} \lambda(k) t}\langle k\rangle^2\langle\frac{1}{k}\rangle^{2 \epsilon} |k|^\frac13\hat{\theta}_k^{\mathrm{NL}}(\xi) |l|^\frac13  \frac{t l(k-l)}{|l|^2+|\eta|^2} \hat{\omega}_l^{\mathrm{NL}}(\eta) \hat{\theta}_{k-l}^{\mathrm{L}}(\xi-\eta) d k d l d \xi d \eta\right| .
		\end{gathered}
		$$  
		
		\underline{ StepVI.1: Estimate of $  I_3$.}      
		For $I_3$, it is similar as $I_1$. \\
		$\bullet$~Case 1: $|k-l| \leq 2|l|$. For $|k|>1$,  it follows from \eqref{eq:trick4}, $\langle \frac1k \rangle^{2\epsilon} \lesssim 1$ and $|l|^{-\frac23} \lesssim |k-l|^{-\frac23} \lesssim \langle\frac{1}{k-l}\rangle^2 |k-l|^{\frac{2}{3}} $ that
		\begin{align} \label{eq:temp30}
			& \left| \int_{|k-l| \leq 2|l|,\,|k|>1} \mathcal{M}(t, k, \xi) e^{2 c \nu^{\frac{1}{3}} \lambda(k) t}\langle k\rangle^2\langle\frac{1}{k}\rangle^{2 \epsilon} |k|^\frac13\hat{\theta}_k^{\mathrm{NL}}(\xi) |l|^\frac13  \frac{\eta(k-l)-l(\xi-\eta+t(k-l))}{|l|^2+|\eta|^2}\right. \no\\
			& \lt. \times \hat{\omega}_l^{\mathrm{NL}}(\eta) \hat{\theta}_{k-l}^{\mathrm{L}}(\xi-\eta) d k d l d \xi d \eta \rt| \no \\
			& \lesssim\|e^{c \nu^{\frac{1}{3}} \lambda(k) t}\langle k\rangle |k|^\frac13  \hat{\theta}_k^{\mathrm{NL}}(\xi)\|_{L_{k, \xi}^2}\|e^{c \nu^{\frac{1}{3}} \lambda(l) t}\langle l\rangle \frac{l}{|l|+|\eta|} \hat{\omega}_l^{\mathrm{NL}}(\eta)\|_{L_{l, \eta}^2}\|\langle k-l, \xi-\eta+(k-l) t\rangle^{-2}\|_{L_{k-l, \xi-\eta}^2} \no\\
			& \quad \times\|e^{c \nu^{\frac{1}{3}} \lambda(k-l) t}\langle k-l, \xi-\eta+(k-l) t\rangle^4\langle\frac{1}{k-l}\rangle^2|k-l|^{\frac{2}{3}} \hat{\theta}_{k-l}^{\mathrm{L}}(\xi-\eta)\|_{L_{k-l, \xi-\eta}^2} \no\\
			& \lesssim\|e^{c \nu^{\frac{1}{3}} \lambda\left(D_x\right) t}\langle D_x\rangle  |D_x|^\frac13  \theta^{\mathrm{NL}}\|_{L^2}\|e^{c \nu^{\frac{1}{3}} \lambda\left(D_x\right) t}\langle D_x\rangle \partial_x \nabla \phi^{\mathrm{NL}}\|_{L^2} \no \\
			& \quad \times\|e^{c \nu^{\frac{1}{3}} \lambda\left(D_x\right) t}\langle D_x, D_y+t D_x\rangle^4\langle\frac{1}{D_x}\rangle^2\left|D_x\right|^{\frac{2}{3}} \theta^{\mathrm{L}}\|_{L^2}
		\end{align}
		and similarly for $|k| \leq 1$,
		\begin{align} \label{eq:temp31}
			& \left\lvert\, \int_{|k-l| \leq 2|l|,|k| \leq 1} \mathcal{M}(t, k, \xi) e^{2 c \nu^{\frac{1}{3}} \lambda(k) t}\langle k\rangle^2\langle\frac{1}{k}\rangle^{2 \epsilon} |k|^\frac13\hat{\theta}_k^{\mathrm{NL}}(\xi) |l|^\frac13  \frac{\eta(k-l)-l(\xi-\eta+t(k-l))}{|l|^2+|\eta|^2}\right. \no\\
			& \times \hat{\omega}_l^{\mathrm{NL}}(\eta) \hat{\theta}_{k-l}^{\mathrm{L}}(\xi-\eta) d k d l d \xi d \eta \mid \no \\
			& \lesssim\|\langle\frac{1}{k}\rangle^\epsilon\|_{L_k^2([-1,1])}\|e^{c \nu^{\frac{1}{3}} \lambda(k) t}\langle k\rangle\langle\frac{1}{k}\rangle^\epsilon |k|^\frac13  \hat{\theta}_k^{\mathrm{NL}}(\xi)\|_{L_{k, \xi}^2}\|e^{c \nu^{\frac{1}{3}} \lambda(l) t}\langle l\rangle \frac{l}{|l|+|\eta|} \hat{\omega}_l^{\mathrm{NL}}(\eta)\|_{L_{l, \eta}^2} \no\\
			& \quad \times\|\langle\xi-\eta+t(k-l)\rangle^{-1}\|_{L_{\xi-\eta}^2} \no\\
			& \times\|e^{c \nu^{\frac{1}{3}} \lambda(k-l) t}\langle k-l, \xi-\eta+(k-l) t\rangle^3\langle\frac{1}{k-l}\rangle^2|k-l|^{\frac{2}{3}} \hat{\theta}_{k-l}^{\mathrm{L}}(\xi-\eta)\|_{L_{k-l, \xi-\eta}^2} \no\\
			& \lesssim\|e^{c \nu^{\frac{1}{3}} \lambda\left(D_x\right) t}\langle D_x\rangle\langle\frac{1}{D_x}\rangle^\epsilon |D_x|^\frac13 \theta^{\mathrm{NL}}\|_{L^2}\|e^{c \nu^{\frac{1}{3}} \lambda\left(D_x\right) t}\langle D_x\rangle \partial_x \nabla \phi^{\mathrm{NL}}\|_{L^2} \no\\
			& \quad \times\|e^{c \nu^{\frac{1}{3}} \lambda\left(D_x\right) t}\langle D_x, D_y+t D_x\rangle^3\langle\frac{1}{D_x}\rangle^2\left|D_x\right|^{\frac{2}{3}} \theta^{\mathrm{L}}\|_{L^2}.
		\end{align}

		$\bullet$~Case 2: $|k-l| > 2|l|$. From $\langle k\rangle \lesssim \langle k-l, \xi-\eta+(k-l) t\rangle$, 
		$$
		|l|^\frac13 \langle\frac{1}{k}\rangle^{2 \epsilon}\lesssim |k-l|^\frac13 \langle\frac{1}{k - l}\rangle^{2\epsilon} \lesssim |k|^\frac13 |k-l|^\frac23 \langle\frac{1}{k - l}\rangle^{2} 
		$$ 
		and \eqref{eq:trick5}, we have
		\begin{align} \label{eq:temp32}
			& \left\lvert\, \int_{2|l|<|k-l|} \mathcal{M}(t, k, \xi) e^{2 c \nu^{\frac{1}{3}} \lambda(k) t}\langle k\rangle^2\langle\frac{1}{k}\rangle^{2 \epsilon} |k|^\frac13\hat{\theta}_k^{\mathrm{NL}}(\xi) |l|^\frac13  \frac{\eta(k-l)-l(\xi-\eta+t(k-l))}{|l|^2+|\eta|^2}\right.\no \\
			& \times \hat{\omega}_l^{\mathrm{NL}}(\eta) \hat{\theta}_{k-l}^{\mathrm{L}}(\xi-\eta) d k d l d \xi d \eta \mid \no\\
			& \lesssim\|e^{c \nu^{\frac{1}{3}} \lambda(k) t}\langle k\rangle|k|^{\frac{2}{3}} \hat{\theta}_k^{\mathrm{NL}}(\xi)\|_{L_{k, \xi}^2}\|e^{c \nu^{\frac{1}{3}} \lambda(l) t}\langle l\rangle\langle\frac{1}{l}\rangle^\epsilon \hat{\omega}_l^{\mathrm{NL}}(\eta)\|_{L_{l, \eta}^2}\|\langle l\rangle^{-1}\langle\frac{1}{l}\rangle^{-\epsilon}\| \frac{1}{|l|+|\eta|}\|_{L_\eta^2}\|_{L_l^2} \no\\
			& \quad \times\|e^{c \nu^{\frac{1}{3}} \lambda(k-l) t}\langle k-l, \xi-\eta+(k-l) t\rangle^2\langle\frac{1}{k-l}\rangle^2|k-l|^{\frac{2}{3}} \hat{\theta}_{k-l}^{\mathrm{L}}(\xi-\eta)\|_{L_{k-l, \xi-\eta}^2} \no\\
			& \lesssim\|e^{c \nu^{\frac{2}{3}} \lambda\left(D_x\right) t}\langle D_x\rangle\left|D_x\right|^{\frac{2}{3}} \theta^{\mathrm{NL}}\|_{L^2}\|e^{c \nu^{\frac{1}{3}} \lambda\left(D_x\right) t}\langle D_x\rangle\langle\frac{1}{D_x}\rangle^\epsilon \omega^{\mathrm{NL}}\|_{L^2} \no\\
			& \quad \times\|e^{c \nu^{\frac{1}{3}} \lambda\left(D_x\right) t}\langle D_x, D_y+t D_x\rangle^2\langle\frac{1}{D_x}\rangle^2\left|D_x\right|^{\frac{2}{3}} \theta^{\mathrm{L}}\|_{L^2}.
		\end{align}
		Combining \eqref{eq:temp30}, \eqref{eq:temp31} with \eqref{eq:temp32}, we get that
		\begin{equation}\begin{aligned} \label{eq:est of I3}
				I_3 \lesssim& \lt( \|e^{c \nu^{\frac{1}{3}} \lambda\left(D_x\right) t}\langle D_x\rangle\langle\frac{1}{D_x}\rangle^\epsilon |D_x|^\frac13 \theta^{\mathrm{NL}}\|_{L^2}\|e^{c \nu^{\frac{1}{3}} \lambda\left(D_x\right) t}\langle D_x\rangle \partial_x \nabla \phi^{\mathrm{NL}}\|_{L^2} \rt. \\
				& \lt. \quad +\|e^{c \nu^{\frac{2}{3}} \lambda\left(D_x\right) t}\langle D_x\rangle\left|D_x\right|^{\frac{2}{3}} \theta^{\mathrm{NL}}\|_{L^2}\|e^{c \nu^{\frac{1}{3}} \lambda\left(D_x\right) t}\langle D_x\rangle\langle\frac{1}{D_x}\rangle^\epsilon \omega^{\mathrm{NL}}\|_{L^2} \rt)\\
				& \times\|e^{c \nu^{\frac{1}{3}} \lambda\left(D_x\right) t}\langle D_x, D_y+t D_x\rangle^4\langle\frac{1}{D_x}\rangle^2\left|D_x\right|^{\frac{2}{3}} \theta^{\mathrm{L}}\|_{L^2}.
		\end{aligned}\end{equation}
	
	    \underline{ StepVI.2: Estimate of $  I_4$.} 
		Now, we turn to $I_4$.\\
		$\bullet$~Case 1: $|k-l| \leq 2|k|$. From \eqref{eq:trick4} and $\langle\frac{1}{k}\rangle^\epsilon \lesssim\langle\frac{1}{k-l}\rangle^\epsilon  \lesssim  \langle \frac{1}{k-l} \rangle$, it holds that
		\begin{equation}\begin{aligned} \label{eq:im2}
				& \left|\int_{|k-l| \leq 2|k|}  \mathcal{M}(t, k, \xi) e^{2 c \nu^{\frac{1}{3}} \lambda(k) t}\langle k\rangle^2\langle\frac{1}{k}\rangle^{2 \epsilon} |k|^\frac13\hat{\theta}_k^{\mathrm{NL}}(\xi) |l|^\frac13  \frac{t l(k-l)}{|l|^2+|\eta|^2} \hat{\omega}_l^{\mathrm{NL}}(\eta) \hat{\theta}_{k-l}^{\mathrm{L}}(\xi-\eta)   d k d l d \xi d \eta\right| \\
				& \lesssim\|e^{c \nu^{\frac{1}{3}} \lambda(l) t}\langle l\rangle\langle\frac{1}{l}\rangle^\epsilon \frac{l}{\sqrt{|l|^2+|\eta|^2}} \hat{\omega}_l^{\mathrm{NL}}(\eta)\|_{L_{l, \eta}^2} \\
				& \quad \times\|t e^{c \nu^{\frac{1}{3}} \lambda(k-l) t}\langle k-l, \xi-\eta+(k-l) t\rangle^3\langle\frac{1}{k-l}\rangle   \hat{\theta}_{k-l}^{\mathrm{L}}(\xi-\eta)\|_{L_{k-l, \xi-\eta}^2} \\
				& \quad \times\left(\int_{\mathbb{R}^4} e^{2 c \nu^{\frac{1}{3}} \lambda(k) t}\langle k\rangle^2\langle\frac{1}{k}\rangle^{2 \epsilon}  \frac{\langle\frac{1}{l}\rangle^{-\frac13 - \kappa}  |l|^\frac23 }{\left(|l|^2+|\eta|^2\right)\langle k-l, \xi-\eta+(k-l) t\rangle^2}\left||k|^\frac13 \hat{\theta}_k^{\mathrm{NL}}(\xi)\right|^2 d k d l d \xi d \eta\right)^{\frac{1}{2}} .
		\end{aligned}\end{equation}
		Hence, noting that
		$$
		\begin{aligned}
			& \|t e^{c \nu^{\frac{1}{3}} \lambda(k-l) t}\langle k-l, \xi-\eta+(k-l) t\rangle^3\langle\frac{1}{k-l}\rangle  \hat{\theta}_{k-l}^{\mathrm{L}}(\xi-\eta)\|_{L_{k-l, \xi-\eta}^2} \\
			\lesssim & \nu^{-\frac{1}{3}}\|e^{c_0 \nu|k-l|^2 t^3}\langle k-l, \xi-\eta+(k-l) t\rangle^3\langle\frac{1}{k-l}\rangle^2 \hat{\theta}_{k-l}^{\mathrm{L}}(\xi-\eta)\|_{L_{k-l, \xi-\eta}^2}
		\end{aligned}
		$$
		and
		$$
		\begin{aligned}
			& \int_{\mathbb{R}^2} \frac{\langle\frac{1}{l}\rangle^{-\frac13 - \kappa}  |l|^\frac23 }{\left(|l|^2+|\eta|^2\right)\langle k-l, \xi-\eta+(k-l) t\rangle^2} d \eta d l \\
			\lesssim & \int_{\mathbb{R}}\langle\frac{1}{l}\rangle^{-\frac13 - \kappa} \frac{1}{|l|^\frac13 } \frac{1+|l|+|k-l|}{(1+|k-l|)\left(\langle l, k-l\rangle^2+|\xi+t(k-l)|^2\right)} d l 
			\lesssim \Upsilon(t, k, \xi),
		\end{aligned}
		$$
		we deduce that 
		\begin{align*}
			& \left|\int_{|k-l| \leq 2|k|} \mathcal{M}(t, k, \xi) e^{2 c \nu^{\frac{1}{3}} \lambda(k) t}\langle k\rangle^2\langle\frac{1}{k}\rangle^{2 \epsilon} |k|^\frac13\hat{\theta}_k^{\mathrm{NL}}(\xi) |l|^\frac13  \frac{t l(k-l)}{|l|^2+|\eta|^2} \hat{\omega}_l^{\mathrm{NL}}(\eta) \hat{\theta}_{k-l}^{\mathrm{L}}(\xi-\eta) d k d l d \xi d \eta\right| \\
			& \lesssim \nu^{-\frac{1}{3} }\|e^{c \nu^{\frac{1}{3}} \lambda\left(D_x\right)}\langle D_x\rangle\langle\frac{1}{D_x}\rangle^\epsilon \partial_x \nabla \phi^{\mathrm{NL}}\|_{L^2}\|e^{c_0 \nu\left|D_x\right|^2 t^3}\langle D_x, D_y+t D_x\rangle^3\langle\frac{1}{D_x}\rangle^2 \theta^{\mathrm{L}}\|_{L^2} \\
			& \quad \times \| e^{c \nu^\frac13 \lambda\left(D_x\right) t} 
			\langle D_x\rangle\langle\frac{1}{D_x}\rangle^{\varepsilon} \sqrt{\Upsilon } |D_x|^\frac13 \theta^{\mathrm{NL}} \|_{L^2} .
		\end{align*}
		$\bullet$~Case 2:  $2|k|<|k-l|$.  \eqref{eq:trick7}, \eqref{eq:trick8} and $|l|^\frac13 \lesssim |k-l|^\frac13 $ implies
		\begin{align*}
			& \left|\int_{2|k|<|k-l|}\mathcal{M}(t, k, \xi) e^{2 c \nu^{\frac{1}{3}} \lambda(k) t}\langle k\rangle^2\langle\frac{1}{k}\rangle^{2 \epsilon} |k|^\frac13\hat{\theta}_k^{\mathrm{NL}}(\xi) |l|^\frac13  \frac{t l(k-l)}{|l|^2+|\eta|^2} \hat{\omega}_l^{\mathrm{NL}}(\eta) \hat{\theta}_{k-l}^{\mathrm{L}}(\xi-\eta) d k d l d \xi d \eta\right| \\
			& \lesssim\|\langle k\rangle^{-1}\langle\frac{1}{k}\rangle^\epsilon\|_{L_k^2}\|e^{c \nu^{\frac{1}{3}} \lambda(l) t}\langle l\rangle\langle\frac{1}{l}\rangle^\epsilon \frac{l}{\sqrt{|l|^2+|\eta|^2}} \hat{\omega}_l^{\mathrm{NL}}(\eta)\|_{L_{l, \eta}^2}\|\langle\xi-\eta+t(k-l)\rangle^{-1}\|_{L_{\xi}^2} \\
			& \times\|t e^{c \nu^{\frac{1}{3}} \lambda(k-l) t}\langle k-l, \xi-\eta+(k-l) t\rangle^4\langle\frac{1}{k-l}\rangle |k-l|^\frac13 \hat{\theta}_{k-l}^{\mathrm{L}}(\xi-\eta)\|_{L_{k-l, \xi-\eta}^{\infty}} \\
			& \times\left(\int_{\mathbb{R}^4} e^{2 c \nu^3 \lambda(k) t}\langle k\rangle^2\langle\frac{1}{k}\rangle^{2 \epsilon} \frac{\langle\frac{1}{l}\rangle^{-1-\kappa}}{\left(|l|^2+|\eta|^2\right)\langle k-l, \xi-\eta+(k-l) t\rangle^2}\left| |k|^\frac13 \hat{\theta}_k^{\mathrm{NL}}(\xi)\right|^2 d k d l d \xi d \eta\right)^{\frac{1}{2}} \\
			& \lesssim\|e^{c \nu^{\frac{1}{3}} \lambda(l) t}\langle l\rangle\langle\frac{1}{l}\rangle^\epsilon \frac{l}{\sqrt{|l|^2+|\eta|^2}} \hat{\omega}_l^{\mathrm{NL}}(\eta)\|_{L_{l, \eta}^2} \\
			& \times \nu^{-\frac{1}{3}}\|e^{c_0 \nu|k-l|^2 t^3}\langle k-l, \xi-\eta+(k-l) t\rangle^4\langle\frac{1}{k-l}\rangle^2  \hat{\theta}_{k-l}^{\mathrm{L}}(\xi-\eta)\|_{L_{k-l, \xi-\eta}^{\infty}} \\
			& \times\|e^{c \nu^{\frac{1}{3}} \lambda(k) t}\langle k\rangle\langle\frac{1}{k}\rangle^\epsilon \sqrt{\Upsilon(t, k, \xi)} |k|^\frac13  \hat{\theta}_k^{\mathrm{NL}}(\xi)\|_{L_{k, \epsilon}^2} \\
			& \lesssim \nu^{-\frac{1}{3}}\|e^{c \nu^{\frac{1}{3}} \lambda\left(D_x\right)}\langle D_x\rangle\langle\frac{1}{D_x}\rangle^\epsilon \partial_x \nabla \phi^{\mathrm{NL}}\|_{L^2}\|e^{c_0 \nu\left|D_x\right|^2 t^3}\langle D_x, D_y+t D_x\rangle^4\langle\frac{1}{D_x}\rangle^2 \theta^{\mathrm{L}}\|_{L^1} \\
			& \times\|e^{c \nu^{\frac{1}{3}} \lambda\left(D_x\right) t}\langle D_x\rangle\langle\frac{1}{D_x}\rangle^\epsilon \sqrt{\Upsilon} |D_x|^\frac13  \theta^{\mathrm{NL}}\|_{L^2} .
		\end{align*}
		
		Adding the above two cases together, we get that
		$$
		\begin{aligned}
			I_4 &\lesssim \nu^{-\frac{1}{3}-\frac{\delta}{6}}\|e^{c \nu^{\frac{1}{3}} \lambda\left(D_x\right)}\langle D_x\rangle\langle\frac{1}{D_x}\rangle^\epsilon \partial_x \nabla \phi^{\mathrm{NL}}\|_{L^2}\|e^{c_0 \nu\left|D_x\right|^2 t^3}\langle D_x, D_y+t D_x\rangle^4\langle\frac{1}{D_x}\rangle^2 \theta^{\mathrm{L}}\|_{L^1 \cap L^2} \\
			&\quad  \times\left(\|e^{c \nu^{\frac{1}{3}} \lambda\left(D_x\right) t}\langle D_x\rangle\langle\frac{1}{D_x}\rangle^\epsilon \sqrt{\Upsilon} |D_x|^\frac13  \theta^{\mathrm{NL}}\|_{L^2}+\nu^{\frac{1}{6}}\|e^{c \nu^{\frac{1}{3}} \lambda\left(D_x\right) t}\langle D_x\rangle\langle\frac{1}{D_x}\rangle^\epsilon\left|D_x\right|^{\frac{2}{3}} \theta^{\mathrm{NL}}\|_{L^2}\right), 
		\end{aligned}
		$$
		which along with \eqref{eq:est of I3} implies
		\begin{align*}
			&\left|\left( |D_x|^\frac13 u^{\mathrm{NL}} \cdot \nabla  \theta^{\mathrm{L}} \left\lvert\, \mathcal{M} e^{2 c \nu^{\frac{1}{3}} \lambda\left(D_x\right) t}\langle D_x\rangle^2\langle\frac{1}{D_x}\rangle^{2 \epsilon} |D_x|^\frac13 \theta^{\mathrm{NL}}\right.\right)\right|\\
			\lesssim & \lt( \|e^{c \nu^{\frac{1}{3}} \lambda\left(D_x\right) t}\langle D_x\rangle\langle\frac{1}{D_x}\rangle^\epsilon |D_x|^\frac13 \theta^{\mathrm{NL}}\|_{L^2}\|e^{c \nu^{\frac{1}{3}} \lambda\left(D_x\right) t}\langle D_x\rangle \partial_x \nabla \phi^{\mathrm{NL}}\|_{L^2} \rt. \\
			& \lt. \quad +\|e^{c \nu^{\frac{2}{3}} \lambda\left(D_x\right) t}\langle D_x\rangle\left|D_x\right|^{\frac{2}{3}} \theta^{\mathrm{NL}}\|_{L^2}\|e^{c \nu^{\frac{1}{3}} \lambda\left(D_x\right) t}\langle D_x\rangle\langle\frac{1}{D_x}\rangle^\epsilon \omega^{\mathrm{NL}}\|_{L^2} \rt)\\
			& \times\|e^{c \nu^{\frac{1}{3}} \lambda\left(D_x\right) t}\langle D_x, D_y+t D_x\rangle^4\langle\frac{1}{D_x}\rangle^2\left|D_x\right|^{\frac{2}{3}} \theta^{\mathrm{L}}\|_{L^2} \\
			& + \nu^{-\frac{1}{3}-\frac{\delta}{6}}\|e^{c \nu^{\frac{1}{3}} \lambda\left(D_x\right)}\langle D_x\rangle\langle\frac{1}{D_x}\rangle^\epsilon \partial_x \nabla \phi^{\mathrm{NL}}\|_{L^2}\|e^{c_0 \nu\left|D_x\right|^2 t^3}\langle D_x, D_y+t D_x\rangle^4\langle\frac{1}{D_x}\rangle^2 \theta^{\mathrm{L}}\|_{L^1 \cap L^2} \\
			& \quad \times\left(\|e^{c \nu^{\frac{1}{3}} \lambda\left(D_x\right) t}\langle D_x\rangle\langle\frac{1}{D_x}\rangle^\epsilon \sqrt{\Upsilon} |D_x|^\frac13  \theta^{\mathrm{NL}}\|_{L^2}+\nu^{\frac{1}{6}}\|e^{c \nu^{\frac{1}{3}} \lambda\left(D_x\right) t}\langle D_x\rangle\langle\frac{1}{D_x}\rangle^\epsilon\left|D_x\right|^{\frac{2}{3}} \theta^{\mathrm{NL}}\|_{L^2}\right).
		\end{align*}

		\underline{\bf Step VII: Estimate of $   u^{\mathrm{NL}} \cdot  \nabla \langle D_x \rangle^\frac13 \theta^{\mathrm{NL}}$.} Due to the different decay behaviors of $u_1^{\mathrm{NL}} \px    $ and $u_2^{\mathrm{NL}} \py   $, our consideration is divided into two parts as follows
		$$
		\begin{aligned}
			& \left|\left(  u^{\mathrm{NL}} \cdot \nabla \langle D_x \rangle^\frac13 \theta^{\mathrm{NL}} \big| \mathcal{M} e^{2 c \nu^{\frac{1}{3}} \lambda\left(D_x\right) t}\langle D_x\rangle^2\langle\frac{1}{D_x}\rangle^{2 \epsilon} \langle D_x \rangle^\frac13 \theta^{\mathrm{NL}} \right)\right| \\
			= & \left|\int_{\mathbb{R}^4} \mathcal{M}(t, k, \xi) e^{2 c \nu^{\frac{1}{3}} \lambda(k) t}\langle k\rangle^2\langle\frac{1}{k}\rangle^{2 \epsilon} \langle k \rangle^\frac13\hat{\theta}_k^{\mathrm{NL}}(\xi) \rt.\\
			&\lt.\quad \times \langle k-l \rangle^\frac13  \frac{\eta(k-l)-l(\xi-\eta)}{|l|^2+|\eta|^2} \hat{\omega}_l^{\mathrm{NL}}(\eta) \hat{\theta}_{k-l}^{\mathrm{NL}}(\xi-\eta) d k d l d \xi d \eta\right| \\
			\leq & I_5+I_6,
		\end{aligned}
		$$ 
		where
		$$
		\begin{gathered}
			I_5 :=\left| \int_{\mathbb{R}^4} \mathcal{M}(t, k, \xi) e^{2 c \nu^{\frac{1}{3}} \lambda(k) t}\langle k\rangle^2\langle\frac{1}{k}\rangle^{2 \epsilon} \langle k \rangle^\frac13\hat{\theta}_k^{\mathrm{NL}}(\xi) \langle k-l \rangle^\frac13  \frac{\eta(k-l) }{|l|^2+|\eta|^2} \hat{\omega}_l^{\mathrm{NL}}(\eta)\right. \\ \lt.
			\times \hat{\theta}_{k-l}^{\mathrm{NL}}(\xi-\eta) d k d l d \xi d \eta \rt| \\
			I_6 : =\left|\int_{\mathbb{R}^4} \mathcal{M}(t, k, \xi) e^{2 c \nu^{\frac{1}{3}} \lambda(k) t}\langle k\rangle^2\langle\frac{1}{k}\rangle^{2 \epsilon} \langle k \rangle^\frac13\hat{\theta}_k^{\mathrm{NL}}(\xi) \langle k-l \rangle^\frac13  \frac{ l(\xi - \eta )}{|l|^2+|\eta|^2} \hat{\omega}_l^{\mathrm{NL}}(\eta) \hat{\theta}_{k-l}^{\mathrm{NL}}(\xi-\eta) d k d l d \xi d \eta\right| .
		\end{gathered}
		$$  
		
		\underline{ Step VII.1: Estimate of $ I_5$.}   
		For $I_5$, we divide the integration domain into three parts according to the relationship between $|k|$ and $|k - l|$, namely the near-resonant region ($\frac{|k-l|}{2} \leq |k| \leq 2|k-l|$), the high–low interaction region ($|k| > 2|k-l| $), and the low–high interaction region ($2|k| < |k-l| $). \\
		$\bullet$~Case 1:   $\frac{|k-l|}{2} \leq |k| \leq 2|k-l|$. Using $\langle k \rangle \lesssim \langle k-l \rangle$ and $\langle \frac1k \rangle^\epsilon \lesssim \langle \frac1{k-l}\rangle^\epsilon$,  we infer that
		\begin{align} \label{eq:temp33}
			& \left| \int_{\frac{|k-l|}{2} \leq |k| \leq 2|k-l|} \mathcal{M}(t, k, \xi) e^{2 c \nu^{\frac{1}{3}} \lambda(k) t}\langle k\rangle^2\langle\frac{1}{k}\rangle^{2 \epsilon} \langle k \rangle^\frac13\hat{\theta}_k^{\mathrm{NL}}(\xi) \langle k-l \rangle^\frac13  \frac{\eta(k-l)}{|l|^2+|\eta|^2}\right. \no\\
			&\lt. \times \hat{\omega}_l^{\mathrm{NL}}(\eta) \hat{\theta}_{k-l}^{\mathrm{NL}}(\xi-\eta) d k d l d \xi d \eta \rt| \no\\
			& \no\lesssim\|e^{c \nu^{\frac{1}{3}} \lambda(k) t}  \langle k \rangle \langle \frac1k \rangle^\epsilon |k|^\frac13\langle k \rangle^\frac13 \hat{\theta}_k^{\mathrm{NL}}(\xi)\|_{L_{k, \xi}^2}\|e^{c \nu^{\frac{1}{3}} \lambda(l) t}\langle l\rangle\langle\frac{1}{l}\rangle^\epsilon \hat{\omega}_l^{\mathrm{NL}}(\eta)\|_{L_{l, \eta}^2}\|\langle l\rangle^{-1}\langle\frac{1}{l}\rangle^{-\epsilon}\| \frac{1}{|l|+|\eta|}\|_{L_\eta^2}\|_{L_l^2} \\
			&\no \quad \times\|e^{c \nu^{\frac{1}{3}} \lambda(k-l) t}\langle k-l \rangle \langle \frac{1}{k-l} \rangle^\epsilon |k-l|^\frac23 \langle k-l \rangle^\frac13 \hat{\theta}_{k-l}^{\mathrm{NL}}(\xi-\eta)\|_{L_{k-l, \xi-\eta}^2} \\
			&\no \lesssim\|e^{c \nu^{\frac{1}{3}} \lambda\left(D_x\right) t}\langle D_x\rangle    \langle \frac{1}{D_x}\rangle^\epsilon  \left|D_x\right|^{\frac{1}{3}}\langle D_x \rangle^\frac13 \theta^{\mathrm{NL}}\|_{L^2}^\frac32 \|e^{c \nu^{\frac{1}{3}} \lambda\left(D_x\right) t}\langle D_x\rangle\langle\frac{1}{D_x}\rangle^\epsilon \omega^{\mathrm{NL}}\|_{L^2} \\
			& \quad \times\|e^{c \nu^{\frac{1}{3}} \lambda\left(D_x\right) t}\langle D_x\rangle\langle\frac{1}{D_x}\rangle^\epsilon \px \langle D_x \rangle^{\frac{1}{3}} \theta^{\mathrm{NL}}\|_{L^2}^\frac12.
		\end{align}
		$\bullet$~Case 2:   $|k| > 2|k-l| $. By $ |k|\sim |l|$ and $|k-l|\lesssim |k| \lesssim |l|$,  we obtain that
		\begin{align} \label{eq:temp34}
			& \no\left| \int_{|k| > 2|k-l| } \mathcal{M}(t, k, \xi) e^{2 c \nu^{\frac{1}{3}} \lambda(k) t}\langle k\rangle^2\langle\frac{1}{k}\rangle^{2 \epsilon} \langle k \rangle^\frac13\hat{\theta}_k^{\mathrm{NL}}(\xi) \langle k-l \rangle^\frac13  \frac{\eta(k-l)}{|l|^2+|\eta|^2}\right. \\
			&\lt. \no\times \hat{\omega}_l^{\mathrm{NL}}(\eta) \hat{\theta}_{k-l}^{\mathrm{NL}}(\xi-\eta) d k d l d \xi d \eta \rt| \\
			&\no \lesssim\|e^{c \nu^{\frac{1}{3}} \lambda(k) t}  \langle k \rangle \langle \frac1k \rangle^\epsilon \langle k \rangle^\frac13  \hat{\theta}_k^{\mathrm{NL}}(\xi)\|_{L_{k, \xi}^2}\|e^{c \nu^{\frac{1}{3}} \lambda(l) t}\langle l\rangle\langle\frac{1}{l}\rangle^\epsilon   \frac{l\eta}{|l|^2  + |\eta|^2} \hat{\omega}_l^{\mathrm{NL}}(\eta)\|_{L_{l, \eta}^2}         \\
			& \no\quad \times \|\langle k- l\rangle^{-1}\langle\frac{1}{k-l}\rangle^{-\epsilon}   \| \frac{1}{|k-l| + |\xi - \eta|} \|_{L_{\xi - \eta}^2}\|_{L_{k-l}^2}  \\
			& \no\quad \times  \|e^{c \nu^{\frac{1}{3}} \lambda(k-l) t}\langle k-l \rangle  \langle \frac{1}{k-l} \rangle^\epsilon   \sqrt{|k-l|^2 + |\xi - \eta|^2 }\langle k-l \rangle ^\frac13  \hat{\theta}_{k-l}^{\mathrm{NL}}(\xi-\eta)\|_{L_{k-l, \xi-\eta}^2} \\
			&\no \lesssim\|e^{c \nu^{\frac{1}{3}} \lambda\left(D_x\right) t}\langle D_x\rangle    \langle \frac{1}{D_x}\rangle^\epsilon  \langle D_x \rangle^{\frac{1}{3}} \theta^{\mathrm{NL}}\|_{L^2} \|e^{c \nu^{\frac{1}{3}} \lambda\left(D_x\right) t}\langle D_x\rangle\langle\frac{1}{D_x}\rangle^\epsilon \px \nabla \phi^{\mathrm{NL}}\|_{L^2} \\
			& \quad \times\|e^{c \nu^{\frac{1}{3}} \lambda\left(D_x\right) t}\langle D_x\rangle\langle\frac{1}{D_x}\rangle^\epsilon \nabla \langle D_x \rangle^{\frac{1}{3}} \theta^{\mathrm{NL}}\|_{L^2} .
		\end{align}
		$\bullet$~Case 3:  $2|k| < |k-l| $. From $ \langle k \rangle^2 \lesssim \langle l \rangle \langle k-l \rangle$ and $$1 \lesssim \langle \frac{1}{k-l} \rangle^\frac12 |k-l|^{\frac12} \lesssim \langle\frac{1}{l}\rangle^\epsilon\langle\frac{1}{k-l}\rangle^\epsilon\left(1+\langle\frac{1}{k}\rangle^{\frac{1}{2}-2 \epsilon}\right)|k-l|^{\frac12},$$  we get
		\begin{align} \label{eq:temp35}
			& \no \left| \int_{2|k| <|k-l| } \mathcal{M}(t, k, \xi) e^{2 c \nu^{\frac{1}{3}} \lambda(k) t}\langle k\rangle^2\langle\frac{1}{k}\rangle^{2 \epsilon} \langle k \rangle^\frac13\hat{\theta}_k^{\mathrm{NL}}(\xi) \langle k-l \rangle^\frac13  \frac{\eta(k-l)}{|l|^2+|\eta|^2}\right. \\
			& \lt. \no\times \hat{\omega}_l^{\mathrm{NL}}(\eta) \hat{\theta}_{k-l}^{\mathrm{NL}}(\xi-\eta) d k d l d \xi d \eta \rt| \\
			&\no \lesssim\|e^{c \nu^{\frac{1}{3}} \lambda(k) t}  \langle k \rangle \langle \frac1k \rangle^\epsilon \langle k \rangle^\frac13  \hat{\theta}_k^{\mathrm{NL}}(\xi)\|_{L_{k, \xi}^2}\|e^{c \nu^{\frac{1}{3}} \lambda(l) t}\langle l\rangle\langle\frac{1}{l}\rangle^\epsilon   \frac{l\eta}{|l|^2  + |\eta|^2} \hat{\omega}_l^{\mathrm{NL}}(\eta)\|_{L_{l, \eta}^2}         \\
			& \no\quad \times \|\langle k\rangle^{-1}   (\langle \frac1k \rangle^\epsilon  + \langle \frac1k \rangle^{\frac12-\epsilon}) |k-l|^\frac12   \| \frac{1}{|k-l| + |\xi - \eta|} \|_{L_{\xi - \eta}^2} \|_{L_{k}^2}  \\
			& \no\quad \times  \|e^{c \nu^{\frac{1}{3}} \lambda(k-l) t}\langle k-l \rangle  \langle \frac{1}{k-l} \rangle^\epsilon   \sqrt{|k-l|^2 + |\xi - \eta|^2 }\langle k-l \rangle ^\frac13  \hat{\theta}_{k-l}^{\mathrm{NL}}(\xi-\eta)\|_{L_{k-l, \xi-\eta}^2} \\
			& \no\lesssim\|e^{c \nu^{\frac{1}{3}} \lambda\left(D_x\right) t}\langle D_x\rangle    \langle \frac{1}{D_x}\rangle^\epsilon  \langle D_x \rangle^{\frac{1}{3}} \theta^{\mathrm{NL}}\|_{L^2} \|e^{c \nu^{\frac{1}{3}} \lambda\left(D_x\right) t}\langle D_x\rangle\langle\frac{1}{D_x}\rangle^\epsilon \px \nabla \phi^{\mathrm{NL}}\|_{L^2} \\
			& \quad \times\|e^{c \nu^{\frac{1}{3}} \lambda\left(D_x\right) t}\langle D_x\rangle\langle\frac{1}{D_x}\rangle^\epsilon \nabla \langle D_x \rangle^{\frac{1}{3}} \theta^{\mathrm{NL}}\|_{L^2} .
		\end{align}
		
		\underline{ Step VII.2: Estimate of $ I_6$.}  
		We analyze $I_6$ by similarly dividing it into three cases.\\
		$\bullet$~Case 1:  $\frac{|k-l|}{2} \leq |k| \leq 2|k-l|$. It follows from 
		\begin{equation} \label{eq:trick9}
			\langle k \rangle \langle \frac1k \rangle^\epsilon\lesssim \langle k-l \rangle\langle \frac1{k-l}\rangle^\epsilon
		\end{equation}
		that
		\begin{align} \label{eq:temp36}
			& \no\left| \int_{\frac{|k-l|}{2} \leq |k| \leq 2|k-l|} \mathcal{M}(t, k, \xi) e^{2 c \nu^{\frac{1}{3}} \lambda(k) t}\langle k\rangle^2\langle\frac{1}{k}\rangle^{2 \epsilon} \langle k \rangle^\frac13\hat{\theta}_k^{\mathrm{NL}}(\xi) \langle k-l \rangle^\frac13  \frac{l(\xi - \eta)}{|l|^2+|\eta|^2}\right. \\
			&\lt. \no \times \hat{\omega}_l^{\mathrm{NL}}(\eta) \hat{\theta}_{k-l}^{\mathrm{NL}}(\xi-\eta) d k d l d \xi d \eta \rt| \\
			& \no\lesssim\|e^{c \nu^{\frac{1}{3}} \lambda(k) t}  \langle k \rangle \langle \frac1k \rangle^\epsilon \langle k \rangle^\frac13  \hat{\theta}_k^{\mathrm{NL}}(\xi)\|_{L_{k, \xi}^2}   \|e^{c \nu^{\frac{1}{3}} \lambda(l) t}\langle l\rangle\langle\frac{1}{l}\rangle^\epsilon \frac l{|l| + |\eta|} \hat{\omega}_l^{\mathrm{NL}}(\eta)\|_{L_{l, \eta}^2}     \\
			&\no \quad \times \|\langle l\rangle^{-1}\langle\frac{1}{l}\rangle^{-\epsilon}\| \frac{1}{|l|+|\eta|}\|_{L_\eta^2}\|_{L_l^2} \|e^{c \nu^{\frac{1}{3}} \lambda(k-l) t}\langle k-l \rangle \langle \frac{1}{k-l} \rangle^\epsilon \langle k-l \rangle ^ \frac13 (\xi - \eta) \hat{\theta}_{k-l}^{\mathrm{NL}}(\xi-\eta)\|_{L_{k-l, \xi-\eta}^2} \\
			&\no \lesssim\|e^{c \nu^{\frac{1}{3}} \lambda\left(D_x\right) t}\langle D_x\rangle    \langle \frac{1}{D_x}\rangle^\epsilon  \langle D_x \rangle^{\frac{1}{3}} \theta^{\mathrm{NL}}\|_{L^2} \|e^{c \nu^{\frac{1}{3}} \lambda\left(D_x\right) t}\langle D_x\rangle\langle\frac{1}{D_x}\rangle^\epsilon\px \nabla \phi^{\mathrm{NL}}\|_{L^2} \\
			& \quad \times\|e^{c \nu^{\frac{1}{3}} \lambda\left(D_x\right) t}\langle D_x\rangle\langle\frac{1}{D_x}\rangle^\epsilon \py \langle D_x \rangle^{\frac{1}{3}} \theta^{\mathrm{NL}}\|_{L^2}.
		\end{align}
		$\bullet$~Case 2:   $|k| > 2|k-l| $. Thanks to $|k-l|\lesssim |k| \sim |l|$,  we get that
		\begin{align} \label{eq:temp37}
			&\no \left| \int_{|k| > 2|k-l| } \mathcal{M}(t, k, \xi) e^{2 c \nu^{\frac{1}{3}} \lambda(k) t}\langle k\rangle^2\langle\frac{1}{k}\rangle^{2 \epsilon} \langle k \rangle^\frac13\hat{\theta}_k^{\mathrm{NL}}(\xi) \langle k-l \rangle^\frac13  \frac{l(\xi - \eta)}{|l|^2+|\eta|^2}\right. \\
			&\lt. \no \times \hat{\omega}_l^{\mathrm{NL}}(\eta) \hat{\theta}_{k-l}^{\mathrm{NL}}(\xi-\eta) d k d l d \xi d \eta \rt| \\
			&\no \lesssim\|e^{c \nu^{\frac{1}{3}} \lambda(k) t}  \langle k \rangle \langle \frac1k \rangle^\epsilon \langle k \rangle^\frac13  \hat{\theta}_k^{\mathrm{NL}}(\xi)\|_{L_{k, \xi}^2}\|e^{c \nu^{\frac{1}{3}} \lambda(l) t}\langle l\rangle\langle\frac{1}{l}\rangle^\epsilon   \frac{|l|(|l|+|\eta|)}{|l|^2  + |\eta|^2} \hat{\omega}_l^{\mathrm{NL}}(\eta)\|_{L_{l, \eta}^2}         \\
			&\no \quad \times \|\langle k- l\rangle^{-1}\langle\frac{1}{k-l}\rangle^{-\epsilon}   \| \frac{1}{|l| + | \eta|} \|_{L_{\eta}^2}\|_{L_{k-l}^2}  \\
			&\no \quad \times  \|e^{c \nu^{\frac{1}{3}} \lambda(k-l) t}\langle k-l \rangle  \langle \frac{1}{k-l} \rangle^\epsilon   |\xi - \eta|\langle k-l \rangle ^\frac13  \hat{\theta}_{k-l}^{\mathrm{NL}}(\xi-\eta)\|_{L_{k-l, \xi-\eta}^2} \\
			&\no \lesssim\|e^{c \nu^{\frac{1}{3}} \lambda\left(D_x\right) t}\langle D_x\rangle    \langle \frac{1}{D_x}\rangle^\epsilon  \langle D_x \rangle^{\frac{1}{3}} \theta^{\mathrm{NL}}\|_{L^2} \|e^{c \nu^{\frac{1}{3}} \lambda\left(D_x\right) t}\langle D_x\rangle\langle\frac{1}{D_x}\rangle^\epsilon \px \nabla \phi^{\mathrm{NL}}\|_{L^2} \\
			& \quad \times\|e^{c \nu^{\frac{1}{3}} \lambda\left(D_x\right) t}\langle D_x\rangle\langle\frac{1}{D_x}\rangle^\epsilon \py \langle D_x \rangle^{\frac{1}{3}} \theta^{\mathrm{NL}}\|_{L^2} .
		\end{align}
		$\bullet$~Case 3:   $2|k| < |k-l| $. Noting that $ \langle k \rangle^2 \lesssim \langle l \rangle \langle k-l \rangle$, $|k-l|^\frac12 |l|^{-\frac12} \lesssim 1$ and 
		\begin{equation} \label{eq:trick10}
			1 \lesssim \langle \frac{1}{k-l} \rangle^\frac12 |k-l|^{\frac12} \lesssim \langle\frac{1}{l}\rangle^\epsilon\langle\frac{1}{k-l}\rangle^\epsilon\left(1+\langle\frac{1}{k}\rangle^{\frac{1}{2}-2 \epsilon}\right)|k-l|^{\frac12},
		\end{equation} 
		we have that
		\begin{align} \label{eq:temp38}
			&\no \left| \int_{2|k| <|k-l| } \mathcal{M}(t, k, \xi) e^{2 c \nu^{\frac{1}{3}} \lambda(k) t}\langle k\rangle^2\langle\frac{1}{k}\rangle^{2 \epsilon} \langle k \rangle^\frac13\hat{\theta}_k^{\mathrm{NL}}(\xi) \langle k-l \rangle^\frac13  \frac{l(\xi - \eta)}{|l|^2+|\eta|^2}\right. \\
			&\lt. \no \times \hat{\omega}_l^{\mathrm{NL}}(\eta) \hat{\theta}_{k-l}^{\mathrm{NL}}(\xi-\eta) d k d l d \xi d \eta \rt| \\
			&\no \lesssim\|e^{c \nu^{\frac{1}{3}} \lambda(k) t}  \langle k \rangle \langle \frac1k \rangle^\epsilon \langle k \rangle^\frac13  \hat{\theta}_k^{\mathrm{NL}}(\xi)\|_{L_{k, \xi}^2}\|e^{c \nu^{\frac{1}{3}} \lambda(l) t}\langle l\rangle\langle\frac{1}{l}\rangle^\epsilon   \frac{|l|(|l| + |\eta|)}{|l|^2  + |\eta|^2} \hat{\omega}_l^{\mathrm{NL}}(\eta)\|_{L_{l, \eta}^2}         \\
			&\no \quad \times \|\langle k\rangle^{-1}   (\langle \frac1k \rangle^\epsilon  + \langle \frac1k \rangle^{\frac12-\epsilon}) |k-l|^\frac12   \| \frac{1}{|l| + |\eta|} \|_{L_{ \eta}^2} \|_{L_{k}^2}  \\
			&\no \quad \times  \|e^{c \nu^{\frac{1}{3}} \lambda(k-l) t}\langle k-l \rangle  \langle \frac{1}{k-l} \rangle^\epsilon    |\xi - \eta|\langle k-l \rangle ^\frac13  \hat{\theta}_{k-l}^{\mathrm{NL}}(\xi-\eta)\|_{L_{k-l, \xi-\eta}^2} \\
			&\no \lesssim\|e^{c \nu^{\frac{1}{3}} \lambda\left(D_x\right) t}\langle D_x\rangle    \langle \frac{1}{D_x}\rangle^\epsilon  \langle D_x \rangle^{\frac{1}{3}} \theta^{\mathrm{NL}}\|_{L^2} \|e^{c \nu^{\frac{1}{3}} \lambda\left(D_x\right) t}\langle D_x\rangle\langle\frac{1}{D_x}\rangle^\epsilon \px \nabla \phi^{\mathrm{NL}}\|_{L^2} \\
			& \quad \times\|e^{c \nu^{\frac{1}{3}} \lambda\left(D_x\right) t}\langle D_x\rangle\langle\frac{1}{D_x}\rangle^\epsilon \py \langle D_x \rangle^{\frac{1}{3}} \theta^{\mathrm{NL}}\|_{L^2} .
		\end{align}

		Adding \eqref{eq:temp33}, \eqref{eq:temp34}, \eqref{eq:temp35}, \eqref{eq:temp36}, \eqref{eq:temp37} and \eqref{eq:temp38} together, it holds that
		\begin{equation*}\begin{aligned}
				&\left|\left(  u^{\mathrm{NL}} \cdot \nabla \langle D_x \rangle^\frac13 \theta^{\mathrm{NL}} \big| \mathcal{M} e^{2 c \nu^{\frac{1}{3}} \lambda\left(D_x\right) t}\langle D_x\rangle^2\langle\frac{1}{D_x}\rangle^{2 \epsilon} \langle D_x \rangle^\frac13 \theta^{\mathrm{NL}}\right)\right| \\
				\lesssim& \|e^{c \nu^{\frac{1}{3}} \lambda\left(D_x\right) t}\langle D_x\rangle    \langle \frac{1}{D_x}\rangle^\epsilon  \left|D_x\right|^{\frac{1}{3}}\langle D_x \rangle^\frac13 \theta^{\mathrm{NL}}\|_{L^2}^\frac32 \|e^{c \nu^{\frac{1}{3}} \lambda\left(D_x\right) t}\langle D_x\rangle\langle\frac{1}{D_x}\rangle^\epsilon \omega^{\mathrm{NL}}\|_{L^2} \\
				& \quad \times\|e^{c \nu^{\frac{1}{3}} \lambda\left(D_x\right) t}\langle D_x\rangle\langle\frac{1}{D_x}\rangle^\epsilon \nabla  \langle D_x \rangle^{\frac{1}{3}} \theta^{\mathrm{NL}}\|_{L^2}^\frac12 \\
				&+ \|e^{c \nu^{\frac{1}{3}} \lambda\left(D_x\right) t}\langle D_x\rangle    \langle \frac{1}{D_x}\rangle^\epsilon  \langle D_x \rangle^{\frac{1}{3}} \theta^{\mathrm{NL}}\|_{L^2} \|e^{c \nu^{\frac{1}{3}} \lambda\left(D_x\right) t}\langle D_x\rangle\langle\frac{1}{D_x}\rangle^\epsilon \px \nabla \phi^{\mathrm{NL}}\|_{L^2} \\
				& \quad \times\|e^{c \nu^{\frac{1}{3}} \lambda\left(D_x\right) t}\langle D_x\rangle\langle\frac{1}{D_x}\rangle^\epsilon \nabla \langle D_x \rangle^{\frac{1}{3}} \theta^{\mathrm{NL}}\|_{L^2}.
		\end{aligned}\end{equation*}

		\underline{\bf Step VIII: Estimate of $ |D_x|^\frac13  u^{\mathrm{NL}} \cdot  \nabla  \theta^{\mathrm{NL}}$.} Note that
		$$
		\begin{aligned}
			& \left|\left( |D_x|^\frac13 u^{\mathrm{NL}} \cdot \nabla \theta^{\mathrm{NL}} \big| \mathcal{M} e^{2 c \nu^{\frac{1}{3}} \lambda\left(D_x\right) t}\langle D_x\rangle^2\langle\frac{1}{D_x}\rangle^{2 \epsilon} |D_x|^\frac13 \theta^{\mathrm{NL}}\right)\right| \\
			= & \left|\int_{\mathbb{R}^4} \mathcal{M}(t, k, \xi) e^{2 c \nu^{\frac{1}{3}} \lambda(k) t}\langle k\rangle^2\langle\frac{1}{k}\rangle^{2 \epsilon} |k|^\frac13\hat{\theta}_k^{\mathrm{NL}}(\xi) |l|^\frac13  \frac{\eta(k-l)-l(\xi-\eta)}{|l|^2+|\eta|^2} \hat{\omega}_l^{\mathrm{NL}}(\eta) \hat{\theta}_{k-l}^{\mathrm{NL}}(\xi-\eta) d k d l d \xi d \eta\right| \\
			\leq & I_7 + I_8,
		\end{aligned}
		$$
		where
		$$
		\begin{gathered}
			I_7:=\left| \int_{\mathbb{R}^4} \mathcal{M}(t, k, \xi) e^{2 c \nu^{\frac{1}{3}} \lambda(k) t}\langle k\rangle^2\langle\frac{1}{k}\rangle^{2 \epsilon} |k|^\frac13\hat{\theta}_k^{\mathrm{NL}}(\xi) |l|^\frac13  \frac{\eta(k-l) }{|l|^2+|\eta|^2} \hat{\omega}_l^{\mathrm{NL}}(\eta)\right. \\ \lt.
			\times \hat{\theta}_{k-l}^{\mathrm{NL}}(\xi-\eta) d k d l d \xi d \eta \rt| \\
			I_8:=\left|\int_{\mathbb{R}^4} \mathcal{M}(t, k, \xi) e^{2 c \nu^{\frac{1}{3}} \lambda(k) t}\langle k\rangle^2\langle\frac{1}{k}\rangle^{2 \epsilon} |k|^\frac13\hat{\theta}_k^{\mathrm{NL}}(\xi) |l|^\frac13  \frac{ l(\xi - \eta )}{|l|^2+|\eta|^2} \hat{\omega}_l^{\mathrm{NL}}(\eta) \hat{\theta}_{k-l}^{\mathrm{NL}}(\xi-\eta) d k d l d \xi d \eta\right| .
		\end{gathered}
		$$  
		
		\underline{ Step VIII.1: Estimate of $I_7$.}    
		Now, we estimate $I_7$.\\
		$\bullet$~Case 1:   $\frac{|k-l|}{2} \leq |k| \leq 2|k-l|$. \eqref{eq:trick9} gives
		\begin{align*}
			& \bigg| \int_{\frac{|k-l|}{2} \leq |k| \leq 2|k-l|} \mathcal{M}(t, k, \xi) e^{2 c \nu^{\frac{1}{3}} \lambda(k) t}\langle k\rangle^2\langle\frac{1}{k}\rangle^{2 \epsilon} |k|^\frac13\hat{\theta}_k^{\mathrm{NL}}(\xi) |l|^\frac13  \frac{\eta(k-l)}{|l|^2+|\eta|^2}  \\
			& \times \hat{\omega}_l^{\mathrm{NL}}(\eta) \hat{\theta}_{k-l}^{\mathrm{NL}}(\xi-\eta) d k d l d \xi d \eta \bigg| \\
			& \lesssim\|e^{c \nu^{\frac{1}{3}} \lambda(k) t}  \langle k \rangle \langle \frac1k \rangle^\epsilon |k|^\frac23  \hat{\theta}_k^{\mathrm{NL}}(\xi)\|_{L_{k, \xi}^2}\|e^{c \nu^{\frac{1}{3}} \lambda(l) t}\langle l\rangle\langle\frac{1}{l}\rangle^\epsilon \hat{\omega}_l^{\mathrm{NL}}(\eta)\|_{L_{l, \eta}^2}  \\
			& \quad \times   \||l|^\frac13 \langle l\rangle^{-1}\langle\frac{1}{l}\rangle^{-\epsilon}\| \frac{1}{|l|+|\eta|}\|_{L_\eta^2}\|_{L_l^2} \|e^{c \nu^{\frac{1}{3}} \lambda(k-l) t}\langle k-l \rangle \langle \frac{1}{k-l} \rangle^\epsilon |k-l|^{\frac23}  \hat{\theta}_{k-l}^{\mathrm{NL}}(\xi-\eta)\|_{L_{k-l, \xi-\eta}^2} \\
			& \lesssim\|e^{c \nu^{\frac{1}{3}} \lambda\left(D_x\right) t}\langle D_x\rangle    \langle \frac{1}{D_x}\rangle^\epsilon  \left|D_x\right|^{\frac{2}{3}} \theta^{\mathrm{NL}}\|_{L^2}^2 \|e^{c \nu^{\frac{1}{3}} \lambda\left(D_x\right) t}\langle D_x\rangle\langle\frac{1}{D_x}\rangle^\epsilon \omega^{\mathrm{NL}}\|_{L^2} .
		\end{align*}
		$\bullet$~Case 2:   $|k| > 2|k-l| $. Making use of $ |k|\sim |l|$ and $|k-l|\lesssim  |l|^\frac23 |k-l|^\frac13 $,  we obtain that
		\begin{align*}
			& \bigg| \int_{|k| > 2|k-l| } \mathcal{M}(t, k, \xi) e^{2 c \nu^{\frac{1}{3}} \lambda(k) t}\langle k\rangle^2\langle\frac{1}{k}\rangle^{2 \epsilon} |k|^\frac13\hat{\theta}_k^{\mathrm{NL}}(\xi) |l|^\frac13  \frac{\eta(k-l)}{|l|^2+|\eta|^2} \\
			& \times \hat{\omega}_l^{\mathrm{NL}}(\eta) \hat{\theta}_{k-l}^{\mathrm{NL}}(\xi-\eta) d k d l d \xi d \eta \bigg| \\
			& \lesssim\|e^{c \nu^{\frac{1}{3}} \lambda(k) t}  \langle k \rangle \langle \frac1k \rangle^\epsilon |k|^\frac13  \hat{\theta}_k^{\mathrm{NL}}(\xi)\|_{L_{k, \xi}^2}\|e^{c \nu^{\frac{1}{3}} \lambda(l) t}\langle l\rangle\langle\frac{1}{l}\rangle^\epsilon   \frac{|l|\eta}{|l|^2  + |\eta|^2} \hat{\omega}_l^{\mathrm{NL}}(\eta)\|_{L_{l, \eta}^2}         \\
			& \quad \times \||k-l|^\frac13 \langle k- l\rangle^{-1}\langle\frac{1}{k-l}\rangle^{-\epsilon}   \| \frac{1}{|k-l| + |\xi - \eta|} \|_{L_{\xi - \eta}^2}\|_{L_{k-l}^2}  \\
			& \quad \times  \|e^{c \nu^{\frac{1}{3}} \lambda(k-l) t}\langle k-l \rangle  \langle \frac{1}{k-l} \rangle^\epsilon   \sqrt{|k-l|^2 + |\xi - \eta|^2 } \hat{\theta}_{k-l}^{\mathrm{NL}}(\xi-\eta)\|_{L_{k-l, \xi-\eta}^2} \\
			& \lesssim\|e^{c \nu^{\frac{1}{3}} \lambda\left(D_x\right) t}\langle D_x\rangle    \langle \frac{1}{D_x}\rangle^\epsilon  \left|D_x\right|^{\frac{1}{3}} \theta^{\mathrm{NL}}\|_{L^2} \|e^{c \nu^{\frac{1}{3}} \lambda\left(D_x\right) t}\langle D_x\rangle\langle\frac{1}{D_x}\rangle^\epsilon \px \nabla \phi^{\mathrm{NL}}\|_{L^2} \\
			& \quad \times\|e^{c \nu^{\frac{1}{3}} \lambda\left(D_x\right) t}\langle D_x\rangle\langle\frac{1}{D_x}\rangle^\epsilon \nabla  \theta^{\mathrm{NL}}\|_{L^2} .
		\end{align*}
		$\bullet$~Case 3:   $2|k| < |k-l| $. We deduce from $ \langle k \rangle^2 \lesssim \langle l \rangle \langle k-l \rangle$, $|l|^\frac13 \sim |k-l|^\frac13 $ and \eqref{eq:trick10} that
		\begin{align*}
			& \bigg| \int_{2|k| <|k-l| } \mathcal{M}(t, k, \xi) e^{2 c \nu^{\frac{1}{3}} \lambda(k) t}\langle k\rangle^2\langle\frac{1}{k}\rangle^{2 \epsilon} |k|^\frac13\hat{\theta}_k^{\mathrm{NL}}(\xi) |l|^\frac13  \frac{\eta(k-l)}{|l|^2+|\eta|^2} \\
			& \times \hat{\omega}_l^{\mathrm{NL}}(\eta) \hat{\theta}_{k-l}^{\mathrm{NL}}(\xi-\eta) d k d l d \xi d \eta \bigg| \\
			& \lesssim\|e^{c \nu^{\frac{1}{3}} \lambda(k) t}  \langle k \rangle \langle \frac1k \rangle^\epsilon |k|^\frac13  \hat{\theta}_k^{\mathrm{NL}}(\xi)\|_{L_{k, \xi}^2}\|e^{c \nu^{\frac{1}{3}} \lambda(l) t}\langle l\rangle\langle\frac{1}{l}\rangle^\epsilon   \frac{l\eta}{|l|^2  + |\eta|^2} \hat{\omega}_l^{\mathrm{NL}}(\eta)\|_{L_{l, \eta}^2}         \\
			& \quad \times \|\langle k\rangle^{-1}   (\langle \frac1k \rangle^\epsilon  + \langle \frac1k \rangle^{\frac12-\epsilon}) |k-l|^\frac12   \| \frac{1}{|k-l| + |\xi - \eta|} \|_{L_{\xi - \eta}^2} \|_{L_{k}^2}  \\
			& \quad \times  \|e^{c \nu^{\frac{1}{3}} \lambda(k-l) t}\langle k-l \rangle  \langle \frac{1}{k-l} \rangle^\epsilon   \sqrt{|k-l|^2 + |\xi - \eta|^2 }|k-l| ^\frac13  \hat{\theta}_{k-l}^{\mathrm{NL}}(\xi-\eta)\|_{L_{k-l, \xi-\eta}^2} \\
			& \lesssim\|e^{c \nu^{\frac{1}{3}} \lambda\left(D_x\right) t}\langle D_x\rangle    \langle \frac{1}{D_x}\rangle^\epsilon  \left|D_x\right|^{\frac{1}{3}} \theta^{\mathrm{NL}}\|_{L^2} \|e^{c \nu^{\frac{1}{3}} \lambda\left(D_x\right) t}\langle D_x\rangle\langle\frac{1}{D_x}\rangle^\epsilon \px \nabla \phi^{\mathrm{NL}}\|_{L^2} \\
			& \quad \times\|e^{c \nu^{\frac{1}{3}} \lambda\left(D_x\right) t}\langle D_x\rangle\langle\frac{1}{D_x}\rangle^\epsilon \nabla \left|D_x\right|^{\frac{1}{3}} \theta^{\mathrm{NL}}\|_{L^2} .
		\end{align*}
		
		\underline{ Step VIII.2: Estimate of $I_8$.} 
		Then,  we consider $I_8$.\\
		$\bullet$~Case 1:   $\frac{|k-l|}{2} \leq |k| \leq 2|k-l|$. \eqref{eq:trick9} shows
		\begin{align*}
			& \bigg| \int_{\frac{|k-l|}{2} \leq |k| \leq 2|k-l|} \mathcal{M}(t, k, \xi) e^{2 c \nu^{\frac{1}{3}} \lambda(k) t}\langle k\rangle^2\langle\frac{1}{k}\rangle^{2 \epsilon} |k|^\frac13\hat{\theta}_k^{\mathrm{NL}}(\xi) |l|^\frac13  \frac{l(\xi - \eta)}{|l|^2+|\eta|^2}  \\
			& \times \hat{\omega}_l^{\mathrm{NL}}(\eta) \hat{\theta}_{k-l}^{\mathrm{NL}}(\xi-\eta) d k d l d \xi d \eta \bigg| \\
			& \lesssim\|e^{c \nu^{\frac{1}{3}} \lambda(k) t}  \langle k \rangle \langle \frac1k \rangle^\epsilon |k|^\frac13  \hat{\theta}_k^{\mathrm{NL}}(\xi)\|_{L_{k, \xi}^2}   \|e^{c \nu^{\frac{1}{3}} \lambda(l) t}\langle l\rangle\langle\frac{1}{l}\rangle^\epsilon \frac l{|l| + |\eta|} \hat{\omega}_l^{\mathrm{NL}}(\eta)\|_{L_{l, \eta}^2}     \\
			& \quad \times \||l|^\frac13 \langle l\rangle^{-1}\langle\frac{1}{l}\rangle^{-\epsilon}\| \frac{1}{|l|+|\eta|}\|_{L_\eta^2}\|_{L_l^2} \|e^{c \nu^{\frac{1}{3}} \lambda(k-l) t}\langle k-l \rangle \langle \frac{1}{k-l} \rangle^\epsilon  (\xi - \eta) \hat{\theta}_{k-l}^{\mathrm{NL}}(\xi-\eta)\|_{L_{k-l, \xi-\eta}^2} \\
			& \lesssim\|e^{c \nu^{\frac{1}{3}} \lambda\left(D_x\right) t}\langle D_x\rangle    \langle \frac{1}{D_x}\rangle^\epsilon  \left|D_x\right|^{\frac{1}{3}} \theta^{\mathrm{NL}}\|_{L^2} \|e^{c \nu^{\frac{1}{3}} \lambda\left(D_x\right) t}\langle D_x\rangle\langle\frac{1}{D_x}\rangle^\epsilon\px \nabla \phi^{\mathrm{NL}}\|_{L^2} \\
			& \quad \times\|e^{c \nu^{\frac{1}{3}} \lambda\left(D_x\right) t}\langle D_x\rangle\langle\frac{1}{D_x}\rangle^\epsilon \py   \theta^{\mathrm{NL}}\|_{L^2}.
		\end{align*}
		$\bullet$~Case 2:   $|k| > 2|k-l| $. Due to $ |k|\sim |l|$ and $|l|^\frac13 \lesssim |l|^\frac12 |k-l|^{-\frac16}$,  we have that
		\begin{align*}
			& \bigg| \int_{|k| > 2|k-l| } \mathcal{M}(t, k, \xi) e^{2 c \nu^{\frac{1}{3}} \lambda(k) t}\langle k\rangle^2\langle\frac{1}{k}\rangle^{2 \epsilon} |k|^\frac13\hat{\theta}_k^{\mathrm{NL}}(\xi) |l|^\frac13  \frac{l(\xi - \eta)}{|l|^2+|\eta|^2}  \\
			& \times \hat{\omega}_l^{\mathrm{NL}}(\eta) \hat{\theta}_{k-l}^{\mathrm{NL}}(\xi-\eta) d k d l d \xi d \eta \bigg|\\
			& \lesssim\|e^{c \nu^{\frac{1}{3}} \lambda(k) t}  \langle k \rangle \langle \frac1k \rangle^\epsilon |k|^\frac13  \hat{\theta}_k^{\mathrm{NL}}(\xi)\|_{L_{k, \xi}^2}\|e^{c \nu^{\frac{1}{3}} \lambda(l) t}\langle l\rangle\langle\frac{1}{l}\rangle^\epsilon   \frac{|l|(|l|+|\eta|)}{|l|^2  + |\eta|^2} \hat{\omega}_l^{\mathrm{NL}}(\eta)\|_{L_{l, \eta}^2}         \\
			& \quad \times \||l|^\frac12 |k-l|^{-\frac16}\langle k- l\rangle^{-1}\langle\frac{1}{k-l}\rangle^{-\epsilon}   \| \frac{1}{|l| + | \eta|} \|_{L_{\eta}^2}\|_{L_{k-l}^2}  \\
			& \quad \times  \|e^{c \nu^{\frac{1}{3}} \lambda(k-l) t}\langle k-l \rangle  \langle \frac{1}{k-l} \rangle^\epsilon   |\xi - \eta|    \hat{\theta}_{k-l}^{\mathrm{NL}}(\xi-\eta)\|_{L_{k-l, \xi-\eta}^2} \\
			& \lesssim\|e^{c \nu^{\frac{1}{3}} \lambda\left(D_x\right) t}\langle D_x\rangle    \langle \frac{1}{D_x}\rangle^\epsilon  \left|D_x\right|^{\frac{1}{3}} \theta^{\mathrm{NL}}\|_{L^2} \|e^{c \nu^{\frac{1}{3}} \lambda\left(D_x\right) t}\langle D_x\rangle\langle\frac{1}{D_x}\rangle^\epsilon \px \nabla \phi^{\mathrm{NL}}\|_{L^2} \\
			& \quad \times\|e^{c \nu^{\frac{1}{3}} \lambda\left(D_x\right) t}\langle D_x\rangle\langle\frac{1}{D_x}\rangle^\epsilon \py  \theta^{\mathrm{NL}}\|_{L^2},
		\end{align*}
		where we use that $|l|^\frac12 \| (|l|+|\eta|)^{-1}\|_{L_\eta^2} \lesssim 1$. \\
		$\bullet$~Case 3:   $2|k| < |k-l| $. From $ \langle k \rangle^2 \lesssim \langle l \rangle \langle k-l \rangle$, $|k-l|^\frac12 |l|^{-\frac12} \lesssim 1$, $|l|^\frac13 \sim |k-l|^\frac13$ and \eqref{eq:trick10},  we get that
		\begin{align*}
			& \bigg| \int_{2|k| <|k-l| } \mathcal{M}(t, k, \xi) e^{2 c \nu^{\frac{1}{3}} \lambda(k) t}\langle k\rangle^2\langle\frac{1}{k}\rangle^{2 \epsilon} |k|^\frac13\hat{\theta}_k^{\mathrm{NL}}(\xi) |l|^\frac13  \frac{l(\xi - \eta)}{|l|^2+|\eta|^2}  \\
			& \times \hat{\omega}_l^{\mathrm{NL}}(\eta) \hat{\theta}_{k-l}^{\mathrm{NL}}(\xi-\eta) d k d l d \xi d \eta \bigg| \\
			& \lesssim\|e^{c \nu^{\frac{1}{3}} \lambda(k) t}  \langle k \rangle \langle \frac1k \rangle^\epsilon |k|^\frac13  \hat{\theta}_k^{\mathrm{NL}}(\xi)\|_{L_{k, \xi}^2}\|e^{c \nu^{\frac{1}{3}} \lambda(l) t}\langle l\rangle\langle\frac{1}{l}\rangle^\epsilon   \frac{|l|(|l| + |\eta|)}{|l|^2  + |\eta|^2} \hat{\omega}_l^{\mathrm{NL}}(\eta)\|_{L_{l, \eta}^2}         \\
			& \quad \times \|\langle k\rangle^{-1}   (\langle \frac1k \rangle^\epsilon  + \langle \frac1k \rangle^{\frac12-\epsilon}) |k-l|^\frac12   \| \frac{1}{|l| + |\eta|} \|_{L_{ \eta}^2} \|_{L_{k}^2}  \\
			& \quad \times  \|e^{c \nu^{\frac{1}{3}} \lambda(k-l) t}\langle k-l \rangle  \langle \frac{1}{k-l} \rangle^\epsilon    |\xi - \eta||k-l| ^\frac13  \hat{\theta}_{k-l}^{\mathrm{NL}}(\xi-\eta)\|_{L_{k-l, \xi-\eta}^2} \\
			& \lesssim\|e^{c \nu^{\frac{1}{3}} \lambda\left(D_x\right) t}\langle D_x\rangle    \langle \frac{1}{D_x}\rangle^\epsilon  \left|D_x\right|^{\frac{1}{3}} \theta^{\mathrm{NL}}\|_{L^2} \|e^{c \nu^{\frac{1}{3}} \lambda\left(D_x\right) t}\langle D_x\rangle\langle\frac{1}{D_x}\rangle^\epsilon \px \nabla \phi^{\mathrm{NL}}\|_{L^2} \\
			& \quad \times\|e^{c \nu^{\frac{1}{3}} \lambda\left(D_x\right) t}\langle D_x\rangle\langle\frac{1}{D_x}\rangle^\epsilon \py \left|D_x\right|^{\frac{1}{3}} \theta^{\mathrm{NL}}\|_{L^2} .
		\end{align*}
		Combining the estimates of $I_7$ and $I_8$ shows that
		\begin{equation*}\begin{aligned}
				&\left|\left( |D_x|^\frac13 u^{\mathrm{NL}} \cdot \nabla \theta^{\mathrm{NL}} \big| \mathcal{M} e^{2 c \nu^{\frac{1}{3}} \lambda\left(D_x\right) t}\langle D_x\rangle^2\langle\frac{1}{D_x}\rangle^{2 \epsilon} |D_x|^\frac13 \theta^{\mathrm{NL}} \right)\right| \\
				\lesssim& \|e^{c \nu^{\frac{1}{3}} \lambda\left(D_x\right) t}\langle D_x\rangle    \langle \frac{1}{D_x}\rangle^\epsilon  \left|D_x\right|^{\frac{2}{3}} \theta^{\mathrm{NL}}\|_{L^2}^2 \|e^{c \nu^{\frac{1}{3}} \lambda\left(D_x\right) t}\langle D_x\rangle\langle\frac{1}{D_x}\rangle^\epsilon \omega^{\mathrm{NL}}\|_{L^2}  \\
				& + \|e^{c \nu^{\frac{1}{3}} \lambda\left(D_x\right) t}\langle D_x\rangle    \langle \frac{1}{D_x}\rangle^\epsilon  \left|D_x\right|^{\frac{1}{3}} \theta^{\mathrm{NL}}\|_{L^2} \|e^{c \nu^{\frac{1}{3}} \lambda\left(D_x\right) t}\langle D_x\rangle\langle\frac{1}{D_x}\rangle^\epsilon \px \nabla \phi^{\mathrm{NL}}\|_{L^2} \\
				& \quad \times\|e^{c \nu^{\frac{1}{3}} \lambda\left(D_x\right) t}\langle D_x\rangle\langle\frac{1}{D_x}\rangle^\epsilon \nabla \langle D_x\rangle^{\frac{1}{3}} \theta^{\mathrm{NL}}\|_{L^2}. 
		\end{aligned}\end{equation*}

		\underline{\bf Step IX: Collecting the estimates.}  In summary, we have
		\footnotesize{\begin{align*} 
				& \frac{d}{d t}\|\sqrt{\mathcal{M}} e^{c \nu^{\frac{1}{3}} \lambda\left(D_x\right) t}\langle D_x\rangle\langle\frac{1}{D_x}\rangle^\epsilon \langle D_x \rangle^\frac13  \theta^{\mathrm{NL}}\|_{L^2}^2+\nu\|e^{c \nu^{\frac{1}{3}} \lambda\left(D_x\right) t}\langle D_x\rangle\langle\frac{1}{D_x}\rangle^\epsilon \nabla \langle D_x \rangle^\frac13  \theta^{\mathrm{NL}}\|_{L^2}^2 \\
				& \quad+\frac{\nu^{\frac{1}{3}}}{16}\|e^{c \nu^{\frac{1}{3}} \lambda\left(D_x\right) t}\langle D_x\rangle\langle\frac{1}{D_x}\rangle^\epsilon\left|D_x\right|^{\frac{1}{3}}  \langle D_x \rangle ^\frac13 \theta^{\mathrm{NL}}\|_{L^2}^2  
				\\
				& \quad+\|e^{c \nu^{\frac{1}{3}} \lambda\left(D_x\right) t}\langle D_x\rangle\langle\frac{1}{D_x}\rangle^\epsilon \sqrt{\Upsilon\left(t, D_x, D_y\right)} \langle D_x \rangle^\frac13  \theta^{\mathrm{NL}}\|_{L^2}^2 \\
				\lesssim	& \langle t\rangle^{-2} \|e^{c \nu^{\frac{1}{3}} \lambda\left(D_x\right) t}\langle D_x\rangle  \langle \frac{1}{D_x}\rangle^\epsilon \langle D_x \rangle^\frac13  \theta^{\mathrm{NL}}\|_{L^2}    \|e^{c \nu^{\frac{1}{3}} \lambda\left(D_x\right) t}\langle D_x, D_y + tD_x\rangle^5\langle\frac{1}{D_x}\rangle^4 \omega^{\mathrm{L}} \|_{L^2} \\
				&\quad \times \left( \|e^{c \nu^{\frac{1}{3}} \lambda\left(D_x\right) t}\langle D_x, D_y + tD_x\rangle^5\langle\frac{1}{D_x}\rangle^4 \langle D_x \rangle^\frac13  \theta^{\mathrm{L}} \|_{L^2} +\|e^{c \nu^{\frac{1}{3}} \lambda\left(D_x\right) t} \nabla \langle D_x \rangle^\frac13  \theta^{\mathrm{NL}} \|_{L^2} \right)\\
				& +  \langle t\rangle^{-1}     \|e^{c \nu^{\frac{1}{3}} \lambda\left(D_x\right) t}\langle D_x, D_y+t D_x\rangle^5\langle\frac{1}{D_x}\rangle^4 \omega^{\mathrm{L}}\|_{L^2} \\
				& \quad \times \left(\|e^{c \nu^{\frac{1}{3}} \lambda\left(D_x\right) t}\langle D_x\rangle\langle\frac{1}{D_x}\rangle^\epsilon \langle D_x \rangle^\frac13 \theta^{\mathrm{NL}}\|_{L^2}  \| e^{c \nu^{\frac{1}{3}} \lambda\left(D_x\right) t} \langle D_x \rangle \langle\frac{1}{D_x}\rangle^\epsilon
				\left|D_x\right|^{\frac{1}{3}}  \langle D_x \rangle ^\frac13 \theta^{\mathrm{NL}} \|_{L^2} \rt.\\
				&\lt. \qquad + \|e^{c \nu^{\frac{1}{3}} \lambda\left(D_x\right) t}\langle D_x\rangle\langle\frac{1}{D_x}\rangle^\epsilon\left|D_x\right|^{\frac{1}{3}}  \langle D_x \rangle^\frac13  \theta^{\mathrm{NL}}\|_{L^2}^{\frac{3}{2}}   \|e^{c \nu^{\frac{1}{3}} \lambda\left(D_x\right) t}\langle D_x\rangle\langle\frac{1}{D_x}\rangle^\epsilon \nabla \langle D_x \rangle^\frac13  \theta^{\mathrm{NL}}\|_{L^2}^{\frac{1}{2}}\right)\\
				&+ \lt( \|e^{c \nu^{\frac{1}{3}} \lambda\left(D_x\right) t}\langle D_x\rangle  \langle\frac{1}{D_x}\rangle^\epsilon \langle D_x \rangle^\frac13  \theta^{\mathrm{NL}}\|_{L^2}\|e^{c \nu^{\frac{1}{3}} \lambda\left(D_x\right) t}\langle D_x\rangle \langle\frac{1}{D_x}\rangle^\epsilon \partial_x \nabla \phi^{\mathrm{NL}}\|_{L^2} \rt. \\
				&\lt. \quad + \|e^{c \nu^{\frac{1}{3}} \lambda\left(D_x\right) t}\langle D_x\rangle    \langle\frac{1}{D_x}\rangle^\epsilon    \left|D_x\right|^{\frac{1}{3}} \langle D_x \rangle^\frac13  \theta^{\mathrm{NL}}\|_{L^2}\|e^{c \nu^{\frac{1}{3}} \lambda\left(D_x\right) t}\langle D_x\rangle\langle\frac{1}{D_x}\rangle^\epsilon \omega^{\mathrm{NL}}\|_{L^2} \rt) \\
				& \qquad \times\|e^{c \nu^{\frac{1}{3}} \lambda\left(D_x\right) t}\langle D_x, D_y+t D_x\rangle^5\langle\frac{1}{D_x}\rangle^4\left|D_x\right|^{\frac{1}{3}} \langle D_x \rangle^\frac13  \theta^{\mathrm{L}}\|_{L^2} \\
				&  + \nu^{-\frac{1}{3}-\frac{\delta}{6}}\|e^{c \nu^{\frac{1}{3}} \lambda\left(D_x\right)}\langle D_x\rangle\langle\frac{1}{D_x}\rangle^\epsilon \partial_x \nabla \phi^{\mathrm{NL}}\|_{L^2}\|e^{c_0 \nu\left|D_x\right|^2 t^3}\langle D_x, D_y+t D_x\rangle^5\langle\frac{1}{D_x}\rangle^4 \theta^{\mathrm{L}}\|_{L^1 \cap L^2} \\
				&\quad  \times\left(\|e^{c \nu^{\frac{1}{3}} \lambda\left(D_x\right) t}\langle D_x\rangle\langle\frac{1}{D_x}\rangle^\epsilon \sqrt{\Upsilon} \langle D_x \rangle^\frac13  \theta^{\mathrm{NL}}\|_{L^2}+\nu^{\frac{1}{6}}\|e^{c \nu^{\frac{1}{3}} \lambda\left(D_x\right) t}\langle D_x\rangle\langle\frac{1}{D_x}\rangle^\epsilon |D_x|^\frac13 \langle D_x \rangle^{\frac{1}{3}} \theta^{\mathrm{NL}}\|_{L^2}\right) \\
				&+\|e^{c \nu^{\frac{1}{3}} \lambda\left(D_x\right) t}\langle D_x\rangle    \langle \frac{1}{D_x}\rangle^\epsilon  \left|D_x\right|^{\frac{1}{3}}\langle D_x \rangle^\frac13 \theta^{\mathrm{NL}}\|_{L^2}^\frac32 \|e^{c \nu^{\frac{1}{3}} \lambda\left(D_x\right) t}\langle D_x\rangle\langle\frac{1}{D_x}\rangle^\epsilon \omega^{\mathrm{NL}}\|_{L^2} \\
				& \quad \times\|e^{c \nu^{\frac{1}{3}} \lambda\left(D_x\right) t}\langle D_x\rangle\langle\frac{1}{D_x}\rangle^\epsilon \nabla  \langle D_x \rangle^{\frac{1}{3}} \theta^{\mathrm{NL}}\|_{L^2}^\frac12 \\
				&+ \|e^{c \nu^{\frac{1}{3}} \lambda\left(D_x\right) t}\langle D_x\rangle    \langle \frac{1}{D_x}\rangle^\epsilon  \langle D_x \rangle^{\frac{1}{3}} \theta^{\mathrm{NL}}\|_{L^2} \|e^{c \nu^{\frac{1}{3}} \lambda\left(D_x\right) t}\langle D_x\rangle\langle\frac{1}{D_x}\rangle^\epsilon \px \nabla \phi^{\mathrm{NL}}\|_{L^2} \\
				& \quad \times\|e^{c \nu^{\frac{1}{3}} \lambda\left(D_x\right) t}\langle D_x\rangle\langle\frac{1}{D_x}\rangle^\epsilon \nabla \langle D_x \rangle^{\frac{1}{3}} \theta^{\mathrm{NL}}\|_{L^2} \\
				&+\|e^{c \nu^{\frac{1}{3}} \lambda\left(D_x\right) t}\langle D_x\rangle    \langle \frac{1}{D_x}\rangle^\epsilon  \left|D_x\right|^{\frac{1}{3}} \langle D_x \rangle^{\frac{1}{3}} \theta^{\mathrm{NL}}\|_{L^2}^2 \|e^{c \nu^{\frac{1}{3}} \lambda\left(D_x\right) t}\langle D_x\rangle\langle\frac{1}{D_x}\rangle^\epsilon \omega^{\mathrm{NL}}\|_{L^2}  .
		\end{align*} }
		Integrating the inequality above over $[T_{0},T]$ and using 
		\begin{equation}\begin{aligned} \label{eq:t}
				\| \langle t \rangle^{-2} \|_{L_{[T_0, T]}^1} + \nu^{-\frac1{12}} \|\langle t \rangle^{-2} \|_{L_{[T_0, T]}^2}  + \nu^{\frac1{12}}\| \langle t \rangle^{-1} \|_{L_{[T_0, T]}^2} + \| \langle t \rangle^{-1} \|_{L_{[T_0, T]}^\infty} \leq \nu^\frac16,
		\end{aligned}\end{equation}
		we complete the proof.
	\end{proof}
	
	\subsection{Estimate of $\omega^{\mathrm{NL}}$ and proof of Proposition \ref{prop2.3}}
	The estimate for $\omega^{\mathrm{NL}}$ is similar, and we omit it.
	\begin{Prop} \label{lem:est of wnl}
		Let $\frac{1-\delta}{2}<\epsilon<\frac{1}{2}$, $0<\kappa<\frac{2\epsilon}{1-\delta}-1$ and $c<c(\delta,\epsilon)$ small enough, if the initial data $(w_{in}, \theta_{in})$ satisfies \eqref{eq:the initial data} then for any $  T \geq T_{0} = \nu^{-\frac{1}{6}} $, it holds that
		\footnotesize{	\begin{align*}
				& \| e^{c \nu^{\frac{1}{3}} \lambda\left(D_x\right) t}\langle D_x\rangle\langle\frac{1}{D_x}\rangle^\epsilon \omega^{\mathrm{NL}}\|_{L_{[T_0, T]}^\infty L^2}^2+\nu\|e^{c \nu^{\frac{1}{3}} \lambda\left(D_x\right) t}\langle D_x\rangle\langle\frac{1}{D_x}\rangle^\epsilon \nabla \omega^{\mathrm{NL}}\|_{L_{[T_0, T]}^2L^2}^2 \\
				& \quad+ \nu^{\frac{1}{3}}\|e^{c \nu^{\frac{1}{3}} \lambda\left(D_x\right) t}\langle D_x\rangle\langle\frac{1}{D_x}\rangle^\epsilon\left|D_x\right|^{\frac{1}{3}}  \omega^{\mathrm{NL}}\|_{L_{[T_0, T]}^2 L^2}^2  + \|e^{c \nu^{\frac{1}{3}} \lambda\left(D_x\right) t}\langle D_x\rangle\langle\frac{1}{D_x}\rangle^\epsilon \partial_x \nabla \phi^{\mathrm{NL}}\|_{L_{[T_0, T]}^2 L^2}^2
				\\
				& \quad+\|e^{c \nu^{\frac{1}{3}} \lambda\left(D_x\right) t}\langle D_x\rangle\langle\frac{1}{D_x}\rangle^\epsilon \sqrt{\Upsilon\left(t, D_x, D_y\right)} \omega^{\mathrm{NL}}\|_{L_{[T_0, T]}^2 L^2}^2 \\
				&\lesssim  \| e^{c \nu^{\frac{1}{3}} \lambda\left(D_x\right) t}\langle D_x\rangle\langle\frac{1}{D_x}\rangle^\epsilon \omega^{\mathrm{NL}} (T_0)\|_{ L^2}^2\\
				&\quad+\nu^{\frac{1}{6}}  \|e^{c \nu^{\frac{1}{3}} \lambda\left(D_x\right) t}\langle D_x\rangle  \langle \frac{1}{D_x}\rangle^\epsilon \omega^{\mathrm{NL}}\|_{L_{[T_0, T]}^\infty L^2}    \|e^{c \nu^{\frac{1}{3}} \lambda\left(D_x\right) t}\langle D_x, D_y + tD_x\rangle^4\langle\frac{1}{D_x}\rangle^4 \omega^{\mathrm{L}} \|_{L_{[T_0, T]}^\infty L^2}^2 \\
				&\quad + \nu^{\frac{1}{4}} \|e^{c \nu^{\frac{1}{3}} \lambda\left(D_x\right) t}\langle D_x\rangle  \langle \frac{1}{D_x}\rangle^\epsilon \omega^{\mathrm{NL}}\|_{L_{[T_0, T]}^\infty L^2}    \|e^{c \nu^{\frac{1}{3}} \lambda\left(D_x\right) t}\langle D_x, D_y + tD_x\rangle^ 5 \langle\frac{1}{D_x}\rangle^3 \omega^{\mathrm{L}} \|_{L^\infty L^2} \\
				&\qquad \times \|e^{c \nu^{\frac{1}{3}} \lambda\left(D_x\right) t} \nabla \omega^{\mathrm{NL}} \|_{L_{[T_0, T]}^2 L^2} \\
				&\quad +      \|e^{c \nu^{\frac{1}{3}} \lambda\left(D_x\right) t}\langle D_x, D_y+t D_x\rangle^4\langle\frac{1}{D_x}\rangle^2 \omega^{\mathrm{L}}\|_{L_{[T_0, T]}^\infty L^2} \\
				& \qquad \times \left(  \nu^{\frac{1}{12}}\|e^{c \nu^{\frac{1}{3}} \lambda\left(D_x\right) t}\langle D_x\rangle\langle\frac{1}{D_x}\rangle^\epsilon \omega^{\mathrm{NL}}\|_{L_{[T_0, T]}^\infty L^2}  \| e^{c \nu^{\frac{1}{3}} \lambda\left(D_x\right) t} \langle D_x \rangle
				\left|D_x\right|^{\frac{1}{3}}  \omega^{\mathrm{NL}} \|_{L_{[T_0, T]}^2 L^2} \rt.\\
				&\qquad \quad  \lt.+ \nu^{\frac{1}{6}} \|e^{c \nu^{\frac{1}{3}} \lambda\left(D_x\right) t}\langle D_x\rangle\langle\frac{1}{D_x}\rangle^\epsilon\left|D_x\right|^{\frac{1}{3}}  \omega^{\mathrm{NL}}\|_{L_{[T_0, T]}^2 L^2}^{\frac{3}{2}}   \|e^{c \nu^{\frac{1}{3}} \lambda\left(D_x\right) t}\langle D_x\rangle\langle\frac{1}{D_x}\rangle^\epsilon \nabla   \omega^{\mathrm{NL}}\|_{L_{[T_0, T]}^2 L^2}^{\frac{1}{2}}\right) \\
				&\quad+  \lt( \|e^{c \nu^{\frac{1}{3}} \lambda\left(D_x\right) t}\langle D_x\rangle \partial_x \nabla \phi^{\mathrm{NL}}\|_{L_{[T_0, T]}^2 L^2} + \|e^{c \nu^{\frac{1}{3}} \lambda\left(D_x\right) t}\langle D_x\rangle\left|D_x\right|^{\frac{1}{3}} \omega^{\mathrm{NL}}\|_{L_{[T_0, T]}^2 L^2} \rt) \\
				&  \qquad \times\|e^{c \nu^{\frac{1}{3}} \lambda\left(D_x\right) t}\langle D_x, D_y+t D_x\rangle^4\langle\frac{1}{D_x}\rangle^2\left|D_x\right|^{\frac{1}{3}} \omega^{\mathrm{L}}\|_{L_{[T_0, T]}^2 L^2} \|e^{c \nu^{\frac{1}{3}} \lambda\left(D_x\right) t}\langle D_x\rangle  \langle\frac{1}{D_x}\rangle^\epsilon \omega^{\mathrm{NL}}\|_{L_{[T_0, T]}^\infty L^2} \\
				& \quad  + \nu^{-\frac{1}{3}-\frac{\delta}{6}}\|e^{c \nu^{\frac{1}{3}} \lambda\left(D_x\right)}\langle D_x\rangle\langle\frac{1}{D_x}\rangle^\epsilon \partial_x \nabla \phi^{\mathrm{NL}}\|_{L_{[T_0, T]}^2 L^2}\|e^{c_0 \nu\left|D_x\right|^2 t^3}\langle D_x, D_y+t D_x\rangle^4\langle\frac{1}{D_x}\rangle^2 \omega^{\mathrm{L}}\|_{L_{[T_0, T]}^\infty (L^1 \cap L^2)} \\
				&\qquad  \times\left(\|e^{c \nu^{\frac{1}{3}} \lambda\left(D_x\right) t}\langle D_x\rangle\langle\frac{1}{D_x}\rangle^\epsilon \sqrt{\Upsilon} \omega^{\mathrm{NL}}\|_{L_{[T_0, T]}^2L^2}+\nu^{\frac{1}{6}}\|e^{c \nu^{\frac{1}{3}} \lambda\left(D_x\right) t}\langle D_x\rangle\langle\frac{1}{D_x}\rangle^\epsilon |D_x|^\frac13 \omega^{\mathrm{NL}}\|_{L_{[T_0, T]}^2L^2}\right) \\
				&\quad + \|e^{c \nu^{\frac{1}{3}} \lambda\left(D_x\right) t}\langle D_x\rangle    \langle \frac{1}{D_x}\rangle^\epsilon  \left|D_x\right|^{\frac{1}{3}}\omega^{\mathrm{NL}}\|_{L_{[T_0, T]}^2L^2}^\frac32 \|e^{c \nu^{\frac{1}{3}} \lambda\left(D_x\right) t}\langle D_x\rangle\langle\frac{1}{D_x}\rangle^\epsilon \omega^{\mathrm{NL}}\|_{L_{[T_0, T]}^\infty L^2} \\
				& \qquad \times\|e^{c \nu^{\frac{1}{3}} \lambda\left(D_x\right) t}\langle D_x\rangle\langle\frac{1}{D_x}\rangle^\epsilon \nabla  \omega^{\mathrm{NL}}\|_{L_{[T_0, T]}^2 L^2}^\frac12 \\
				& \quad + \|e^{c \nu^{\frac{1}{3}} \lambda\left(D_x\right) t}\langle D_x\rangle    \langle \frac{1}{D_x}\rangle^\epsilon  \omega^{\mathrm{NL}}\|_{L_{[T_0, T]}^\infty L^2} \|e^{c \nu^{\frac{1}{3}} \lambda\left(D_x\right) t}\langle D_x\rangle\langle\frac{1}{D_x}\rangle^\epsilon \px \nabla \phi^{\mathrm{NL}}\|_{L_{[T_0, T]}^2L^2} \\
				& \qquad \times\|e^{c \nu^{\frac{1}{3}} \lambda\left(D_x\right) t}\langle D_x\rangle\langle\frac{1}{D_x}\rangle^\epsilon \nabla \omega^{\mathrm{NL}}\|_{L_{[T_0, T]}^2L^2} \\
				&\quad + \nu^{-\frac13 }\| e^{c \nu^{\frac{1}{3}} \lambda\left(D_x\right) t}\langle D_x\rangle\langle\frac{1}{D_x}\rangle^\epsilon |D_x|^\frac23 \theta^{\mathrm{NL}}  \|_{L_{[T_0, T]}^2L^2}^2.
		\end{align*} }
	\end{Prop}
	\begin{proof}
		Taking the inner product of $\eqref{eq:nonlinear part}_1$ with $\mathcal{M} e^{2 c \nu^{\frac{1}{3}} \lambda\left(D_x\right) t}\langle D_x\rangle^2\langle\frac{1}{D_x}\rangle^{2 \epsilon} \omega^{\mathrm{NL}}$, we have that
		$$
		\begin{aligned}
			\frac{d}{d t} \| \sqrt{\mathcal{M}} & e^{c \nu^{\frac{1}{3}} \lambda\left(D_x\right) t}\langle D_x\rangle \langle\frac{1}{D_x}\rangle^\epsilon \omega^{\mathrm{NL}}\|_{L^2}^2-2 c \nu^{\frac{1}{3}}\| \sqrt{\mathcal{M}} e^{c \nu^{\frac{1}{3}} \lambda\left(D_x\right) t}\langle D_x\rangle\langle\frac{1}{D_x}\rangle^\epsilon \sqrt{\lambda\left(D_x\right)} \omega^{\mathrm{NL}} \|_{L^2}^2 \\
			& +2 \nu\|\sqrt{\mathcal{M}} e^{c \nu^{\frac{1}{3}} \lambda\left(D_x\right) t}\langle D_x\rangle\langle\frac{1}{D_x}\rangle^\epsilon \nabla\omega^{\mathrm{NL}}\|_{L^2}^2 \\
			& +\int_{\mathbb{R}^2}\left(-\partial_t+k \partial_{\xi}\right) \mathcal{M}(k, \xi) e^{2 c \nu^{\frac{1}{3}} \lambda(k) t}\langle k\rangle^2\langle\frac{1}{k}\rangle^{2 \epsilon}  \left|\hat{\omega}^{\mathrm{NL}}(k, \xi)\right|^2 d k d \xi \\
			= & -2 \Re\left(\left(u^{\mathrm{L}}+u^{\mathrm{NL}}\right) \cdot \nabla\left(\omega^{\mathrm{L}}+\omega^{\mathrm{NL}}\right) - \px \theta^{\mathrm{NL}} \left\lvert\, \mathcal{M} e^{2 c \nu^{\frac{1}{3}} \lambda\left(D_x\right) t}\langle D_x\rangle^2\langle\frac{1}{D_x}\rangle^{2 \epsilon}  \omega^{\mathrm{NL}}\right.\right).
		\end{aligned}
		$$
		which gives 
		\begin{align*}
			& \frac{d}{d t}\|\sqrt{\mathcal{M}} e^{c \nu^{\frac{1}{3}} \lambda\left(D_x\right) t}\langle D_x\rangle\langle\frac{1}{D_x}\rangle^\epsilon \omega^{\mathrm{NL}}\|_{L^2}^2+\nu\|e^{c \nu^{\frac{1}{3}} \lambda\left(D_x\right) t}\langle D_x\rangle\langle\frac{1}{D_x}\rangle^\epsilon \nabla \omega^{\mathrm{NL}}\|_{L^2}^2 \\
			& \quad+\frac{\nu^{\frac{1}{3}}}{16}\|e^{c \nu^{\frac{1}{3}} \lambda\left(D_x\right) t}\langle D_x\rangle\langle\frac{1}{D_x}\rangle^\epsilon\left|D_x\right|^{\frac{1}{3}}  \omega^{\mathrm{NL}}\|_{L^2}^2  + \|e^{c \nu^{\frac{1}{3}} \lambda\left(D_x\right) t}\langle D_x\rangle\langle\frac{1}{D_x}\rangle^\epsilon \partial_x \nabla \phi^{\mathrm{NL}}\|_{L^2}^2
			\\
			& \quad+\|e^{c \nu^{\frac{1}{3}} \lambda\left(D_x\right) t}\langle D_x\rangle\langle\frac{1}{D_x}\rangle^\epsilon \sqrt{\Upsilon\left(t, D_x, D_y\right)} \omega^{\mathrm{NL}}\|_{L^2}^2 \\
			& \leq 2\left|\Re\left((u^{\mathrm{L}}+u^{\mathrm{NL}}) \cdot \nabla(\omega^{\mathrm{L}}+\omega^{\mathrm{NL}})  \big| \mathcal{M} e^{2 c \nu^{\frac{1}{3}} \lambda\left(D_x\right) t}\langle D_x\rangle^2\langle\frac{1}{D_x}\rangle^{2 \epsilon}  \omega^{\mathrm{NL}} \right)\right| \\
			&\quad  + 2 \| \sqrt{\mathcal{M}} e^{c \nu^{\frac{1}{3}} \lambda\left(D_x\right) t}\langle D_x\rangle\langle\frac{1}{D_x}\rangle^\epsilon |D_x|^\frac23 \theta^{\mathrm{NL}}  \|_{L^2}  \| \sqrt{\mathcal{M}} e^{c \nu^{\frac{1}{3}} \lambda\left(D_x\right) t}\langle D_x\rangle\langle\frac{1}{D_x}\rangle^\epsilon |D_x|^\frac13 \omega^{\mathrm{NL}} \|_{L^2} 
		\end{align*}
		
		\underline{\bf Estimate of  $u^{\mathrm{L}} \cdot  \nabla \omega^{\mathrm{L}}$.} It is similar the estimate of  $u^{\mathrm{L}} \cdot  \nabla \langle D_x \rangle^\frac13 \theta^{\mathrm{L}}$ and we get
		\begin{equation*}\begin{aligned} \label{eq:sum11}
				&\left|\left(u^{\mathrm{L}} \cdot \nabla \omega^{\mathrm{L}} \big| \mathcal{M} e^{2 c \nu^{\frac{1}{3}} \lambda\left(D_x\right) t}\langle D_x\rangle^2\langle\frac{1}{D_x}\rangle^{2 \epsilon} \omega^{\mathrm{NL}} \right)\right| \\
				\lesssim	& \langle t\rangle^{-2} \|e^{c \nu^{\frac{1}{3}} \lambda\left(D_x\right) t}\langle D_x\rangle  \langle \frac{1}{D_x}\rangle^\epsilon \omega^{\mathrm{NL}}\|_{L^2}    \|e^{c \nu^{\frac{1}{3}} \lambda\left(D_x\right) t}\langle D_x, D_y + tD_x\rangle^4\langle\frac{1}{D_x}\rangle^4 \omega^{\mathrm{L}} \|_{L^2}^2.
		\end{aligned}\end{equation*}

		\underline{\bf Estimate of $  u^{\mathrm{L}} \cdot  \nabla \omega^{\mathrm{NL}}$.} Similarly, we have that
		\begin{align*} 
			&\left|\left(  u^{\mathrm{L}} \cdot\nabla \omega^{\mathrm{NL}} \big| \mathcal{M} e^{2 c \nu^{\frac{1}{3}} \lambda\left(D_x\right) t}\langle D_x\rangle^2\langle\frac{1}{D_x}\rangle^{2 \epsilon} \omega^{\mathrm{NL}} \right)\right| \\
			&\lesssim \langle t\rangle^{-2} \|e^{c \nu^{\frac{1}{3}} \lambda\left(D_x\right) t}\langle D_x\rangle  \langle \frac{1}{D_x}\rangle^\epsilon \omega^{\mathrm{NL}}\|_{L^2}    \|e^{c \nu^{\frac{1}{3}} \lambda\left(D_x\right) t}\langle D_x, D_y + tD_x\rangle^ 5 \langle\frac{1}{D_x}\rangle^3 \omega^{\mathrm{L}} \|_{L^2} \\
			&\qquad \times \|e^{c \nu^{\frac{1}{3}} \lambda\left(D_x\right) t} \py \omega^{\mathrm{NL}} \|_{L^2} \\
			&\quad +  \langle t\rangle^{-1}     \|e^{c \nu^{\frac{1}{3}} \lambda\left(D_x\right) t}\langle D_x, D_y+t D_x\rangle^4\langle\frac{1}{D_x}\rangle^2 \omega^{\mathrm{L}}\|_{L^2} \\
			& \qquad \times \left(\|e^{c \nu^{\frac{1}{3}} \lambda\left(D_x\right) t}\langle D_x\rangle\langle\frac{1}{D_x}\rangle^\epsilon \omega^{\mathrm{NL}}\|_{L^2}  \| e^{c \nu^{\frac{1}{3}} \lambda\left(D_x\right) t} \langle D_x \rangle
			\left|D_x\right|^{\frac{1}{3}}  \omega^{\mathrm{NL}} \|_{L^2} \rt.\\
			&\qquad \quad  \lt.+ \|e^{c \nu^{\frac{1}{3}} \lambda\left(D_x\right) t}\langle D_x\rangle\langle\frac{1}{D_x}\rangle^\epsilon\left|D_x\right|^{\frac{1}{3}}  \omega^{\mathrm{NL}}\|_{L^2}^{\frac{3}{2}}   \|e^{c \nu^{\frac{1}{3}} \lambda\left(D_x\right) t}\langle D_x\rangle\langle\frac{1}{D_x}\rangle^\epsilon \px \omega^{\mathrm{NL}}\|_{L^2}^{\frac{1}{2}}\right).
		\end{align*}
		
		\underline{\bf Estimate of $  u^{\mathrm{NL}} \cdot  \nabla \omega^{\mathrm{L}}$.} Analogous to the estimates of $\theta$ it concludes that
		\begin{equation*}\begin{aligned}
				&\left|\left(u^{\mathrm{NL}} \cdot \nabla \omega^{\mathrm{L}} \big| \mathcal{M} e^{2 c \nu^{\frac{1}{3}} \lambda\left(D_x\right) t}\langle D_x\rangle^2\langle\frac{1}{D_x}\rangle^{2 \epsilon} \omega^{\mathrm{NL}} \right)\right| \\
				\lesssim &  \lt( \|e^{c \nu^{\frac{1}{3}} \lambda\left(D_x\right) t}\langle D_x\rangle \partial_x \nabla \phi^{\mathrm{NL}}\|_{L^2} + \|e^{c \nu^{\frac{1}{3}} \lambda\left(D_x\right) t}\langle D_x\rangle\left|D_x\right|^{\frac{1}{3}} \omega^{\mathrm{NL}}\|_{L^2} \rt) \\
				&\quad  \times\|e^{c \nu^{\frac{1}{3}} \lambda\left(D_x\right) t}\langle D_x, D_y+t D_x\rangle^4\langle\frac{1}{D_x}\rangle^2\left|D_x\right|^{\frac{1}{3}} \omega^{\mathrm{L}}\|_{L^2} \|e^{c \nu^{\frac{1}{3}} \lambda\left(D_x\right) t}\langle D_x\rangle  \langle\frac{1}{D_x}\rangle^\epsilon \omega^{\mathrm{NL}}\|_{L^2} \\
				&  + \nu^{-\frac{1}{3}-\frac{\delta}{6}}\|e^{c \nu^{\frac{1}{3}} \lambda\left(D_x\right)}\langle D_x\rangle\langle\frac{1}{D_x}\rangle^\epsilon \partial_x \nabla \phi^{\mathrm{NL}}\|_{L^2}\|e^{c_0 \nu\left|D_x\right|^2 t^3}\langle D_x, D_y+t D_x\rangle^4\langle\frac{1}{D_x}\rangle^2 \omega^{\mathrm{L}}\|_{L^1 \cap L^2} \\
				&\quad  \times\left(\|e^{c \nu^{\frac{1}{3}} \lambda\left(D_x\right) t}\langle D_x\rangle\langle\frac{1}{D_x}\rangle^\epsilon \sqrt{\Upsilon} \omega^{\mathrm{NL}}\|_{L^2}+\nu^{\frac{1}{6}}\|e^{c \nu^{\frac{1}{3}} \lambda\left(D_x\right) t}\langle D_x\rangle\langle\frac{1}{D_x}\rangle^\epsilon |D_x|^\frac13 \omega^{\mathrm{NL}}\|_{L^2}\right).
		\end{aligned}\end{equation*}

		\underline{\bf Estimate of $   u^{\mathrm{NL}} \cdot  \nabla \omega^{\mathrm{NL}}$.} Using the same trick, we obtain that
		\begin{equation*}\begin{aligned}
				&\left|\left(  u^{\mathrm{NL}} \cdot \nabla \omega^{\mathrm{NL}} \big| \mathcal{M} e^{2 c \nu^{\frac{1}{3}} \lambda\left(D_x\right) t}\langle D_x\rangle^2\langle\frac{1}{D_x}\rangle^{2 \epsilon} \omega^{\mathrm{NL}} \right)\right| \\
				\lesssim& \|e^{c \nu^{\frac{1}{3}} \lambda\left(D_x\right) t}\langle D_x\rangle    \langle \frac{1}{D_x}\rangle^\epsilon  \left|D_x\right|^{\frac{1}{3}}\omega^{\mathrm{NL}}\|_{L^2}^\frac32 \|e^{c \nu^{\frac{1}{3}} \lambda\left(D_x\right) t}\langle D_x\rangle\langle\frac{1}{D_x}\rangle^\epsilon \omega^{\mathrm{NL}}\|_{L^2} \\
				& \quad \times\|e^{c \nu^{\frac{1}{3}} \lambda\left(D_x\right) t}\langle D_x\rangle\langle\frac{1}{D_x}\rangle^\epsilon \nabla  \omega^{\mathrm{NL}}\|_{L^2}^\frac12 \\
				&+ \|e^{c \nu^{\frac{1}{3}} \lambda\left(D_x\right) t}\langle D_x\rangle    \langle \frac{1}{D_x}\rangle^\epsilon  \omega^{\mathrm{NL}}\|_{L^2} \|e^{c \nu^{\frac{1}{3}} \lambda\left(D_x\right) t}\langle D_x\rangle\langle\frac{1}{D_x}\rangle^\epsilon \px \nabla \phi^{\mathrm{NL}}\|_{L^2} \\
				& \quad \times\|e^{c \nu^{\frac{1}{3}} \lambda\left(D_x\right) t}\langle D_x\rangle\langle\frac{1}{D_x}\rangle^\epsilon \nabla \omega^{\mathrm{NL}}\|_{L^2}.
		\end{aligned}\end{equation*}
		Summing up, we get that
		\footnotesize{	\begin{align*}
				& \frac{d}{d t}\| \sqrt{\mathcal{M}}  e^{c \nu^{\frac{1}{3}} \lambda\left(D_x\right) t}\langle D_x\rangle\langle\frac{1}{D_x}\rangle^\epsilon \omega^{\mathrm{NL}}\|_{L^2}^2+\nu\|e^{c \nu^{\frac{1}{3}} \lambda\left(D_x\right) t}\langle D_x\rangle\langle\frac{1}{D_x}\rangle^\epsilon \nabla \omega^{\mathrm{NL}}\|_{L^2}^2 \\
				& \quad+ \nu^{\frac{1}{3}}\|e^{c \nu^{\frac{1}{3}} \lambda\left(D_x\right) t}\langle D_x\rangle\langle\frac{1}{D_x}\rangle^\epsilon\left|D_x\right|^{\frac{1}{3}}  \omega^{\mathrm{NL}}\|_{L^2}^2  + \|e^{c \nu^{\frac{1}{3}} \lambda\left(D_x\right) t}\langle D_x\rangle\langle\frac{1}{D_x}\rangle^\epsilon \partial_x \nabla \phi^{\mathrm{NL}}\|_{L^2}^2
				\\
				& \quad+\|e^{c \nu^{\frac{1}{3}} \lambda\left(D_x\right) t}\langle D_x\rangle\langle\frac{1}{D_x}\rangle^\epsilon \sqrt{\Upsilon\left(t, D_x, D_y\right)} \omega^{\mathrm{NL}}\|_{L^2}^2 \\
				&\lesssim \langle t\rangle^{-2} \|e^{c \nu^{\frac{1}{3}} \lambda\left(D_x\right) t}\langle D_x\rangle  \langle \frac{1}{D_x}\rangle^\epsilon \omega^{\mathrm{NL}}\|_{L^2}    \|e^{c \nu^{\frac{1}{3}} \lambda\left(D_x\right) t}\langle D_x, D_y + tD_x\rangle^4\langle\frac{1}{D_x}\rangle^4 \omega^{\mathrm{L}} \|_{L^2}^2 \\
				&\quad +  \langle t\rangle^{-2} \|e^{c \nu^{\frac{1}{3}} \lambda\left(D_x\right) t}\langle D_x\rangle  \langle \frac{1}{D_x}\rangle^\epsilon \omega^{\mathrm{NL}}\|_{L^2}    \|e^{c \nu^{\frac{1}{3}} \lambda\left(D_x\right) t}\langle D_x, D_y + tD_x\rangle^ 5 \langle\frac{1}{D_x}\rangle^3 \omega^{\mathrm{L}} \|_{L^2} \\
				&\qquad \times \|e^{c \nu^{\frac{1}{3}} \lambda\left(D_x\right) t} \py \omega^{\mathrm{NL}} \|_{L^2} \\
				&\quad +  \langle t\rangle^{-1}     \|e^{c \nu^{\frac{1}{3}} \lambda\left(D_x\right) t}\langle D_x, D_y+t D_x\rangle^4\langle\frac{1}{D_x}\rangle^2 \omega^{\mathrm{L}}\|_{L^2} \\
				& \qquad \times \left(\|e^{c \nu^{\frac{1}{3}} \lambda\left(D_x\right) t}\langle D_x\rangle\langle\frac{1}{D_x}\rangle^\epsilon \omega^{\mathrm{NL}}\|_{L^2}  \| e^{c \nu^{\frac{1}{3}} \lambda\left(D_x\right) t} \langle D_x \rangle
				\left|D_x\right|^{\frac{1}{3}}  \omega^{\mathrm{NL}} \|_{L^2} \rt.\\
				&\qquad \quad  \lt.+ \|e^{c \nu^{\frac{1}{3}} \lambda\left(D_x\right) t}\langle D_x\rangle\langle\frac{1}{D_x}\rangle^\epsilon\left|D_x\right|^{\frac{1}{3}}  \omega^{\mathrm{NL}}\|_{L^2}^{\frac{3}{2}}   \|e^{c \nu^{\frac{1}{3}} \lambda\left(D_x\right) t}\langle D_x\rangle\langle\frac{1}{D_x}\rangle^\epsilon \px \omega^{\mathrm{NL}}\|_{L^2}^{\frac{1}{2}}\right) \\
				&\quad+  \lt( \|e^{c \nu^{\frac{1}{3}} \lambda\left(D_x\right) t}\langle D_x\rangle \partial_x \nabla \phi^{\mathrm{NL}}\|_{L^2} + \|e^{c \nu^{\frac{1}{3}} \lambda\left(D_x\right) t}\langle D_x\rangle\left|D_x\right|^{\frac{1}{3}} \omega^{\mathrm{NL}}\|_{L^2} \rt) \\
				&  \qquad \times\|e^{c \nu^{\frac{1}{3}} \lambda\left(D_x\right) t}\langle D_x, D_y+t D_x\rangle^4\langle\frac{1}{D_x}\rangle^2\left|D_x\right|^{\frac{1}{3}} \omega^{\mathrm{L}}\|_{L^2} \|e^{c \nu^{\frac{1}{3}} \lambda\left(D_x\right) t}\langle D_x\rangle  \langle\frac{1}{D_x}\rangle^\epsilon \omega^{\mathrm{NL}}\|_{L^2} \\
				& \quad  + \nu^{-\frac{1}{3}-\frac{\delta}{6}}\|e^{c \nu^{\frac{1}{3}} \lambda\left(D_x\right)}\langle D_x\rangle\langle\frac{1}{D_x}\rangle^\epsilon \partial_x \nabla \phi^{\mathrm{NL}}\|_{L^2}\|e^{c_0 \nu\left|D_x\right|^2 t^3}\langle D_x, D_y+t D_x\rangle^4\langle\frac{1}{D_x}\rangle^2 \omega^{\mathrm{L}}\|_{L^1 \cap L^2} \\
				&\qquad  \times\left(\|e^{c \nu^{\frac{1}{3}} \lambda\left(D_x\right) t}\langle D_x\rangle\langle\frac{1}{D_x}\rangle^\epsilon \sqrt{\Upsilon} \omega^{\mathrm{NL}}\|_{L^2}+\nu^{\frac{1}{6}}\|e^{c \nu^{\frac{1}{3}} \lambda\left(D_x\right) t}\langle D_x\rangle\langle\frac{1}{D_x}\rangle^\epsilon |D_x|^\frac13 \omega^{\mathrm{NL}}\|_{L^2}\right) \\
				&\quad + \|e^{c \nu^{\frac{1}{3}} \lambda\left(D_x\right) t}\langle D_x\rangle    \langle \frac{1}{D_x}\rangle^\epsilon  \left|D_x\right|^{\frac{1}{3}}\omega^{\mathrm{NL}}\|_{L^2}^\frac32 \|e^{c \nu^{\frac{1}{3}} \lambda\left(D_x\right) t}\langle D_x\rangle\langle\frac{1}{D_x}\rangle^\epsilon \omega^{\mathrm{NL}}\|_{L^2} \\
				& \qquad \times\|e^{c \nu^{\frac{1}{3}} \lambda\left(D_x\right) t}\langle D_x\rangle\langle\frac{1}{D_x}\rangle^\epsilon \nabla  \omega^{\mathrm{NL}}\|_{L^2}^\frac12 \\
				& \quad + \|e^{c \nu^{\frac{1}{3}} \lambda\left(D_x\right) t}\langle D_x\rangle    \langle \frac{1}{D_x}\rangle^\epsilon  \omega^{\mathrm{NL}}\|_{L^2} \|e^{c \nu^{\frac{1}{3}} \lambda\left(D_x\right) t}\langle D_x\rangle\langle\frac{1}{D_x}\rangle^\epsilon \px \nabla \phi^{\mathrm{NL}}\|_{L^2} \\
				& \qquad \times\|e^{c \nu^{\frac{1}{3}} \lambda\left(D_x\right) t}\langle D_x\rangle\langle\frac{1}{D_x}\rangle^\epsilon \nabla \omega^{\mathrm{NL}}\|_{L^2} \\
				&\quad + \nu^{-\frac13 }\| e^{c \nu^{\frac{1}{3}} \lambda\left(D_x\right) t}\langle D_x\rangle\langle\frac{1}{D_x}\rangle^\epsilon |D_x|^\frac23 \theta^{\mathrm{NL}}  \|_{L^2}^2,
		\end{align*} }
		which, combined with \eqref{eq:t}, completes the proof.
	\end{proof}
	
	At last, we present the proof of Proposition \ref{prop2.3}.
	\begin{proof}[Proof of Proposition \ref{prop2.3}]
		Let us designate $T$ as the terminal point of the largest range $[T_0, T]$ such that the following hypothesis holds:
		\begin{align*}
			& \| e^{c \nu^{\frac{1}{3}} \lambda\left(D_x\right) t}\langle D_x\rangle\langle\frac{1}{D_x}\rangle^\epsilon \omega^{\mathrm{NL}}\|_{L_{[T_0, T]}^\infty L^2}^2+\nu\|e^{c \nu^{\frac{1}{3}} \lambda\left(D_x\right) t}\langle D_x\rangle\langle\frac{1}{D_x}\rangle^\epsilon \nabla \omega^{\mathrm{NL}}\|_{L_{[T_0, T]}^2L^2}^2 \\
			& \quad+ \nu^{\frac{1}{3}}\|e^{c \nu^{\frac{1}{3}} \lambda\left(D_x\right) t}\langle D_x\rangle\langle\frac{1}{D_x}\rangle^\epsilon\left|D_x\right|^{\frac{1}{3}}  \omega^{\mathrm{NL}}\|_{L_{[T_0, T]}^2 L^2}^2  + \|e^{c \nu^{\frac{1}{3}} \lambda\left(D_x\right) t}\langle D_x\rangle\langle\frac{1}{D_x}\rangle^\epsilon \partial_x \nabla \phi^{\mathrm{NL}}\|_{L_{[T_0, T]}^2 L^2}^2
			\\
			& \quad+\|e^{c \nu^{\frac{1}{3}} \lambda\left(D_x\right) t}\langle D_x\rangle\langle\frac{1}{D_x}\rangle^\epsilon \sqrt{\Upsilon\left(t, D_x, D_y\right)} \omega^{\mathrm{NL}}\|_{L_{[T_0, T]}^2 L^2}^2 \leq 2C_1\nu^{1+4\delta},\\
			& \| e^{c \nu^{\frac{1}{3}} \lambda\left(D_x\right) t}\langle D_x\rangle\langle\frac{1}{D_x}\rangle^\epsilon \langle D_x \rangle^\frac13  \theta^{\mathrm{NL}}\|_{L_{[T_0, T]}^\infty L^2}^2+\nu\|e^{c \nu^{\frac{1}{3}} \lambda\left(D_x\right) t}\langle D_x\rangle\langle\frac{1}{D_x}\rangle^\epsilon \nabla \langle D_x \rangle^\frac13  \theta^{\mathrm{NL}}\|_{L_{[T_0, T]}^2 L^2}^2 \\
			& \quad+ \nu^{\frac{1}{3}}\|e^{c \nu^{\frac{1}{3}} \lambda\left(D_x\right) t}\langle D_x\rangle\langle\frac{1}{D_x}\rangle^\epsilon\left|D_x\right|^{\frac{1}{3}}  \langle D_x \rangle ^\frac13 \theta^{\mathrm{NL}}\|_{L_{[T_0, T]}^2 L^2}^2  
			\\
			& \quad+\|e^{c \nu^{\frac{1}{3}} \lambda\left(D_x\right) t}\langle D_x\rangle\langle\frac{1}{D_x}\rangle^\epsilon \sqrt{\Upsilon\left(t, D_x, D_y\right)} \langle D_x \rangle^\frac13  \theta^{\mathrm{NL}}\|_{L_{[T_0, T]}^2 L^2}^2 \leq 2C_1'\nu^{\frac53 + 6\delta}.  		
		\end{align*}
		By Proposition \ref{lem:est of wnl}, Proposition \ref{lem:linear theta} and Corollary \ref{cor2}, we get that
		\begin{align*}
			& \| e^{c \nu^{\frac{1}{3}} \lambda\left(D_x\right) t}\langle D_x\rangle\langle\frac{1}{D_x}\rangle^\epsilon \omega^{\mathrm{NL}}\|_{L_{[T_0, T]}^\infty L^2}^2+\nu\|e^{c \nu^{\frac{1}{3}} \lambda\left(D_x\right) t}\langle D_x\rangle\langle\frac{1}{D_x}\rangle^\epsilon \nabla \omega^{\mathrm{NL}}\|_{L_{[T_0, T]}^2L^2}^2 \\
			& \quad+ \nu^{\frac{1}{3}}\|e^{c \nu^{\frac{1}{3}} \lambda\left(D_x\right) t}\langle D_x\rangle\langle\frac{1}{D_x}\rangle^\epsilon\left|D_x\right|^{\frac{1}{3}}  \omega^{\mathrm{NL}}\|_{L_{[T_0, T]}^2 L^2}^2  + \|e^{c \nu^{\frac{1}{3}} \lambda\left(D_x\right) t}\langle D_x\rangle\langle\frac{1}{D_x}\rangle^\epsilon \partial_x \nabla \phi^{\mathrm{NL}}\|_{L_{[T_0, T]}^2 L^2}^2
			\\
			& \quad+\|e^{c \nu^{\frac{1}{3}} \lambda\left(D_x\right) t}\langle D_x\rangle\langle\frac{1}{D_x}\rangle^\epsilon \sqrt{\Upsilon\left(t, D_x, D_y\right)} \omega^{\mathrm{NL}}\|_{L_{[T_0, T]}^2 L^2}^2 \\
			&\leq C_1\nu^{\frac43 + 4\delta }+ C_2\lt[ \nu^\frac16 \nu^{\frac12 + 2\delta} \nu^{\frac23 + 2\delta}  +  \nu^{\frac14}\nu^{\frac12 + 2\delta}\nu^{\frac13 + \delta}\nu^{  2\delta}   + (\nu^{\frac12 + 2\delta} + \nu^{\frac13 + 2\delta})\nu^{\frac16 + \delta}\nu^{\frac12 + 2\delta} \rt.\\
			&\lt.\qquad + \nu^{\frac13 + \delta}(\nu^{\frac1{12}}\nu^{\frac12 + 2\delta}\nu^{\frac13 + 2\delta}  + \nu^{\frac16 }\nu^{\frac32 (\frac13 + 2\delta)}\nu^{\frac12 (2\delta)}) + \nu^{-\frac13 - \frac\delta6}\nu^{\frac12 + 2\delta}\nu^{\frac13 +\delta}(\nu^{\frac12 + 2\delta}+ \nu^{\frac16}\nu^{\frac13 + 2\delta}) \rt.\\
			&\lt. \qquad   +  \nu^{\frac32(\frac13 + 2\delta)}\nu^{\frac12 + 2\delta}\nu^{\frac12 ( 2\delta)}   +  \nu^{\frac12 + 2\delta}\nu^{\frac12 + 2\delta}\nu^{ 2\delta}  + \nu^{-\frac13}\nu^{2(\frac23 + 3\delta)}\rt] \\
			&\leq  C_1\nu^{\frac43 + 4\delta }  + C_2 \nu^{1+4\delta}\left(  \nu^\frac16 + \nu^{\frac1{12} + \delta} + \nu^{\frac{5}{6} \delta} + \nu^{\frac1{4} + \delta}  \right) \leq (C_1 + C_2\nu^{\frac{5}{6} \delta} ) \nu^{1+4\delta} .
		\end{align*}
		For $\theta^{\mathrm{NL}}$, using Proposition \ref{lem:est of thetanl}, Proposition \ref{lem:linear theta} and Corollary \ref{cor2}, we have that
		\begin{align*}
			& \| e^{c \nu^{\frac{1}{3}} \lambda\left(D_x\right) t}\langle D_x\rangle\langle\frac{1}{D_x}\rangle^\epsilon \langle D_x \rangle^\frac13  \theta^{\mathrm{NL}}\|_{L_{[T_0, T]}^\infty L^2}^2+\nu\|e^{c \nu^{\frac{1}{3}} \lambda\left(D_x\right) t}\langle D_x\rangle\langle\frac{1}{D_x}\rangle^\epsilon \nabla \langle D_x \rangle^\frac13  \theta^{\mathrm{NL}}\|_{L_{[T_0, T]}^2 L^2}^2 \\
			& \quad+ \nu^{\frac{1}{3}}\|e^{c \nu^{\frac{1}{3}} \lambda\left(D_x\right) t}\langle D_x\rangle\langle\frac{1}{D_x}\rangle^\epsilon\left|D_x\right|^{\frac{1}{3}}  \langle D_x \rangle ^\frac13 \theta^{\mathrm{NL}}\|_{L_{[T_0, T]}^2 L^2}^2  
			\\
			& \quad+\|e^{c \nu^{\frac{1}{3}} \lambda\left(D_x\right) t}\langle D_x\rangle\langle\frac{1}{D_x}\rangle^\epsilon \sqrt{\Upsilon\left(t, D_x, D_y\right)} \langle D_x \rangle^\frac13  \theta^{\mathrm{NL}}\|_{L_{[T_0, T]}^2 L^2}^2 \\
			&\leq C_1'\nu^{\frac53 + 6\delta} +  C_2' \lt[ \nu^\frac16 \nu^{\frac56 + 3\delta} \nu^{\frac13 +\delta } \nu^{\frac23 + 2\delta}  + \nu^\frac14 \nu^{\frac56 + 3\delta} \nu^{\frac13 +\delta } \nu^{\frac13 + 3\delta} \rt. \\
			&\lt.\quad +\nu^{\frac13 +\delta }(\nu^{\frac1{12}} \nu^{\frac56 + 3\delta} \nu^{\frac23 + 3\delta} +  \nu^\frac16  \nu^{\frac32 (\frac23 + 3\delta)}    \nu^{\frac12 (\frac13 + 3\delta)}  ) \rt. \\
			&\lt. \quad + (\nu^{\frac56 + 3\delta}  \nu^{\frac12 + 2\delta} + \nu^{\frac23 + 3\delta} \nu^{\frac12 + 2\delta} ) \nu^{\frac12 + 2\delta}   + \nu^{-\frac13 - \frac\delta6} \nu^{\frac12 + 2\delta} \nu^{\frac23 + 2\delta}(\nu^{\frac56 + 3\delta} + \nu^\frac16 \nu^{\frac23 + 3\delta}) \rt. \\
			& \lt. \quad + \nu^{\frac32 (\frac23 + 3\delta)}  \nu^{\frac12 + 2\delta} \nu^{\frac12 (\frac13 + 3\delta)}  + \nu^{\frac56 + 3\delta} \nu^{\frac12 + 2\delta} \nu^{\frac13 + 3\delta}  + \nu^{2(\frac23 + 3\delta)}\nu^{\frac12 + 2\delta} \rt] \\
			&\leq  C_1'\nu^{\frac53 + 6\delta} +  C_2'
			\nu^{\frac{5}{3} + 6\delta}  \left(
			\nu^{\frac{1}{3}} 
			+ \nu^{\frac{1}{12} + \delta} 
			+ \nu^{\frac{1}{4} + \delta} 
			+ \nu^\delta
			+ \nu^{\frac{1}{3} + 2\delta} 
			+ \nu^{\frac{1}{6} + 2\delta} 
			+ \nu^{\frac{5}{6}\delta} 
			+ \nu^{2\delta} 
			+ \nu^{\frac{1}{6} + 2\delta}
			\right) \\
			&\leq ( C_1' +   C_2'\nu^{\frac{5}{6} \delta}   )  \nu^{\frac{5}{3} + 6\delta}.
		\end{align*}
		Choose $0<\nu_2<1$ sufficiently small such that for all $0<\nu<\nu_2$, 
		$$
		\nu^{\frac{5}{6} \delta} \max\{C_2, C_2'\} \leq \frac{1}{2} \min\{C_1, C_1'\}.
		$$
		Then, applying a standard bootstrap argument, we deduce that $T=+\infty$ and 
		\begin{equation*} 
			\begin{aligned}
				&\| e^{c \nu^{\frac{1}{3}} \lambda(D_x) t} \langle D_x \rangle \langle \tfrac{1}{D_x} \rangle^\epsilon \omega^{\mathrm{NL}} \|_{L^\infty_{[T_0, +\infty]} L^2}  + \|e^{c \nu^{\frac{1}{3}} \lambda\left(D_x\right) t}\langle D_x\rangle\langle\frac{1}{D_x}\rangle^\epsilon \partial_x \nabla \phi^{\mathrm{NL}}\|_{L_{[T_0, +\infty]}^2 L^2} \\
				&\quad + \nu^{-\frac{1}{3} - \delta} \| e^{c \nu^{\frac{1}{3}} \lambda(D_x) t} \langle D_x \rangle \langle \tfrac{1}{D_x} \rangle^\epsilon \langle D_x \rangle^{\frac{1}{3}} \theta^{\mathrm{NL}} \|_{L^\infty_{[T_0, +\infty]} L^2} 
				\lesssim \nu^{\frac{1}{2} + 2\delta},
			\end{aligned}
		\end{equation*}
		which completes the proof.
	\end{proof}

    \noindent {\bf Acknowledgments.}
	The authors would like to thank Professors Zhifei Zhang, Weiren Zhao and Hui Li for some helpful communications. W. Wang was supported by National Key R\&D Program of China (No. 2023YFA1009200) and NSFC under grant 12471219.\\
	\noindent {\bf Declaration of competing interest.}
	The authors state that there is no conflict of interest.\\
	
	\noindent {\bf Data availability.}
	No data was used in this paper.

\end{document}